\documentclass[11pt]{amsart}
\usepackage[T1]{fontenc}
\usepackage{lmodern}
\usepackage[utf8]{inputenc}
\usepackage{amsmath,amssymb,amsthm,mathtools,mathrsfs}
\usepackage{enumitem,array}
\usepackage{needspace}
\usepackage{hyperref}
\usepackage[margin=1.15in]{geometry}
\hypersetup{
  colorlinks=true,
  citecolor=blue,
  linkcolor=blue,
  urlcolor=blue,
  pdftitle={Jacobi Endpoint Pencils and Sharp Interlacing for Centered Binomial Samples},
  pdfauthor={Seokho Jin},
  pdfsubject={Jacobi endpoint-pencil preservation, sharp interlacing for centered binomial samples, and arithmetic applications},
  pdfkeywords={Jacobi spectral multipliers, endpoint-pencil preservation, real-rooted pencils, weak and strict interlacing, total nonnegativity, centered binomial samples, Dedekind zeta functions}
}

\newtheorem{theorem}{Theorem}[section]
\newtheorem{corollary}[theorem]{Corollary}
\newtheorem{proposition}[theorem]{Proposition}
\newtheorem{lemma}[theorem]{Lemma}

\theoremstyle{definition}
\newtheorem{definition}[theorem]{Definition}
\newtheorem{example}[theorem]{Example}
\theoremstyle{remark}
\newtheorem{remark}[theorem]{Remark}

\newcommand{\C}{\mathbb C}
\newcommand{\R}{\mathbb R}
\newcommand{\T}{\mathbb T}
\newcommand{\diag}{\operatorname{diag}}
\newcommand{\Residue}{\operatorname*{res}}
\newcommand{\Resultant}{\operatorname{Res}}
\newcommand{\Disc}{\operatorname{Disc}}
\newcommand{\Zset}{\mathcal Z}
\newcommand{\JacOp}{\mathscr J}
\newcommand{\JacMult}{\mathscr M}

\title[Jacobi endpoint pencils]{Jacobi Endpoint Pencils and Sharp Interlacing for Centered Binomial Samples}
\author{Seokho Jin}
\address{Department of Mathematics, Chung-Ang University, 84 Heukseok-ro, Dongjak-gu, Seoul 06974, Republic of Korea}
\email{archimed@cau.ac.kr}
\date{13 October 2026}

\newcommand{\Id}{\operatorname{Id}}
\newcommand{\Jk}{J_{\kappa_0,\kappa_1}}
\newcommand{\Nx}{\mathrm N_{[0,1]}}
\begin{document}
\raggedbottom
\begin{abstract}
We study interlacing between real polynomial quotients obtained from centered
binomial samples of a fixed even or odd real entire function at consecutive
odd degrees. For sources of order at most one, we determine the optimal
vertical-strip half-widths guaranteeing weak interlacing for every adjacent
pair: $\sqrt{15/28}$ for even sources and the explicit algebraic constant
$h_*\approx 1.0656$ for odd sources. These
constants are forced by the first nontrivial pair and suffice at every
degree. Under the same sharp bounds, nonpolynomial sources have simple zeros
in $(-2,2)$ and strict interlacing at every positive index. We also prove an
eventual result: if a nonpolynomial source has growth order less than two and
only finitely many zeros off the imaginary axis, then all sufficiently high
adjacent quotients have simple interior zeros and strictly interlace; the
starting index depends only on the finite factor formed from those zeros.
The proof turns the sampling problem into an interval real-rootedness problem
for a Jacobi differential operator. For fixed degree, we also determine the
maximal range of scalar shifts preserving real-rooted polynomials with at
most one zero outside the interval. Applications give strict interlacing for
quotients attached to completed Dedekind zeta derivatives and to successive
centered blocks of critical values of self-dual newforms satisfying explicit
spacing and parity conditions.
\end{abstract}
\maketitle
\section{Introduction and main results}

We study interlacing between the real quotients of centered binomial
samples of a fixed even or odd real entire function. Interlacing of
consecutive degrees is familiar for orthogonal polynomial sequences
\cite[\S18.2(vi)]{NISTHandbook}, but no common orthogonality measure is
assumed for the sampled quotients considered here. The source and spacing
remain fixed, while increasing the sampling degree changes the binomial
weights at every existing sampling point. The problem is to obtain an
order across degrees from the zero geometry of the source, rather than
from separate zero-location statements.

Such sampling constructions arise from completed $L$-functions and
period-polynomial analogues
\cite{CFI,DiamantisRolenCrelle,DiamantisRolenSurvey,ElGuindyRaji,JMOS}.
We determine sharp strip conditions for interlacing at every degree,
and an eventual comparison permitting finitely many arbitrary zeros off
the imaginary axis. Individual unit-circle zero location does not imply
this comparison, as Proposition~\ref{prop:circle-vs-interlacing} shows.

The paper has two principal conclusions. First, the optimal strip works at
every sampling degree and, for nonpolynomial sources, gives simple zeros and
strict interlacing. Second, finitely many zeros away from the imaginary axis
are allowed and the same strict conclusion holds from some degree onward;
the starting degree depends only on their finite factor. Both results reduce
to preservation of real-rootedness by a Jacobi differential operator for
polynomials with at most one zero outside the interval.

\subsection*{Centered samples and sharp interlacing}
For a function $H$, an integer $d\ge0$, and a spacing $\delta>0$, define
\begin{equation}\label{eq:centered-B-notation}
 B_{d,\delta}[H](z)=\sum_{j=0}^d\binom dj
 H\!\left(\delta\left(j-\frac d2\right)\right)z^j,
 \qquad B_d[H]=B_{d,1}[H].
\end{equation}
We call $H$ the \emph{source} and $d$ the \emph{sampling degree}. A polynomial
or entire function is \emph{real} if
\[
 H(\overline u)=\overline{H(u)}.
\]
We use unit spacing: increasing the sampling degree by two adds one
sampling point at each end and changes the binomial weights at the
existing points.

For odd sampling degree $2n+1$, an even source $F$ forces the factor $z+1$,
while an odd source $G$ forces $z-1$. The symmetry of the remaining
coefficients gives unique real polynomials $C^-_{F,n}$ and $C^+_{G,n}$
such that
\begin{align}
 B_{2n+1}[F](z)&=z^n(z+1)C_{F,n}^-(z+z^{-1}),
 \label{eq:intro-even-quotient}\\
 B_{2n+1}[G](z)&=z^n(z-1)C_{G,n}^+(z+z^{-1}).
 \label{eq:intro-odd-quotient}
\end{align}
Here $n\ge0$ is the quotient index. On the unit circle $\T=\{z:|z|=1\}$,
\[
 X=z+z^{-1}=2\cos\theta\qquad(z=e^{\mathrm{i}\theta},\ \mathrm{i}^2=-1),
\]
so the coordinate $X$ turns a unit-circle zero problem into a real-interval
problem on $[-2,2]$.

For polynomials of degrees $n$ and $n+1$, with simple real zeros
\[
 \alpha_1<\cdots<\alpha_n,\qquad\beta_1<\cdots<\beta_{n+1},
\]
strict interlacing means
\[
 \beta_1<\alpha_1<\beta_2<\cdots<\alpha_n<\beta_{n+1}.
\]
To allow common zeros and degree drops, for real polynomials $A,B$, write
$A\preceq B$ if every nonzero polynomial
\[
 B+tA,\qquad t\in\R,
\]
has only real zeros. The family $B+tA$ is a \emph{polynomial pencil}; the
condition is the weak Obreschkoff relation \cite{Obreschkoff1963}.
For coprime polynomials of adjacent exact degrees it is equivalent to
interlacing. We include the zero polynomial with the conventions stated in
Section~\ref{sec:sampling-framework}.

Let
\[
 S_h=\{u\in\C:|\Re u|\le h\},
\]
and write $\Zset(H)$ for the zero set of $H$, without multiplicity.
The growth assumption of order at most one allows the passage from
polynomial sources to entire sources by the approximations in
Section~\ref{sec:sampling-framework}. The odd threshold below is
\[
 h_*:=\sqrt{\tfrac14-\alpha_*}\approx1.0656,
\]
where $\alpha_*$ is the unique root in $(-1,-3/4)$ of
\[
 7\alpha^4+56\alpha^3+142\alpha^2+168\alpha+72=0.
\]

\begin{theorem}[Weak interlacing for centered binomial samples]
\label{thm:main-sturm}
Let $F$ and $G$ be nonzero even and odd real entire functions, respectively,
of order at most one. For every $n\ge0$,
\[
 \begin{aligned}
 \Zset(F)\subseteq S_{\sqrt{15/28}}
 &\quad\Longrightarrow\quad C_{F,n}^-\preceq C_{F,n+1}^-,\\
 \Zset(G)\subseteq S_{h_*}
 &\quad\Longrightarrow\quad C_{G,n}^+\preceq C_{G,n+1}^+.
 \end{aligned}
\]
Both strip bounds are optimal uniformly in $n$. For every $\varepsilon>0$,
there are an even real polynomial with zeros in
$S_{\sqrt{15/28}+\varepsilon}$ and an odd real polynomial with zeros in
$S_{h_*+\varepsilon}$ for which the corresponding interlacing relation
fails at $n=1$, that is, for sampling degrees $3$ and $5$.
\end{theorem}

Both sharp constants are forced by the first nontrivial comparison,
between sampling degrees $3$ and $5$. Their sufficiency is an all-degree
assertion: no higher degree requires a narrower common strip. Thus the
lowest-degree obstruction determines the uniform range for the entire
sampling sequence.

The distinction from individual zero location is already visible for
$F_a(u)=u^2-a^2$: if $\sqrt{15/28}<a\le\sqrt3/2$, then $B_3[F_a]$ and
$B_5[F_a]$ have all their zeros on the unit circle, but their real quotients
do not weakly interlace (Proposition~\ref{prop:circle-vs-interlacing}).

The sharp widths in Theorem~\ref{thm:main-sturm} also give simple quotient
zeros in $(-2,2)$ and strict interlacing for every $n\ge1$ when the source
is nonpolynomial
(Corollary~\ref{cor:sharp-source-strictness}). Theorem~\ref{thm:eventual-finite-exception}
gives a complementary conclusion without a fixed strip bound: for a
nonpolynomial even source of growth order less than two with only finitely
many zeros off the imaginary axis, strict interlacing holds at all
sufficiently large indices.

\subsection*{Operator theorem underlying the uniform result}
After translating $[-2,2]$ to $[0,4]$ by $y=X+2$, the two adjacent
quotients, up to positive scalar factors, are images of consecutive
monomials under a common linear map.
For any linear operator $T$,
\[
 T\bigl(y^n(y+t)\bigr)=T(y^{n+1})+t\,T(y^n),
\]
so weak interlacing reduces to real-rootedness of the images of the pencil
$\{y^n(y+t):t\in\R\}$. We call this the \emph{endpoint pencil}: $n$ input
zeros are fixed at $0$, and the remaining zero is free, possibly outside
the interval. Thus allowing one exterior zero comes from the geometry of
the comparison itself.

For centered binomial samples, the required linear maps are diagonal in
Jacobi polynomial bases. For $0<\nu<2$, let
\[
 \JacOp_\nu=y(y-4)D_y^2+(2y-4\nu)D_y,\qquad D_y=\frac{d}{dy}.
\]
This operator has a monic eigenpolynomial $\mathsf P_r^{(\nu)}$ of each
degree $r\ge0$, with eigenvalue $r(r+1)$. For an even function $F$, define
its Jacobi spectral multiplier by
\[
 \JacMult_{F,\nu}\mathsf P_r^{(\nu)}
 =F(r+1/2)\mathsf P_r^{(\nu)}.
\]

Section~\ref{sec:sampling-framework} proves
\[
 \begin{aligned}
 C^-_{F,n}(y-2)&=\JacMult_{F,3/2}(y^n),\\
 C^+_{G,n}(y-2)&=\left(n+\frac12\right)\JacMult_{G/u,1/2}(y^n).
 \end{aligned}
\]
For an odd entire $G$, the quotient $G/u$ is extended at the origin by
$G'(0)$. It is even and has no additional zeros outside the strip
containing the zeros of $G$. The positive factors $n+\tfrac12$ do not
affect the zeros or the weak comparison. The Jacobi polynomials provide
an eigenbasis; the interlacing problem concerns the images of consecutive
monomials, not this eigenbasis itself.

Source factorization is compatible with this representation because
\[
 \left(r+\frac12\right)^2-\frac14=r(r+1).
\]
A polynomial in $u^2-\tfrac14$ therefore becomes the same polynomial in
$\JacOp_\nu$, so factoring an even polynomial source expresses its
multiplier as a composition of simpler operators. After the first factor,
however, the image need not have the form $y^n(y+t)$. To compose the
factors, we therefore control the exterior-root count on a larger class
containing these intermediate images.

For $N\ge0$, write $\R[y]_{\le N}$ for the real polynomials of degree at
most $N$, including zero, and let $\mathcal E_N$ consist of the zero
polynomial and the nonzero polynomials in $\R[y]_{\le N}$ with at most
one zero outside $[0,4]$, counting multiplicities. Every nonzero member is
real-rooted, since a nonreal zero would occur together with its conjugate.

Preservation of the endpoint pencil means that every image belongs to
$\mathcal E_{n+1}$; no endpoint multiplicity is required of the image.

\begin{theorem}[Endpoint-pencil thresholds for Jacobi spectral multipliers]
\label{thm:jacobi-family-endpoint-pencil}
Let $0<\nu<2$ and let $F\not\equiv0$ be real and even, either a polynomial
or an entire function of order at most one. For $n\ge1$ put
\begin{equation}\label{eq:general-strip-parameters}
 \mathfrak e_{n,\nu}=\frac{n(n+1)(2-\nu)}{2n+\nu},\qquad
 \mathfrak h_{n,\nu}=\min\{1,\sqrt{\tfrac14+\mathfrak e_{n,\nu}}\}.
\end{equation}
\begin{enumerate}[label=\textup{(\roman*)}]
\item If $n\ge1$ and $\Zset(F)\subseteq S_{\mathfrak h_{n,\nu}}$, then
\[
 \JacMult_{F,\nu}[y^n(y+t)]\in\mathcal E_{n+1}\quad(t\in\R),
 \qquad \JacMult_{F,\nu}(y^n)\preceq\JacMult_{F,\nu}(y^{n+1}).
\]
This width is sharp when $\mathfrak e_{n,\nu}\le3/4$, including equality
(where $\mathfrak h_{n,\nu}=1$).
\item If $\Zset(F)\subseteq S_{\mathfrak h_{1,\nu}}$, the weak comparison
holds for every $n\ge0$, with uniform sharpness when $\nu\ge10/11$.
\end{enumerate}
\end{theorem}

At $\nu=3/2$, this gives the even optimum in Theorem~\ref{thm:main-sturm}.
At $\nu=1/2$, it gives the odd width $1$. Reaching the larger odd threshold
$h_*$ requires an additional preservation argument.
Section~\ref{subsec:odd-endpoint} also explains how complete consecutive
pencils supply the interval support of the individual quotient zeros.

\subsection*{Completion of the sharp theorem: simple zeros and strict interlacing}
The weak bounds become strict without any additional assumption on the
quotient zeros. Nonpolynomiality supplies source deformations that rule
out multiple zeros, endpoint zeros and common zeros of adjacent images.

\begin{theorem}[Source-side strict interlacing for Jacobi spectral images]
\label{thm:remote-orbit-strictification}
Fix $0<\nu<2$ and $n\ge1$. Let $E$ be a nonpolynomial even real entire
function of order at most one such that
\[
 \Zset(E)\subseteq S_{\mathfrak h_{n,\nu}}.
\]
Then $\JacMult_{E,\nu}(y^n)$ and $\JacMult_{E,\nu}(y^{n+1})$ have exact
degrees $n$ and $n+1$, respectively, have only simple zeros in $(0,4)$,
and strictly interlace. When $\nu=1/2$, the same conclusions hold under
the larger strip condition $\Zset(E)\subseteq S_{h_*}$.
\end{theorem}

For nonpolynomial even and odd sources, the original quotients therefore
have simple zeros in $(-2,2)$ and strictly interlace for $n\ge1$ at the widths
$\sqrt{15/28}$ and $h_*$, respectively. Both widths remain optimal in this
class (Corollary~\ref{cor:sharp-source-strictness}); strictness requires
neither a narrower strip nor a genericity assumption. In contrast, for
$E(u)=p(u^2-1/4)$ with $\deg p=K$, Proposition~\ref{prop:polynomial-endpoint-multiplicity}
gives an endpoint zero of exact multiplicity $d-K$ in every image
$\JacMult_{E,\nu}(y^d)$ with $d>K$. Within these uniform strip classes,
nonpolynomiality thus distinguishes
strict interlacing with simple interior zeros at every positive index
from the forced endpoint degeneracy of polynomial sources.

To prove strictness, we move one source zero together with its negative and
conjugates in both directions without leaving the class covered by the weak
theorem. On each fixed degree space, linear combinations of the induced
first-order variations reproduce the Jacobi operator, whose action at a zero
rules out multiple zeros, endpoint zeros and common zeros of adjacent images;
see Section~\ref{sec:strict-consecutive-degrees}.

The eventual theorem permits finitely many zeros outside the sharp uniform
strips. In the factorization below, $p(u^2-\tfrac14)$ contains the finite
exceptional part, including all zeros off the imaginary axis; $H$ has only
imaginary-axis zeros; and $e^{au^2}$, $a\ge0$, is an allowed Gaussian factor.
The starting index depends on $p$ and $\nu$, but not on $a$ or $H$.

\begin{theorem}[Eventual strict interlacing with a finite exceptional factor]
\label{thm:eventual-finite-exception}
Fix $0<\nu<2$, a nonzero polynomial $p\in\R[\lambda]$ and $a\ge0$. Let
\[
 H(u)=c\,u^{2r}\prod_{j\in I}\left(1+\frac{u^2}{b_j^2}\right),
 \qquad c\in\R\setminus\{0\},\quad r\in\mathbb Z_{\ge0},\quad b_j>0,
\]
where $I$ is finite or countably infinite and
$\sum_{j\in I}b_j^{-2}<\infty$ in the infinite case. Assume that $a>0$
or $I$ is infinite, and put
\[
 E(u)=e^{au^2}p(u^2-1/4)H(u).
\]
There exists $n_0=n_0(p,\nu)$, independent of $a$ and $H$, such that for every
$n\ge n_0$ the polynomials $\JacMult_{E,\nu}(y^n)$ and
$\JacMult_{E,\nu}(y^{n+1})$ have exact degrees $n,n+1$, have only simple
zeros in $(0,4)$, and strictly interlace.

In particular, the same eventual conclusion holds for every nonpolynomial
even real entire function of growth order less than two having only
finitely many zeros off the imaginary axis.
\end{theorem}

For example, $E(u)=(u^4+16)\cosh u$ has four zeros off the imaginary
axis, each with $|\Re u|=\sqrt2>h_*$. Thus it lies in neither sharp strip
of Theorem~\ref{thm:main-sturm}. Nevertheless,
Theorem~\ref{thm:eventual-finite-exception} applies: for every fixed
$0<\nu<2$, all sufficiently high adjacent Jacobi images have simple
interior zeros and strictly interlace. Section~\ref{subsec:eventual-finite-exception}
also gives an effective sufficient cutoff and records the sign distinction
for Gaussian factors.

When $a=0$,
\[
 E(j+1/2)=p(j(j+1))H(j+1/2)\qquad(j\ge0),
\]
so a nonconstant $p$ changes the multiplier at every Jacobi degree, not only
finitely many sampled values. The finite polynomial $p$ is treated by a
Jacobi--Laguerre limit and an exact adjacent-degree identity. The remaining
factors are positive scalar shifts $\JacOp_\nu+s\Id$, $s>0$, or limits of
their products, so they do not increase the number of zeros outside the
interval. Two-sided variation then excludes multiple, endpoint and common
adjacent zeros. Since the starting index is fixed before these factors are
attached, it depends only on $p$ and $\nu$, uniformly over the allowed $a$
and $H$. The sign restriction on the Gaussian factor is essential:
Section~\ref{subsec:eventual-finite-exception} gives a zero-free negative
Gaussian whose positive-index images have no zeros in the interval.

Eventual real-rootedness and interlacing for finite combinations of
orthogonal polynomials have recently been studied by Dur\'an
\cite{DuranOrthogonal,DuranLaguerre}. In particular, the Laguerre results
permit nonreal zeros of the coefficient polynomial while still giving
eventual simple positive zeros and adjacent-degree interlacing for certain
normalizations. A related simplicity mechanism for infinite-order
differential operators with infinitely many zeros appears in
\cite{CardonDeGaston}. These results do not directly yield
Theorem~\ref{thm:eventual-finite-exception}: here the finite core is a
polynomial in a Jacobi spectral operator, an arbitrary infinite
imaginary-axis factor is then attached with a cutoff uniform in that
factor, and the conclusion orders adjacent sampling degrees. In the
order-at-most-one strip setting, \cite{JinSharpSampling} already supplies
unit-circle zero location for each individual sample once the sampling
degree dominates the strip width; the additional conclusion here is the
strict ordering of adjacent degrees for arbitrary $0<\nu<2$ in the factorized
setting of growth order less than two.

\subsection*{Fixed-degree preservation and boundary contacts}
The associated fixed-degree preservation problem allows the constrained
input zeros to vary throughout the interval. For arbitrary positive
Jacobi endpoint parameters, we reduce the boundary of this continuous
extremal problem to finitely many endpoint multiplicity patterns.

On the normalized interval $[0,1]$, consider
\[
 J_{\kappa_0,\kappa_1}=x(x-1)D_x^2+
 \bigl((\kappa_0+\kappa_1)x-\kappa_0\bigr)D_x,
 \qquad D_x=\frac{d}{dx},
\]
with arbitrary $\kappa_0,\kappa_1>0$. The operator in
Theorem~\ref{thm:jacobi-family-endpoint-pencil} corresponds, after $y=4x$,
to $(\kappa_0,\kappa_1)=(\nu,2-\nu)$.
More generally, Theorem~\ref{thm:general-positive-factor-eventual}
separates the two roles in the eventual argument: the finite polynomial
factor determines the cutoff, while the attached positive spectral factors
preserve the endpoint pencils and, in the nonpolynomial regime, remove
degeneracies. This principle holds for arbitrary positive
$\kappa_0,\kappa_1$, not only for the sampling slice
$(\nu,2-\nu)$.

For each $m\ge1$, Theorem~\ref{thm:positive-charge-lower-ray} determines
a positive threshold $\gamma_m$ such that $J_{\kappa_0,\kappa_1}-\eta\Id$
preserves real-rooted polynomials of exact degree $m+1$, positive leading
coefficient and at most one zero outside $[0,1]$ for every
$\eta\le\gamma_m$. Here $\Id$ denotes the identity operator; scalar
multiples of it are abbreviated in expressions such as $\JacOp_\nu+\alpha$.
Explicit monic inputs develop a nonreal conjugate pair at
$\eta=\gamma_m+\varepsilon$ for every sufficiently small $\varepsilon>0$,
showing that $(-\infty,\gamma_m]$ is the maximal preserving lower ray.

The formula for $\gamma_m$ minimizes over the $m+1$ ways of distributing
the constrained roots between the two endpoints. These are not an assumed
restriction on the input class: the proof reduces the first possible
obstruction among all admissible root configurations to such an
endpoint-supported input with one free zero. The finite formula is thus
a consequence of the extremal reduction. Controlling the exterior count
also distinguishes the result from real-rootedness preservation and from
interval preservation for inputs with no exterior zero.

General zero-region preservers are studied in
\cite{BorceaBrandenCircular,BrandenChasseStrip}, and diagonal Jacobi and
Legendre operators in
\cite{BatesYoshida,BlakemanDavisForgacsUrabe,ChasseForgacsPiotrowski}.
The comparison following Theorem~\ref{thm:positive-charge-lower-ray} also
distinguishes our condition from the finite-degree interval-preservation
theorem in \cite{LeakeRepresentation}.

Preservation under composition links the operator theory to the sampling
results. A matrix is \emph{totally nonnegative} (TN) if all its square minors
are nonnegative; in suitable polynomial bases, TN controls the number of
zeros outside the interval. A nonreal zero $a+ib$ of an even real source
belongs to the quartet $\{\pm(a+ib),\pm(a-ib)\}$; Section~\ref{sec:zero-orbit-factors}
calls the corresponding factor a quartet factor. TN controls such factors
only for $|a|\le1$, whereas the sharp odd strip has $h_*>1$.
Sections~\ref{sec:odd-enlarged-blocks}--\ref{sec:endpoint-sturm} therefore
analyze exact-degree paths and exclude every first failure: an interior
multiple zero, an endpoint zero, or escape to infinity. Real-pair factors
are grouped in twos; if one remains, it acts first on the endpoint pencil.

\subsection*{Arithmetic consequences and related comparisons}
For every number field and every nonnegative derivative order, the
sampling quotients of the centered entire completion of the Dedekind zeta
function have simple zeros in $(-2,2)$, and the quotients of degrees $n$
and $n+1$ strictly interlace for every $n\ge1$
(Theorem~\ref{thm:main-dedekind-array}).

The fixed-degree results in
\cite[Theorems~1.1(2), 1.2]{JinSharpSampling} give unit-circle zero location
and, under their strict-strip and nonvanishing assumptions, simplicity
and interlacing across derivative orders. We use this input only for the
derivative-order comparison. Simplicity and the sampling-degree comparison
at fixed derivative order follow directly from the source-side strictness
theorem proved here. Appendix~\ref{app:newform-nondegeneracy} verifies the
required finite-difference nonvanishing conditions.

For self-dual newforms, we compare two successive centered blocks of
critical values. The spacing and parity conditions in
Section~\ref{sec:arithmetic-applications} keep every sample in the critical
range, and the corresponding quotients strictly interlace at each fixed
derivative order (Corollary~\ref{cor:main-newform-blocks}). Interlacing under
changes of modular weight or level is studied in
\cite{BrelandEtAl,KoMackenzieXue}. Further signed and common-root consequences
of varying the sampling degree at fixed derivative order are recorded at the
end of Section~\ref{sec:arithmetic-applications}.

Section~\ref{sec:sampling-framework} establishes the correspondence between
samples and Jacobi endpoint pencils.
Sections~\ref{sec:real-spectral-factors}--\ref{sec:zero-orbit-factors} develop
the preservation results, Sections~\ref{sec:exact-first-pencil}--\ref{sec:endpoint-sturm}
prove the weak interlacing theorems and their sharpness, and
Sections~\ref{sec:strict-consecutive-degrees}--\ref{sec:arithmetic-applications}
prove strict interlacing for nonpolynomial sources, establish the eventual
result and give the arithmetic applications.

\section{From centered samples to Jacobi endpoint pencils}
\label{sec:sampling-framework}

The coordinate $X=z+z^{-1}$ and translation $y=X+2$ turn the quotient
problem into an interval problem on $[0,4]$; the removed endpoint selects
the Jacobi basis. We derive the dictionary using classical Chebyshev and
Jacobi identities \cite[Chapter~IV]{SzegoOrthogonal} and \cite[Chapter~18]{NISTHandbook}.

\subsection{The quotient model}

Let \(U_r\) be the Chebyshev polynomial of the second kind, characterized for
\(0<\theta<\pi\) by
\[
 U_r(\cos\theta)=\frac{\sin((r+1)\theta)}{\sin\theta}
 \qquad(r\ge0),
\]
with the endpoint values understood by continuity, and set \(U_{-1}=0\).  Put
\[
 H_r^-(X)=U_r(X/2)-U_{r-1}(X/2),\qquad
 H_r^+(X)=U_r(X/2)+U_{r-1}(X/2).
\]
We introduce the translated endpoint bases
\[
 \widehat H_r^{\pm}(y)=H_r^\pm(y-2),\qquad y=X+2.
\]
Let \(F\) be even and \(G\) odd, both real-valued on $\mathbb R$.
Using the parity of $F$ and $G$ and the symmetry of the binomial
coefficients, pair the terms indexed by $j=n-r$ and $j=n+1+r$ in the
centered samples.  The Laurent-polynomial identities
\[
 \frac{z^{-r}(1+z^{2r+1})}{z+1}=H_r^-(z+z^{-1}),
 \qquad
 \frac{z^{-r}(z^{2r+1}-1)}{z-1}=H_r^+(z+z^{-1}),
\]
then give, upon comparison with
\eqref{eq:intro-even-quotient}--\eqref{eq:intro-odd-quotient}, the quotient
formulas
\begin{align*}
 C_{F,n}^{-}(X)
 &=\sum_{r=0}^n \binom{2n+1}{n-r}F(r+1/2)H_r^-(X),\\
 C_{G,n}^{+}(X)
 &=\sum_{r=0}^{n}\binom{2n+1}{n-r}G(r+1/2)H_r^+(X).
\end{align*}
Taking $F\equiv1$ in the even formula and using
\[
 (1+z)^{2n+1}=z^n(z+1)(z+z^{-1}+2)^n
\]
gives
\[
 y^n=\sum_{r=0}^n \binom{2n+1}{n-r}\widehat H_r^{-}(y).
\]
The odd expansion follows by the raising identity in
Lemma~\ref{lem:endpoint-raising}.

\subsection{The Jacobi operator and endpoint pencils}

Retain $\JacOp_\nu=y(y-4)D_y^2+(2y-4\nu)D_y$ from the Introduction,
where $D_y=d/dy$ and $0<\nu<2$.
For each \(r\ge0\), let \(\mathsf P_r^{(\nu)}\) be the unique monic
degree-\(r\) eigenpolynomial
\[
 \JacOp_{\nu}\mathsf P_r^{(\nu)}
 =r(r+1)\mathsf P_r^{(\nu)}.
\]
Up to normalization it is
\[
 P_r^{(1-\nu,\nu-1)}(y/2-1),
\]
where \(P_r^{(\alpha,\beta)}\) denotes the standard Jacobi polynomial
\cite[Chapter~IV]{SzegoOrthogonal}, and the spectrum \(r(r+1)\) is
independent of \(\nu\).  The two translated endpoint quotient bases are
\[
 \mathsf P_r^{(3/2)}=\widehat H_r^{-},
 \qquad
 \mathsf P_r^{(1/2)}=\widehat H_r^{+}.
\]
The resulting endpoint dictionary is
\[
\begin{array}{c|c|c}
\text{parity of the function} & \text{removed endpoint} & \text{Jacobi parameter}\\
\hline
F\text{ even} & z=-1 & \nu=3/2\\
G\text{ odd} & z=1 & \nu=1/2
\end{array}
\]
For the odd row, put \(\widetilde G(u):=G(u)/u\) for \(u\ne0\). If \(G\) is
entire, then \(\widetilde G\) extends to an even entire function by
\(\widetilde G(0):=G'(0)\); this is the even spectral source.

\begin{remark}[Why the paper treats odd sampling degrees]
The one-endpoint reduction is specific to odd sampling degrees.  At positive
even degree, even sources have no forced endpoint and, after reciprocal
normalization, expand in the first-kind Chebyshev basis.  Odd sources force
both $z=1$ and $z=-1$ and, after removing $z^2-1$, expand in the second-kind
basis. Thus even sampling degrees have a different endpoint geometry.
\end{remark}

The basis $\{\mathsf P_r^{(\nu)}\}$ defines the multiplier by
$\JacMult_{F,\nu}\mathsf P_r^{(\nu)}=F(r+1/2)\mathsf P_r^{(\nu)}$.
For a polynomial $h$, spectral calculus similarly gives
$h(\JacOp_\nu)\mathsf P_r^{(\nu)}=h(r(r+1))\mathsf P_r^{(\nu)}$.
On each fixed degree space, $\JacMult_{F,\nu}$ depends on finitely many
half-integer values of $F$. Thus locally uniform convergence of sources
implies coefficientwise convergence of every fixed polynomial image.
The preceding expansion proves
$\JacMult_{F,3/2}(y^n)=C_{F,n}^-(y-2)$. For the odd source $G$, the other
endpoint uses the following identity.

\begin{lemma}[Endpoint-raising identity]
\label{lem:endpoint-raising}
For every \(r\ge0\),
\[
 \left(yD_y+\frac12\right)\widehat H_r^{-}
 =\left(r+\frac12\right)\widehat H_r^{+}.
\]
Consequently, with \(\widetilde G=G/u\) as above, for every \(n\ge0\),
\begin{equation}\label{eq:odd-multiplier-representation}
 C_{G,n}^{+}(y-2)
 =\left(n+\frac12\right)\JacMult_{\widetilde G,1/2}(y^n).
\end{equation}
\end{lemma}

\begin{proof}
Put $\mathcal D_z=z\,d/dz$. From the Laurent representations of
$\widehat H_r^\pm$,
\[
 (\mathcal D_z-r-\tfrac12)(1+z^{2r+1})
 = (r+\tfrac12)(z^{2r+1}-1),
 \qquad \mathcal D_z y=z-z^{-1}.
\]
After division by $z^r(z-1)$, this is the displayed raising identity.
Applying $yD_y+\tfrac12$ termwise to the expansion of $y^n$ gives
\[
 \sum_{r=0}^n\binom{2n+1}{n-r}(r+\tfrac12)\widehat H_r^+(y)
 = (n+\tfrac12)y^n.
\]
Now apply $\JacMult_{\widetilde G,1/2}$ and use
$G(r+1/2)=(r+1/2)\widetilde G(r+1/2)$.
\end{proof}

For a nonzero real polynomial \(P\) and a real interval \(I\), let
\(\mathrm N_I(P)\) denote the number of zeros of \(P\) in \(I\), counted with
multiplicity.  For integers \(N,q\ge0\), define the interval-defect class
\[
 \mathcal E_{N,q}
 =\{0\}\cup
 \left\{P\in\R[y]_{\le N}\setminus\{0\}:
 \mathrm N_{[0,4]}(P)\ge\deg P-q
 \right\}.
\]
Thus a nonzero member of $\mathcal E_{N,q}$ has at most $q$ zeros outside
$[0,4]$, counted with multiplicity, and \(\mathcal E_{N,0}\) is the class of
polynomials rooted in \([0,4]\).
This extends the notation from the Introduction: \(\mathcal E_{N,1}=\mathcal E_N\),
the one-defect class used for the pencils generated by $y^n$ and
$y^{n+1}$.  Every nonzero member of \(\mathcal E_N\) is real-rooted: a nonreal
zero would bring its conjugate and hence create at least two defects.  The
classes are invariant under real scaling.  Closedness and the fixed-degree generic approximation used below are
proved in Section~\ref{subsec:defect-closure}.

The input $y^n(y+t)$ belongs to $\mathcal E_{n+1}$ for every real $t$.

\begin{definition}[Endpoint-pencil preservation]
\label{def:endpoint-pencil-preservation}
Let $0<\nu<2$, let $F$ be real-valued on the positive half-integers, and let
$n\ge0$.  The family $\{y^n(y+t):t\in\R\}$ is the \emph{endpoint pencil
indexed by $n$}.  We say that $\JacMult_{F,\nu}$ preserves this pencil if
\[
 \JacMult_{F,\nu}\bigl(y^n(y+t)\bigr)\in\mathcal E_{n+1}
 \qquad(t\in\R).
\]
\end{definition}

By linearity, its image is
$\JacMult_{F,\nu}(y^{n+1})+t\JacMult_{F,\nu}(y^n)$.
Endpoint-pencil preservation therefore implies
$\JacMult_{F,\nu}(y^n)\preceq\JacMult_{F,\nu}(y^{n+1})$.

\subsection{Degree, limits, and interlacing conventions}

We retain the notation $B_{d,\delta}[F]$ from
\eqref{eq:centered-B-notation}.  For nonzero real polynomials $A$ and $B$,
let $D$ be their greatest common divisor.  Then $A\preceq B$ if and only if
$D$ is real-rooted and the zeros of $A/D$ and $B/D$ weakly interlace, counted
with multiplicity.  In particular, the degrees of the two reduced polynomials
differ by at most one.

The sampling degree is the index $d$ of $B_{d,\delta}$; polynomial degree
always means actual degree after any cancellation. Limits and pencil members
with vanishing leading coefficient are interpreted in their lower degree,
with $0\preceq0$, and $0\preceq B$ and $B\preceq0$ for real-rooted $B$.
In sampling and arithmetic
sections, $m$ is derivative order and $n$ is quotient index; in the real-factor
analysis, $m=N-1$ counts constrained roots. All root counts include
multiplicity. For a nonzero polynomial $P$, $\operatorname{ord}_\xi P$
denotes its order of vanishing at $\xi$; we set
$\operatorname{ord}_\xi0=+\infty$.

Interlacing finite sets are viewed as increasing labelled lists with alternating
entries and cardinalities differing by at most one. Weak interlacing permits
common entries; strict interlacing does not. This convention also applies
after endpoints are adjoined.

The weak Obreschkoff relation is closed under coefficientwise limits in
uniformly bounded degree, with the zero convention above: for each
$t\in\R$, pass to the limit in $B_j+tA_j$ in the closed class of
zero-or-real-rooted polynomials.

\subsection{Canonical truncations and strip-preserving differentiation}

We use symmetric canonical truncations to pass the polynomial zero-orbit
arguments to entire functions.  The same lemma preserves the zero strip under
differentiation.

\begin{lemma}[Symmetric canonical truncations and strip-preserving derivatives]
\label{lem:derivative-strip}
Let \(h\ge0\), and let \(F\not\equiv0\) be a real entire function of order at
most one.  Assume that \(F(-u)=\epsilon F(u)\) for some
\(\epsilon\in\{\pm1\}\), and suppose that \(\Zset(F)\subseteq S_h\).  Then
there are real polynomials \(F_L\) of
the same parity as \(F\), all of whose zeros lie in \(S_h\), such that
\(F_L\to F\) locally uniformly.  Consequently, for every \(m\ge0\), the
derivative \(F^{(m)}\) is either identically zero or has all zeros in
\(S_h\).
\end{lemma}

\begin{proof}
Pairing the genus-one Hadamard factors under $\rho\mapsto-\rho$
\cite[Chapter~I]{LevinEntire} gives
\[
 F(u)=C u^\kappa\prod_{\{\rho,-\rho\}}(1-u^2/\rho^2),
 \qquad C\in\R\setminus\{0\},\quad (-1)^\kappa=\epsilon,
\]
with $\sum_{\rho\ne0}|\rho|^{-2}<\infty$. The paired exponential factors
cancel, and parity removes the remaining linear exponential.
Conjugation-stable finite subproducts therefore have the same parity,
have all zeros in $S_h$, and converge locally uniformly to $F$.
Their derivatives converge locally uniformly as well; Gauss--Lucas keeps
their zeros in the convex strip, and Hurwitz gives the final assertion
whenever $F^{(m)}\not\equiv0$.
\end{proof}

\subsection{Normalized coordinates and exact-degree classes}
\label{subsec:normalized-exact-degree}
The quotient coordinate is $X=z+z^{-1}$, the original interval coordinate is
$y=X+2$, and the normalized interval coordinate used below is $x=y/4$.
For positive charges $\kappa_0,\kappa_1$ set
\[
 \Jk=x(x-1)\partial_x^2+
 ((\kappa_0+\kappa_1)x-\kappa_0)\partial_x,
 \qquad \sigma=\kappa_0+\kappa_1.
\]
Its degree-$j$ eigenvalue is $\lambda_j^{(\sigma)}=j(j-1+\sigma)$.
On the fixed-spectrum slice $(\kappa_0,\kappa_1)=(\nu,2-\nu)$,
the substitution $p(y)\mapsto p(4x)$ intertwines $\JacOp_\nu$ with
$J_{\nu,2-\nu}$. Hence
\[
 \mathcal L:=J_{1/2,3/2}+\tfrac14\Id,
 \qquad \mathcal Lx^j=(j+\tfrac12)^2x^j-j(j-\tfrac12)x^{j-1}.
\]
The source convention is $F(u)=H(u^2)$, so that the odd multiplier becomes
$H(\mathcal L)$ after pullback. For every $\nu$, the original multiplier
still samples $F(j+1/2)$, not $F(j+\nu)$.

For $N\ge1$, write $\mathfrak E_N$ for the positive-leading, exact-degree-$N$
real polynomials having all roots real and at least $N-1$ roots in $[0,1]$.
This class does not include the zero polynomial. In contrast,
$\mathcal E_{N,q}$ on $[0,4]$ is degree-bounded and includes zero.
Let $\mathbb R[X,Y]_N$ be the space of homogeneous binary forms of degree
$N$, write $\mathbb P$ for projectivization, and let $\mathbb{RP}^1$ be the
real projective line. The projective one-defect closure of $\mathfrak E_N$ is
\[
 \widehat{\mathfrak E}_N=\operatorname{im}\left(
 [0,1]^{N-1}\times\mathbb{RP}^1\longrightarrow
 \mathbb P(\mathbb R[X,Y]_N)\right),
\]
\[
 (\mathbf r,[a:b])\longmapsto
 \left[(aX+bY)\prod_{i=1}^{N-1}(X-r_iY)\right].
\]
The variables $X,Y$ in these two displays are homogeneous coordinates,
not the quotient variable. The stratum $a=0$ represents one
simple projective root at infinity.
For a function $v$, $M_v$ denotes multiplication by $v$, and
$[T,M_v]=T\circ M_v-M_v\circ T$. We also write $[T,x]=[T,M_x]$.

\begin{remark}[Sign-normalized canonical truncations]
\label{rem:positive-truncations}
For a nonzero real even entire source of order at most one with zeros
in $S_h$, normalize its constant nonzero sign on $(h,\infty)$ to be positive. Include its finitely many real zeros and its origin factor
in the initial truncation of Lemma~\ref{lem:derivative-strip}.
Subsequent imaginary pairs and quartets contribute respectively
$1+u^2/b^2>0$ and $|1-u^2/\rho^2|^2>0$ on $\R$.
Thus every truncation is positive on $(h,\infty)$, with positive leading
coefficient. Restore the nonzero real scalar after the exact-degree argument.
\end{remark}
\section{Real spectral factors and endpoint extremizers}
\label{sec:real-spectral-factors}

We distinguish preservation of the whole exact-degree class $\mathfrak E_N$
from preservation of the original endpoint pencil. Throughout this section
$m=N-1\ge1$ is the number of constrained interval roots, not a derivative
order. The charges $\kappa_0,\kappa_1$ are positive, and
$\sigma=\kappa_0+\kappa_1$.

\subsection{Positive shifts and interval-root counts}

\begin{lemma}[Positive-shift interval-root principle]
\label{lem:positive-charge-shift}
For $a>0$ and a nonzero real polynomial $f$,
\[
 \Nx((\Jk+a\Id)f)\ge\Nx(f).
\]
The analogous interval-defect classes are preserved also at $a=0$ when
zero outputs are included.
\end{lemma}
\begin{proof}
First suppose the $s$ roots of $f$ in $[0,1]$ are simple and interior.
In each of their $s+1$ complementary sign intervals, select a positive
maximum or a negative minimum of $f$. At an interior selected point,
$f'=0$ and $ff''\le0$; hence $(\Jk+a\Id)f$ has the strict sign of $f$.
At zero a positive maximum has $f'(0)\le0$, so
$-\kappa_0f'(0)+af(0)>0$. At one it has $f'(1)\ge0$, so
$\kappa_1f'(1)+af(1)>0$. Negative minima follow by changing sign.
The alternating selected signs give at least $s$ roots of the output in
$[0,1]$. If $s=0$, the desired root bound is vacuous.

Split multiple and endpoint interval roots into distinct interior roots,
leaving all other factors unchanged, and pass to a fixed-degree coefficient
limit. The degree-$M$ multiplier
$M(M-1+\sigma)+a$ is positive, so the output retains degree and is nonzero.
Continuity of its root multiset and compactness of $[0,1]$ preserve the count
with multiplicity. Finally let $a\downarrow0$. A nonconstant input retains
its degree, whereas a constant maps to zero. Thus the corresponding
upper bound on the number of roots outside the interval also passes to the
limit.
\end{proof}

With $(\kappa_0,\kappa_1)=(\nu,2-\nu)$ and $y=4x$, the preceding
lemma gives, for $0<\nu<2$, $\mathfrak a\ge0$ and all $N,q\ge0$,
\[
 (\JacOp_\nu+\mathfrak a)(\mathcal E_{N,q})\subseteq\mathcal E_{N,q}.
\]
For $\mathfrak a>0$, it also gives
$\mathrm N_{[0,4]}((\JacOp_\nu+\mathfrak a)P)\ge
\mathrm N_{[0,4]}(P)$ for every nonzero real polynomial $P$.

\subsection{The exact lower ray for positive charges}

The candidate boundary contacts come from endpoint-supported factors
$x^k(1-x)^{m-k}$ with $0\le k\le m$. We record their contact locations and scalar shifts,
then take the least shift over the possible endpoint multiplicities.
Recall $\lambda_j^{(\sigma)}=j(j-1+\sigma)$. For integers $0\le k\le m$
write $\ell=m-k$ and define
\begin{align}
 \xi_{m,k}&=\frac{2k+\kappa_0}{2m+\sigma},
 \label{eq:positive-charge-contact-site}\\
 \tau_{m,k}&=m(m+\sigma-1)-(2m+\sigma)
 \left\{\frac{k(k+\kappa_0-1)}{2k+\kappa_0}
       +\frac{\ell(\ell+\kappa_1-1)}{2\ell+\kappa_1}\right\},
 \qquad \gamma_m=\min_{0\le k\le m}\tau_{m,k}.
 \label{eq:positive-charge-threshold}
\end{align}
All denominators are positive, and a zero numerator at $k=0$ or $\ell=0$
is evaluated as zero.

\begin{theorem}[Exact maximal lower ray]
\label{thm:positive-charge-lower-ray}
For every $\eta\le\gamma_m$,
\[
 (\Jk-\eta\Id)\mathfrak E_{m+1}\subseteq\mathfrak E_{m+1}.
\]
For every sufficiently small $\varepsilon>0$, there is a monic
$f_\varepsilon\in\mathfrak E_{m+1}$ such that
$(\Jk-(\gamma_m+\varepsilon)\Id)f_\varepsilon$ has a nonreal conjugate
pair. Thus $(-\infty,\gamma_m]$ is the maximal preserving lower ray.
This does not classify any disconnected preservation set above $\gamma_m$.
\end{theorem}

The preservation class is essential. For the same operator,
the Bates--Yoshida Jacobi multiplier criterion~\cite{BatesYoshida} gives
preservation of all real-rooted polynomials, in every degree, precisely for
$-\kappa_0\kappa_1\le\eta\le0$. In contrast, Theorem~\ref{thm:positive-charge-lower-ray}
fixes the degree, allows one real zero outside $[0,1]$, and controls this
exterior count. Leake's finite-degree interval theorem~\cite[Theorem~7.8]{LeakeRepresentation}
requires both weak real-rootedness preservation and that every input rooted
in one interval have a nonzero image rooted in another. It does not impose
our one-exterior-zero condition. Weighted Rolle arguments and the fact that
the unshifted Jacobi operator does not increase the number of nonreal zeros
appear in~\cite{BuntonEtAlCZDO}; that count does not include exterior real
zeros. Related inverse problems study sets containing all input and output
zeros, including fixed-degree variants~\cite{AlexanderssonBrandenShapiroI,AlexanderssonHemmingssonShapiroII}.

For $g=x^k(1-x)^\ell$, direct logarithmic differentiation gives
\begin{align}
 V_{m,k}(x):=\frac{\Jk g}{g}
 &=\lambda_m^{(\sigma)}
 -\frac{k(k+\kappa_0-1)}x
 -\frac{\ell(\ell+\kappa_1-1)}{1-x},
 \label{eq:positive-charge-gauge}\\
 \frac{[\Jk,M_x]g}{g}&=(2m+\sigma)x-2k-\kappa_0.
 \label{eq:positive-charge-commutator}
\end{align}
Consequently, at $x\in(0,1)$,
\begin{equation}\label{eq:positive-charge-contact}
 (\Jk-\eta\Id)g=(\Jk-\eta\Id)(xg)=0
 \quad\Longleftrightarrow\quad x=\xi_{m,k},\quad \eta=\tau_{m,k}.
\end{equation}
At a companion zero avoiding the support of any interval-rooted degree-$m$
polynomial $g$ with roots $r_i$ counted with multiplicity, put
$S_j=\sum_i(x-r_i)^{-j}$,
$\Delta=x(1-x)$ and $\beta=\sigma x-\kappa_0$. The two equations give
\[
 S_1=\frac{\beta}{2\Delta},\qquad
 \eta=\Delta S_2+\frac{\beta^2}{4\Delta}>0.
\]
Thus $\tau_{m,k}>0$. At the two extreme multiplicities,
\[
 \tau_{m,0}=\frac{\kappa_0m(m+1)}{2m+\kappa_1},\qquad
 \tau_{m,m}=\frac{\kappa_1m(m+1)}{2m+\kappa_0}.
\]
They imply the useful strict bounds
\begin{equation}\label{eq:gamma-spectral-gap}
 0<\gamma_m<m\min(\kappa_0,\kappa_1)
 <\lambda_m^{(\sigma)}<\lambda_{m+1}^{(\sigma)}.
\end{equation}
Here $(m+1)/(2m+\kappa_i)<1$ because $m-1+\kappa_i>0$.

\subsection{Local protection and the endpoint obstruction}

The two-sided multiple-root tangent and regular local Rolle factorization
are proved in Lemmas~\ref{lem:two-sided-multiple-root-tangent} and
\ref{lem:local-weighted-rolle}. Those lemmas apply to arbitrary real regular
differential operators and do not depend on the odd block theorem. The
following order-two endpoint version completes the local inputs here.

\begin{lemma}[Order-two endpoint cluster protection]
\label{lem:positive-charge-endpoint-cluster}
Fix positive charges and a compact real interval of $\eta$. For
$T_\eta=\Jk-\eta\Id$, an input root of multiplicity $e\ge1$ at either
endpoint has output multiplicity exactly $e-1$. Under a nearby real-rooted
input and parameter perturbation, its $e-1$ forced output roots stay real;
their wrong-side count is no greater than the wrong-side count of the input
cluster. The neighborhoods can be chosen uniformly along the parameter
interval.
\end{lemma}
\begin{proof}
At zero,
\[
 T_\eta(x^e h)=-e(e-1+\kappa_0)h(0)x^{e-1}+O(x^e).
\]
To prove reality rather than just the limiting multiplicity, solve
$T_\eta v=0$ analytically with $v(0)=1$. Its coefficients satisfy
\[
 v(x)=\sum_{j\ge0}v_jx^j,\qquad
 v_{j+1}=\frac{j(j-1+\sigma)-\eta}{(j+1)(j+\kappa_0)}v_j.
\]
The quotients are uniformly bounded on compact $\eta$-intervals, giving a
common convergence disk. On a smaller common real interval $v>0$, and
\[
 T_\eta(vh)=-v(1-x)(x\partial_x+\gamma(x))h',\qquad
 \gamma(x)=2x\frac{v'}v+\frac{\kappa_0-\sigma x}{1-x}.
\]
Here $\gamma(0)=\kappa_0>0$, so $\gamma>0$ after shrinking once more.
On each side of zero use the positive weight
\[
 w(x)=|x|^{\kappa_0}
 \exp\left(\int_0^x\frac{\gamma(t)-\kappa_0}{t}\,dt\right).
\]
Then $(x\partial_x+\gamma)h=x(wh)'/w$ and $wh\to0$ at zero.
Multiplicity Rolle counting from this zero boundary value shows that the
Euler factor loses no roots on either closed half-interval. Its order at
zero equals that of $h$, so adding the half-interval counts and subtracting
the zero multiplicity gives the same conclusion on the two-sided interval.

Division by $v$ keeps the $e$ input roots. Differentiation gives at least
$e-1$ real roots in the cluster, and the Euler factor does not reduce that
number. Exact limiting multiplicity and complex root continuity show that
there are exactly $e-1$ output roots, so all are real. If the input cluster
has $q$ wrong-side roots, differentiation and the right half-interval Euler
count retain at least $e-q-1$ nonnegative roots when this number is positive.
The wrong-side output count is therefore at most $q$; if $e-q-1\le0$ the
same assertion follows from the total size $e-1\le q$. The reflection
$x\mapsto1-x$ exchanges the charges and proves the statement at one.
The convergence and positivity choices above are uniform in $\eta$.
\end{proof}

\begin{lemma}[Second-order strict deletion]
\label{lem:second-order-strict-deletion}
Let $T=\Jk-\eta\Id$. Suppose $g(\xi)\ne0$, $\xi\in(0,1)$ and
$Tg(\xi)=T(xg)(\xi)=0$. If $g=(x-r_i)g_i$, then $Tg_i(\xi)\ne0$.
\end{lemma}
\begin{proof}
The identities
\[
 Tg=(x-r_i)Tg_i+[\Jk,M_x]g_i,
\]
\[
 [\Jk,M_x]g=2x(x-1)g_i+(x-r_i)[\Jk,M_x]g_i
\]
show that $Tg_i(\xi)=0$ would force
$2\xi(\xi-1)g_i(\xi)=0$, which is impossible.
\end{proof}

\begin{lemma}[Endpoint signs and escape accounting]
\label{lem:positive-charge-escape}
Along a compact initial preserving path $0\le\eta\le\eta_c<\gamma_m$,
a sequence of real-rooted outputs at nearby parameters cannot lose the
one-defect bound at its weak-preserver boundary.
\end{lemma}
\begin{proof}
If an interval-rooted polynomial $h$ of degree $d\in\{m,m+1\}$ avoids
zero, then
\[
 \frac{T_\eta h(0)}{h(0)}
 =\kappa_0\sum_i\frac1{r_i}-\eta
 \ge d\kappa_0-\eta>0.
\]
The reflected value at one is positive as well. Inherited endpoint
multiplicity is always $e-1$, with the nonzero leading coefficient in
Lemma~\ref{lem:positive-charge-endpoint-cluster} independent of $\eta$.
For a fixed all-interval input, the already real-rooted output along the
initial preserving path therefore cannot cross either endpoint; the leading
multiplier is positive by \eqref{eq:gamma-spectral-gap}. Its output is
all-interval, also at the weak boundary by closure.

For a finite exterior defect write $f=(x-\rho)g$. An unforced output zero
at zero requires
\[
 0=T_\eta f(0)=-\rho T_\eta g(0)-\kappa_0g(0),\qquad
 \rho=-\kappa_0\frac{g(0)}{T_\eta g(0)}<0.
\]
At one it similarly requires
$\rho=1+\kappa_1g(1)/T_\eta g(1)>1$. Thus it occurs only on the defect
side. At a weak-preserver boundary it is simple: otherwise varying $\rho$
in both directions and applying the tangent lemma would force $T_\eta g=0$
at that endpoint. It cannot coexist with a separate exterior output root,
since its nonzero derivative with respect to $\rho$ lets it move outside
while the separate exterior root persists, violating weak preservation.

To sum the counts, choose pairwise disjoint neighborhoods of the limiting
inherited interior clusters, the two endpoint clusters, and any exterior or
infinity root. The other roots stay in a compact subset of $(0,1)$.
Regular Rolle protection keeps inherited interior clusters real and inside.
Lemma~\ref{lem:positive-charge-endpoint-cluster} bounds endpoint wrong-side
counts by those of the input clusters. For an all-interval limiting input,
only its single nearby allowed defect can contribute to those counts.
For an exterior limiting defect away from the endpoints, the constrained
endpoint clusters contribute zero, and only its exterior output continuation
or the simple unforced endpoint zero can contribute; these cannot coexist.
For an infinity defect, the affine degree-$m$ output is all-interval and the
projective infinity zero is simple by \eqref{eq:gamma-spectral-gap}; the
finite clusters contribute no exterior root. These cases give total exterior
count at most one for every sufficiently nearby real-rooted output.
\end{proof}

\begin{proof}[Proof of preservation in Theorem~\ref{thm:positive-charge-lower-ray}]
At a parameter $\eta$, write $T=\Jk-\eta\Id$.
For $\eta\le0$, Lemma~\ref{lem:positive-charge-shift} retains degree $m+1$
and at least $m$ roots in $[0,1]$. The last root cannot be nonreal without
its conjugate, so preservation holds in this range.

To extend preservation up to $\gamma_m$, suppose there is a failing value
$0<\eta_*<\gamma_m$. The two leading multipliers are positive on
$[0,\eta_*]$, so Lemma~\ref{lem:fixed-degree-first-exit} gives a weakly
preserving first boundary $\eta_c$ and a convergent sequence of failing
inputs at nearby parameters.

We exclude the possible limiting contacts. By
Lemma~\ref{lem:positive-charge-escape}, failure cannot consist only of two
real exterior roots. A nonreal pair must therefore converge to a multiple
real output root. This root cannot be exterior, by the weak one-defect
bound, and it cannot be at infinity, where the output root is simple.

Consider first an interior contact inherited from the input, of
multiplicity $e$. The
second-order leading coefficient gives output order $e-2$ for $e\ge2$.
The cases $e=2,3$ are nonzero or simple; for $e\ge4$ regular local Rolle
protection excludes a nonreal pair.

For $e=1$, select a different free labelled
root (the exterior defect if one exists). Its two-sided variation gives a
companion collision for a factor $g$ having a simple root at $\xi$, whereas
\[
 [\Jk,M_x]g(\xi)=2\xi(\xi-1)g'(\xi)\ne0.
\]
The infinity input uses $(1+\varepsilon x)g$ and gives the same contradiction.
Thus no inherited interior root can be the limit of the nonreal pair.

At an off-support contact, Lemma~\ref{lem:rigid-contact-shapes}, with
strict deletion supplied by Lemma~\ref{lem:second-order-strict-deletion},
reduces the input to an exterior-defect endpoint binomial, a pure power,
or an infinity-defect endpoint binomial. In the two endpoint-binomial
cases the tangent variation also gives $Tg(\xi)=T(xg)(\xi)=0$.

Pure powers have no off-support double zero. With $N=m+1$ and $z=x-r$,
the residual of $T(x-r)^N$ after division by $z^{N-2}$ is
\[
 (\lambda_N^{(\sigma)}-\eta)z^2
 +N\{(N-1)(2r-1)+\sigma r-\kappa_0\}z
 +N(N-1)r(r-1).
\]
For $0<r<1$ its leading coefficient is positive and constant negative, so
its roots are distinct and real. At $r=0,1$, removal of the further forced
factor leaves a nondegenerate linear polynomial. The remaining endpoint-
binomial contact would require $\eta_c=\tau_{m,k}\ge\gamma_m$ by
\eqref{eq:positive-charge-contact}, impossible.

It remains to consider the endpoints. A noninherited multiple zero is excluded by the strict
endpoint signs or the defect tangent in Lemma~\ref{lem:positive-charge-escape}.
Multiplicity one gives nonzero output, multiplicity two a simple root, and
higher inherited clusters are protected by
Lemma~\ref{lem:positive-charge-endpoint-cluster}. This excludes every possible
limit of a nonreal pair, and hence proves preservation for $\eta<\gamma_m$.

At $\eta=\gamma_m$, take a coefficient limit. The positive leading multiplier
in \eqref{eq:gamma-spectral-gap} retains exact degree at equality.
\end{proof}

\subsection{Sharp witnesses and spectral products}

\begin{proof}[Proof of sharpness in Theorem~\ref{thm:positive-charge-lower-ray}]
Choose a minimizing $k$, set $g=x^k(x-1)^{m-k}$, $\xi=\xi_{m,k}$,
$E=\gamma_m$, $C=2m+\sigma$, and $V=\Jk g/g$. On $(0,1)$,
\[
 V''(x)=-\frac{2k(k+\kappa_0-1)}{x^3}
        -\frac{2(m-k)(m-k+\kappa_1-1)}{(1-x)^3}<0.
\]
Put $v_1=V'(\xi)$, $v_2=V''(\xi)<0$, and
$f_*=(v_1(x-\xi)-C)g$. By
\eqref{eq:positive-charge-gauge}--\eqref{eq:positive-charge-contact},
\[
 R_\varepsilon:=\frac{(\Jk-(E+\varepsilon)\Id)f_*}{g}
 =(v_1(x-\xi)-C)(V-E-\varepsilon)+v_1C(x-\xi).
\]
Thus $R_0(\xi)=R'_0(\xi)=0$ and
\[
 R''_0(\xi)=2v_1^2-Cv_2>0,\qquad R_\varepsilon(\xi)=C\varepsilon.
\]
After $x-\xi=\sqrt\varepsilon\,t$ and division by $\varepsilon$, the limit is
$\tfrac12(2v_1^2-Cv_2)t^2+C$. Its two simple nonreal zeros persist by
Rouch\'e. If $v_1\ne0$, the real-rooted polynomial $f_*/v_1$ is monic of
degree $m+1$ with at least $m$ interval roots, so it is the required witness.
If $v_1=0$, then $f_*=-Cg$. For each fixed sufficiently small $\varepsilon>0$,
\[
 -M^{-1}(\Jk-(E+\varepsilon)\Id)((x-M)g)
       \longrightarrow(\Jk-(E+\varepsilon)\Id)g
\]
coefficientwise. Rouch\'e on disjoint disks about the two simple nonreal
zeros gives a finite monic witness $(x-M)g$ for large $M$, which may depend
on $\varepsilon$.
\end{proof}

\begin{corollary}[Optimal universal real-spectral cap]
\label{cor:universal-spectral-cap}
For every monic real-rooted polynomial $P$ whose roots are at most $\gamma_m$,
\[
 P(\Jk)\mathfrak E_{m+1}\subseteq\mathfrak E_{m+1}.
\]
The constant $\gamma_m$ is the largest cap valid for every such $P$, with
no bound on its degree or root multiplicities.
\end{corollary}
\begin{proof}
Compose the linear factors, whose top multipliers are positive by
\eqref{eq:gamma-spectral-gap}. A larger cap fails for a single factor
just above $\gamma_m$; no necessary rootwise condition on an individual
preserving product is asserted.
\end{proof}

The Legendre spectral products in~\cite{ChasseForgacsPiotrowski} likewise
illustrate why preserving a product is different from preserving each
factor separately.

\subsection{The fixed-spectrum specialization}

\begin{corollary}[The fixed-spectrum slice]
\label{cor:charge-symmetry-slice}
For $0<\nu<2$ and $\delta=\min(\nu,2-\nu)$,
\[
 \gamma_m(\nu,2-\nu)=\frac{\delta m(m+1)}{2m+2-\delta}.
\]
In particular, with $\Theta_m=\tfrac14+m(m+1)/(4m+3)$,
\begin{equation}\label{eq:odd-real-factor-specialization}
 (\mathcal L-\theta\Id)\mathfrak E_{m+1}\subseteq\mathfrak E_{m+1}
 \quad(\theta\le\Theta_m),
\end{equation}
and $\Theta_m$ is the exact endpoint of the maximal lower ray.
\end{corollary}
\begin{proof}
Writing $m_0=k$, $m_1=m-k$, polynomial division gives
\begin{equation}\label{eq:charge-curvature-form}
 \tau_{m,k}=m+\sigma-\frac{\sigma^2}{4}
 +\frac{2m+\sigma}{4}\sum_{i=0}^1
                \frac{\kappa_i(\kappa_i-2)}{2m_i+\kappa_i}.
\end{equation}
On the stated slice both numerator coefficients are negative, so this
function is concave for real $k\in[0,m]$. Its minimum is at an endpoint;
the smaller value is $\delta m(m+1)/(2m+2-\delta)$.
Take $\nu=1/2$ and shift the spectrum by $1/4$ for the final assertion.
\end{proof}

\begin{remark}[Eventual preservation at fixed charges]
\label{rem:eventual-real-spectral}
For fixed $\kappa_0,\kappa_1>0$, formula
\eqref{eq:charge-curvature-form} gives
\[
 \frac{\gamma_m}{m}\longrightarrow
 \min\left\{1,\frac{\kappa_0}{2},\frac{\kappa_1}{2}\right\}>0.
\]
Hence every fixed monic real-rooted polynomial $P$ preserves the class in
all sufficiently large input degrees. No cutoff uniform over the positive
charges, and no extension to nonreal spectral roots, is asserted.
\end{remark}

\subsection{A possible unpaired factor on the original seed}

\begin{lemma}[Single factor on a Jacobi endpoint pencil]
\label{lem:single-outer-endpoint-pencil}
Let \(0<\nu<2\), \(n\ge1\), and put
\[
 \mathfrak e_{n,\nu}
 =\frac{n(n+1)(2-\nu)}{2n+\nu}.
\]
If \(-\mathfrak e_{n,\nu}\le\alpha\le0\), then, for every \(t\in\R\),
\[
 (\JacOp_{\nu}+\alpha)\bigl(y^n(y+t)\bigr)
 \in\mathcal E_{n+1}.
\]
The lower endpoint \(-\mathfrak e_{n,\nu}\) is sharp for this fixed
\((n,\nu)\).  The sequence \(\mathfrak e_{n,\nu}\) is strictly increasing
in \(n\), so the uniform threshold is
\[
 \mathfrak e_{1,\nu}=\frac{2(2-\nu)}{2+\nu}.
\]
For \(n=0\), the image has degree at most one, and the conclusion holds
for every \(\alpha\).
\end{lemma}
\begin{proof}
For $j\ge1$ put $a_j=j(j+1)+\alpha$, $b_j=4j(j-1+\nu)$.
Then $(\JacOp_\nu+\alpha)y^j=y^{j-1}(a_jy-b_j)$.
For $\alpha\ge-\mathfrak e_{n,\nu}$ the two numbers $a_n,a_{n+1}$
are positive, and the nonzero roots $r_j=b_j/a_j$ satisfy
\[
 0<r_n\le r_{n+1}<4,
\]
because
\[
 a_na_{n+1}(r_{n+1}-r_n)
 =4\{n(n+1)(2-\nu)+(2n+\nu)\alpha\}\ge0.
\]
The endpoint bounds follow by substituting
$\mathfrak e_{n,\nu}<n(2-\nu)$ and
$\mathfrak e_{n,\nu}<(n+1)(2-\nu)$ in $4a_j-b_j$.
After removing $y^{n-1}$ and positive constants the output pencil is
\[
 y(y-r_{n+1})+s(y-r_n),\qquad s\in\mathbb R.
\]
Its value at $r_n$ is $r_n(r_n-r_{n+1})\le0$ and the endpoint budget is
$r_n(4-r_{n+1})>0$. The elementary quadratic criterion in
Lemma~\ref{lem:quadratic-endpoint-pencil-criterion} gives exactly the required
real-rootedness and one-defect count.

For $\alpha$ just below $-\mathfrak e_{n,\nu}$ the same positive leading
coefficients persist but $r_n>r_{n+1}$. Choosing $s=r_{n+1}-2r_n$ gives a
quadratic with positive minimum $r_n(r_n-r_{n+1})$, hence a nonreal pair.
Finally
\[
 \mathfrak e_{n+1,\nu}-\mathfrak e_{n,\nu}
 =\frac{2(n+1)(n+\nu)(2-\nu)}{(2n+\nu)(2n+\nu+2)}>0.
\]
The $n=0$ output has degree at most one.
\end{proof}
\section{Zero-orbit factors and total nonnegativity}
\label{sec:zero-orbit-factors}

The zeros of an even polynomial with real coefficients are stable under
$\rho\mapsto-\rho$ and $\rho\mapsto\bar\rho$.  We group the zeros,
counted with multiplicity, into orbits under these two symmetries.  Each
\emph{zero orbit} is a real pair $\{\pm a\}$, a purely imaginary pair
$\{\pm\mathrm{i}b\}$, or a reflected nonreal quartet
$\{\pm(a+\mathrm{i}b),\pm(a-\mathrm{i}b)\}$.  We choose $a\ge0$ for a real pair and
$b>0$ for a purely imaginary pair; for a genuinely nonreal quartet we take
$a,b>0$.  The degenerate real pair $a=0$ represents one factor $u^2$, and
repeated zeros correspond to repeated orbit factors.

Put $\lambda=u^2-1/4$.  The corresponding monic real even \emph{orbit
polynomials} are
\[
 \begin{aligned}
 \{\pm a\}:&\quad
 u^2-a^2
 =\lambda+\left(\frac14-a^2\right),\\
 \{\pm\mathrm{i}b\}:&\quad
 u^2+b^2
 =\lambda+\left(\frac14+b^2\right),\\
 \{\pm(a+\mathrm{i}b),\pm(a-\mathrm{i}b)\}:&\quad
 ((u-a)^2+b^2)((u+a)^2+b^2)
 =\left(\lambda+\frac14-a^2+b^2\right)^2+4a^2b^2.
 \end{aligned}
\]
A real pair is called \emph{outer} when $a>1/2$, equivalently when both
zeros lie outside $S_{1/2}$.  If an orbit polynomial is $h(\lambda)$, its
\emph{Jacobi spectral factor} is obtained by replacing $\lambda$ with
$\JacOp_\nu$, namely $h(\JacOp_\nu)$, because $u=r+1/2$ gives
$\lambda=r(r+1)$ on the Jacobi spectrum.

For later use, denote the quartet factor by
\[
 \mathscr A_{a,b}^{(\nu)}=
 (\JacOp_\nu+(\tfrac14-a^2+b^2)\Id)^2+4a^2b^2\Id.
\]
The four configurations are summarized below. For a real or imaginary pair,
$\mathfrak a=1/4-a^2$ or $1/4+b^2$, respectively. For outer real pairs,
$\alpha$ and $\beta$ denote the shifts $1/4-a^2$ associated with the
respective pairs.
The matrix shapes refer to the Bernstein representation defined below.

\begin{center}
\small
\renewcommand{\arraystretch}{1.18}
\begin{tabular}{@{}>{\raggedright\arraybackslash}p{0.20\textwidth}>{\raggedright\arraybackslash}p{0.28\textwidth}>{\raggedright\arraybackslash}p{0.14\textwidth}>{\raggedright\arraybackslash}p{0.21\textwidth}@{}}
\hline
orbit configuration & spectral factor & matrix shape & main tool\\
\hline
real pair in $S_{1/2}$, or a purely imaginary pair
& $\JacOp_{\nu}+\mathfrak a$
& tridiagonal
& positive-shift root counting\\
reflected nonreal quartet
& $\mathscr A_{a,b}^{(\nu)}$
& pentadiagonal
& pentadiagonal total nonnegativity\\
two outer real pairs
& $(\JacOp_{\nu}+\alpha)(\JacOp_{\nu}+\beta)$
& pentadiagonal
& divided-difference criterion\\
one unpaired outer real pair
& $\JacOp_{\nu}+\alpha$
& tridiagonal
& direct endpoint-pencil calculation\\
\hline
\end{tabular}
\end{center}

\paragraph{Bernstein matrices.}
For $P\in\R[y]_{\le N}$ with $N\ge0$, define the coefficient vector $c_N(P)$ by
\[
 P(4t)=\sum_{k=0}^N c_{N,k}(P)\mathcal B_k^N(t),\qquad
 \mathcal B_k^N(t)=\binom Nk t^k(1-t)^{N-k}.
\]
Increasing $N$ re-expands the same polynomial in a higher-degree basis.
If $\mathcal A$ leaves $\R[y]_{\le N}$ invariant, its matrix $[\mathcal A]_N$
is characterized by $c_N(\mathcal AP)=[\mathcal A]_N c_N(P)$. Set
\[
 \Sigma_N=\diag((-1)^k)_{k=0}^N,\qquad
 \mathbf J_{N,\nu}(\alpha)=\Sigma_N[\JacOp_\nu+\alpha]_N\Sigma_N.
\]
We call a matrix $A$ \emph{checkerboard TN} when $\Sigma_N A\Sigma_N$ is
TN, and write $I_{N+1}$ for the identity matrix. The explicit entries of
$\mathbf J_{N,\nu}(\alpha)$ are computed in
Lemma~\ref{lem:bernstein-jacobi-matrix}.

Recall that $\mathcal E_{N,q}$ is the interval-defect class from
Section~\ref{sec:sampling-framework}, and that
Proposition~\ref{prop:interval-root-monotonicity} converts nonsingular
checkerboard TN into preservation of every such class.  We call $S_{1/2}$ the
\emph{central strip}.  The first three configurations preserve every
$\mathcal E_{N,q}$ when, respectively,
\[
 \mathfrak a\ge0,\qquad |a|\le1,\qquad
 (\alpha,\beta)\in[-3/4,0]^2.
\]
The last range is the square required in
Section~\ref{sec:endpoint-sturm}; pairing the outer real orbits leaves at most
one unpaired pair.  Its factor need not preserve the whole one-defect class
and is therefore treated directly on the endpoint pencil.

\subsection{Total nonnegativity and preservation of every defect class}
Write
\[
 \mathbf R_{N,\nu}(\alpha,\beta)=
 \mathbf J_{N,\nu}(\alpha)\mathbf J_{N,\nu}(\beta).
\]
The following results control the number of zeros outside the interval
for arbitrary inputs, rather than only endpoint pencils. The quartet
threshold comes from a finite spectral-product phase bound for shifted
minors. Appendix~\ref{app:orbit-tn-proofs}
proves this bound and the results below. The larger odd
endpoint-pencil range will instead use the fixed-degree argument of
Section~\ref{sec:odd-enlarged-blocks}.

\begin{theorem}[Sharp total nonnegativity for a reflected quartet]
\label{thm:sharp-quartet-TN}
Let \(0<\nu<2\) and \(b>0\).  If \(\lvert a\rvert\le1\), then
\[
 \mathbf K_{N,\nu}(a,b)=
 \mathbf J_{N,\nu}\!\left(\frac14-a^2+b^2\right)^2
 +4a^2b^2I_{N+1}
\]
is nonsingular and totally nonnegative for every \(N\ge0\).  Conversely, if
\(\lvert a\rvert>1\), then for every fixed \(b>0\) and every
\(0<\nu<2\), the matrix \(\mathbf K_{N,\nu}(a,b)\) fails to be
totally nonnegative for at least one \(N\).  Thus
\(\lvert a\rvert\le1\) is the sharp range uniform in both \(N\) and
\(\nu\).
\end{theorem}

For $0<\nu<2$, $b>0$ and $|a|\le1$, Proposition~\ref{prop:interval-root-monotonicity}
therefore gives
\[
 \mathrm N_{[0,4]}(\mathscr A_{a,b}^{(\nu)}P)
 \ge\mathrm N_{[0,4]}(P)\quad(P\ne0),\qquad
 \mathscr A_{a,b}^{(\nu)}(\mathcal E_{N,q})\subseteq\mathcal E_{N,q}
 \quad(N,q\ge0).
\]
Here $P$ is real, and the operator is nonsingular on every degree space
by its spectral eigenvalues, also when $a=0$.

\begin{lemma}[Square pairing rule for two outer real factors]
\label{lem:square-outer-real-pairing}
Let $0<\nu<2$ and $\alpha,\beta\in[-3/4,0]$.  Then
$\mathbf R_{N,\nu}(\alpha,\beta)$ is totally nonnegative for every $N\ge0$.
Moreover, for every $N,q\ge0$,
\[
 (\JacOp_{\nu}+\alpha)(\JacOp_{\nu}+\beta)
 (\mathcal E_{N,q})\subseteq\mathcal E_{N,q}.
\]
\end{lemma}

\section{Exact thresholds for the first nontrivial endpoint pencil}
\label{sec:exact-first-pencil}

Throughout this section $0<\nu<2$. When $F(3/2)F(5/2)\ne0$, the adjacent
Jacobi images at $n=1$ have degrees one and two. Preservation of this
endpoint pencil then reduces to three scalar inequalities. For the sample
ratios arising here, the weak Obreschkoff inequality implies the other two.
We use this reduction to determine the exact first-row threshold and the
sharpness examples needed for the uniform theorem.

The sufficient uniform bound in
Theorem~\ref{thm:jacobi-family-endpoint-pencil} is already optimal for
$\nu\ge10/11$. For $0<\nu<10/11$, that argument is capped at $1$,
but the first pencil admits a larger exact strip. At $\nu=1/2$, its
optimal half-width is the constant $h_*$ of Theorem~\ref{thm:main-sturm};
Sections~\ref{sec:odd-enlarged-blocks}--\ref{sec:endpoint-sturm} show that
this value also works for every row.

Two source configurations determine the boundary: a repeated outer real
pair and a single unpaired outer real pair. Their transition gives the
two branches in the theorem below.
Define
\begin{equation}\label{eq:first-pair-critical-polynomial}
 \mathfrak p_\nu(\alpha)
 =(\nu+3)\alpha^4+28\alpha^3+(88-34\nu)\alpha^2
 +(112-56\nu)\alpha+48-24\nu.
\end{equation}
Let $\alpha_{\mathrm{rep}}(\nu)$ denote the zero in $(-1,-3/4)$; the
following calculation proves that it exists and is unique.  Put
\[
 \alpha_{\mathrm{sing}}(\nu)=\frac{2(\nu-2)}{\nu+2},
 \qquad
 \alpha_\sharp(\nu)
 =\max\{\alpha_{\mathrm{rep}}(\nu),\alpha_{\mathrm{sing}}(\nu)\},
\]
and
\begin{equation}\label{eq:first-pair-optimal-width}
 \mathfrak h_{1,\nu}^{\mathrm{opt}}
 =\sqrt{\frac14-\alpha_\sharp(\nu)}.
\end{equation}

\medskip
\noindent\emph{Monotonicity of the repeated-pair threshold.}
At the endpoints of $(-1,-3/4)$,
\[
 \mathfrak p_\nu(-1)=-(\nu+1)<0,
 \qquad
 \mathfrak p_\nu(-3/4)=\frac{9(75-23\nu)}{256}>0.
\]
Writing $s=\alpha+1\in[0,1/4]$, one has
\[
 \frac14\mathfrak p_\nu'(\alpha)
 =(3s^3+12s^2+11s+2)
 +\nu(s^3-3s^2-14s+2).
\]
This is affine in $\nu$. At $\nu=0$ it is positive, and at $\nu=2$
it is $5s^3+6s^2-17s+6\ge6-17/4>0$ on $[0,1/4]$.
Thus $\mathfrak p_\nu$ is strictly increasing on $[-1,-3/4]$, and
$\alpha_{\mathrm{rep}}(\nu)$ is well-defined and unique.

The factor-by-factor width $\mathfrak h_{1,\nu}$ is uniform in the quotient
index, whereas $\mathfrak h_{1,\nu}^{\mathrm{opt}}$ is exact for the single
pencil $n=1$.  One has
\[
 \mathfrak h_{1,\nu}^{\mathrm{opt}}\ge\mathfrak h_{1,\nu},
 \qquad
 \mathfrak h_{1,\nu}^{\mathrm{opt}}>\mathfrak h_{1,\nu}
 \quad\Longleftrightarrow\quad 0<\nu<\frac{10}{11}.
\]

\Needspace{8\baselineskip}
\begin{theorem}[Exact threshold for the first nontrivial endpoint pencil]
\label{thm:exact-first-jacobi-pair}
Let $0<\nu<2$, and let $F\not\equiv0$ be a real even polynomial or a real
even entire function of order at most one. If
$\Zset(F)\subseteq S_{\mathfrak h_{1,\nu}^{\mathrm{opt}}}$, then
\begin{equation}\label{eq:exact-first-jacobi-pencil}
 \JacMult_{F,\nu}\bigl(y(y+t)\bigr)
 \in\mathcal E_2
 \qquad(t\in\mathbb R).
\end{equation}
Consequently, $\JacMult_{F,\nu}(y)\preceq\JacMult_{F,\nu}(y^2)$.
The strip constant is optimal.  The two sharpness branches meet at
\[
 \nu_0=\frac{\sqrt{17}-1}{4},
\]
and
\begin{equation}\label{eq:first-pair-piecewise-extremal}
 \mathfrak h_{1,\nu}^{\mathrm{opt}}
 =\begin{cases}
 \displaystyle
 \sqrt{\frac14-\alpha_{\mathrm{rep}}(\nu)},
 &0<\nu\le\nu_0,\\[2mm]
 \displaystyle
 \sqrt{\frac14-\frac{2(\nu-2)}{\nu+2}},
 &\nu_0<\nu<2.
 \end{cases}
\end{equation}
In particular,
\[
 \mathfrak h_{1,1/2}^{\mathrm{opt}}=h_*,
 \qquad
 \mathfrak h_{1,3/2}^{\mathrm{opt}}=\sqrt{\frac{15}{28}}.
\]
\end{theorem}

\subsection{Scalar reduction and multiplicative closure}

We first separate the elementary geometry of a linear--quadratic pencil from
the Jacobi evaluation formulas.

\begin{lemma}[A quadratic endpoint-pencil criterion]
\label{lem:quadratic-endpoint-pencil-criterion}
Let \(A(y)=y-r\), and let \(B\) be a monic real quadratic.  Then
\[
 B+sA\in\mathcal E_2\qquad(s\in\mathbb R)
\]
if and only if
\begin{equation}\label{eq:quadratic-endpoint-pencil-criterion}
 0\le r\le4,\qquad B(r)\le0,\qquad B(r)+r(4-r)\ge0.
\end{equation}
\end{lemma}

\begin{proof}
Write $B(y)=(y-r)^2+c(y-r)+B(r)$. The discriminant of $B+sA$ is
$(c+s)^2-4B(r)$, so all members are real-rooted exactly when $B(r)\le0$;
their roots then bracket $r$. One root tends to $r$ as $s\to\pm\infty$,
while the other tends to $\mp\infty$. Thus $r\notin[0,4]$ is impossible.

For $r\in[0,4]$, both roots are exterior precisely when
$B(0)-sr<0$ and $B(4)+s(4-r)<0$. Such $s$ exists exactly when
\[
 (4-r)B(0)+rB(4)=4\bigl(B(r)+r(4-r)\bigr)<0.
\]
This equivalence includes $r=0,4$, where one inequality is independent
of $s$. Excluding such $s$ gives the last condition in
\eqref{eq:quadratic-endpoint-pencil-criterion}.
\end{proof}

The three elementary conditions can be expressed through the first three
spectral values of $F$.  Their behavior under products is then captured by a
multiplicative closure lemma.

\begin{lemma}[Exact scalar criterion for the first nontrivial endpoint pencil]
\label{lem:exact-first-pencil-scalar-criterion}
Let $0<\nu<2$, let $F$ be real and even, and assume that
$F(3/2)$ and $F(5/2)$ are nonzero and have the same sign.  After multiplying
$F$ by $-1$ if necessary, put
\[
 X_F=\frac{F(1/2)}{F(3/2)},
 \qquad
 Q_F=\frac{F(3/2)}{F(5/2)}>0,
\]
and define
\begin{align}
 \Phi_\nu(X,Q)
 &:=3\nu X^2+3(1-\nu)X+\nu-2-(\nu+1)QX,
 \label{eq:general-first-pair-Phi}\\
 \Omega_\nu(X,Q)
 &:=4-2\nu+3(\nu-1)X-(\nu+1)QX.
 \label{eq:general-first-pair-Omega}
\end{align}
Then
\[
 \JacMult_{F,\nu}\bigl(y(y+t)\bigr)\in\mathcal E_2
 \qquad(t\in\R)
\]
if and only if
\begin{equation}\label{eq:exact-first-pair-three-conditions}
 1-\frac2\nu\le X_F\le1,
 \qquad
 \Phi_\nu(X_F,Q_F)\le0,
 \qquad
 \Omega_\nu(X_F,Q_F)\ge0.
\end{equation}
In particular, $\Phi_\nu(X_F,Q_F)\le0$ is exactly the weak Obreschkoff
condition for the two adjacent images. If $Q_F\le1$, this inequality alone
is equivalent to preservation of the endpoint pencil.
\end{lemma}

The three inequalities have distinct geometric roles.  The bounds on $X_F$
place the zero of the linear image in $[0,4]$; the inequality
$\Phi_\nu\le0$ makes every member of the linear--quadratic pencil
real-rooted; and $\Omega_\nu\ge0$ prevents a real-rooted pencil member from
placing one root below $0$ and the other above $4$.

\begin{proof}
Put $\lambda_j=F(j+1/2)$. Expanding $y,y^2$ in the monic Jacobi basis gives,
after positive normalization,
\begin{equation}\label{eq:first-pair-linear-image}
 A_F(y):=\frac{\JacMult_{F,\nu}(y)}{\lambda_1}
 =y-r_F,
 \qquad
 r_F=2\nu(1-X_F),
\end{equation}
and
\begin{equation}\label{eq:first-pair-quadratic-image}
 \begin{split}
 B_F(y):=\frac{\JacMult_{F,\nu}(y^2)}{\lambda_2}
 ={}&y^2+2(\nu+1)(Q_F-1)y\\
 &+\frac{4\nu(\nu+1)}3(2Q_FX_F-3Q_F+1).
 \end{split}
\end{equation}
Because $\lambda_1/\lambda_2>0$, varying $t\in\R$ in the original
pencil is equivalent to varying $s\in\R$ in $B_F+sA_F$.  A direct
substitution gives
\begin{equation}\label{eq:first-pair-value-at-linear-root}
 B_F(r_F)=\frac{4\nu}{3}\Phi_\nu(X_F,Q_F).
\end{equation}

Applying Lemma~\ref{lem:quadratic-endpoint-pencil-criterion} to
\eqref{eq:first-pair-linear-image} and
\eqref{eq:first-pair-quadratic-image}, the condition $0\le r_F\le4$ becomes
the first inequality in \eqref{eq:exact-first-pair-three-conditions}, while
\eqref{eq:first-pair-value-at-linear-root} gives the second.  Finally,
\[
 \Phi_\nu(X,Q)+3(1-X)(2-\nu+\nu X)=\Omega_\nu(X,Q),
\]
so the lower bound
$B_F(r_F)\ge-r_F(4-r_F)$ is exactly the third inequality.

For fixed $0<Q\le1$, the quadratic $\Phi_\nu(X,Q)$ opens upward, is negative
at $X=0$, and satisfies
\[
 \Phi_\nu(1-2/\nu,Q)=\frac{2-\nu}{\nu}\bigl((\nu+1)Q+3-\nu\bigr)>0,
 \qquad \Phi_\nu(1,Q)=(\nu+1)(1-Q)\ge0.
\]
Its two zeros therefore lie in $[1-2/\nu,1]$. At each zero the preceding
identity gives $\Omega_\nu=3(1-X)(2-\nu+\nu X)\ge0$.
Since $\Omega_\nu(X,Q)$ is affine in $X$, it is nonnegative between these
zeros, precisely where $\Phi_\nu(X,Q)\le0$. This proves the final assertion.
\end{proof}

\Needspace{12\baselineskip}
\begin{lemma}[Multiplicative closure at the first nontrivial endpoint pencil]
\label{lem:general-first-pair-multiplicative}
Let $0<\nu<2$.
\begin{enumerate}[label=\textup{(\roman*)}]
\item Suppose that $0<X_i\le1$, $0<Q_i\le1$, and
$\Phi_\nu(X_i,Q_i)\le0$ for $i=1,2$.  Then
\[
 \Phi_\nu(X_1X_2,Q_1Q_2)\le0.
\]
\item Let
\[
 0<p\le\frac{2-\nu}{2\nu},
 \qquad
 0<X_0,Q_0\le1.
\]
Then
\[
 \Phi_\nu\left(-pX_0,\frac{Q_0}{3+2p}\right)\le0.
\]
\end{enumerate}
\end{lemma}

\begin{proof}
For $X>0$, put
\[
 \psi_\nu(X)
 =\frac{3\nu X^2+3(1-\nu)X+\nu-2}{(\nu+1)X}.
\]
Then $\Phi_\nu(X,Q)\le0$ is equivalent to $Q\ge\psi_\nu(X)$.  Exact
factorization gives, for $0<s,t\le1$,
\begin{equation}\label{eq:general-first-pair-closure-identity}
 \begin{split}
 &(\nu+1)^2st
 \bigl(\psi_\nu(s)\psi_\nu(t)-\psi_\nu(st)\bigr)\\
 &\quad=3(s-1)(t-1)
 \bigl((2-\nu)+\nu(2-\nu)(s+t)+\nu(2\nu-1)st\bigr).
 \end{split}
\end{equation}
The last factor equals
\[
 (2-\nu)\{1+\nu s(1-t)+\nu t(1-s)\}+3\nu st>0.
\]
Apply the identity with $s=X_1$ and $t=X_2$.
If $\psi_\nu(st)\le0$, assertion~\textup{(i)} follows from $Q_1Q_2>0$.
Otherwise $\psi_\nu'(X)=(3\nu+(2-\nu)/X^2)/(\nu+1)>0$ and
$st\le s,t$ imply $\psi_\nu(s),\psi_\nu(t)>0$.
The hypotheses can then be multiplied, and
\eqref{eq:general-first-pair-closure-identity} gives
$Q_1Q_2\ge\psi_\nu(s)\psi_\nu(t)\ge\psi_\nu(st)$.

For \textup{(ii)}, $\Phi_\nu(-pX,Q)$ is increasing in $Q$, so it is enough
to put $Q_0=1$.  The function
\[
 f_{\nu,p}(X)
 =\Phi_\nu\left(-pX,\frac1{3+2p}\right)
\]
is convex, with
\[
 f_{\nu,p}(0)=\nu-2<0,
 \qquad
 f_{\nu,p}(1)
 =\frac{3(p+1)^2(2\nu p+\nu-2)}{2p+3}\le0.
\]
Thus $f_{\nu,p}(X)\le0$ on $[0,1]$.
\end{proof}

\begin{lemma}[Coefficient certificate for reflected quartets]
\label{lem:first-pair-quartet-certificate}
Let $0<\nu<2$, let
$\alpha_\sharp(\nu)\le\alpha\le1/4$, and let $B\ge0$.  Define
\[
 \mathfrak q_\nu(\alpha)
 =2(\nu+1)\alpha^3+(19-16\nu)\alpha^2
 +(78-42\nu)\alpha+60-25\nu
\]
and
\begin{equation}\label{eq:general-first-pair-quartet-polynomial}
 \begin{split}
 \mathcal C_{\nu,\alpha}(B)
 ={}&(\nu+3)B^4+2(2(\nu+1)\alpha+15)B^3\\
 &+\bigl(99-27\nu+(48-24\nu)\alpha
          -(14+10\nu)\alpha^2\bigr)B^2\\
 &+2\mathfrak q_\nu(\alpha)B+\mathfrak p_\nu(\alpha).
 \end{split}
\end{equation}
Then $\mathcal C_{\nu,\alpha}(B)\ge0$.
\end{lemma}

\begin{proof}
The leading two coefficients in \eqref{eq:general-first-pair-quartet-polynomial}
are positive since $\alpha>-1$. The coefficient of $B^2$ is concave in
$\alpha$ on $[-1,1/4]$, with positive endpoint values
$37-13\nu$ and $(881-269\nu)/8$.

The monotonicity calculation above and
$\mathfrak p_\nu(-3/4)>0$ give
$\mathfrak p_\nu\ge0$ on $[\alpha_{\mathrm{rep}}(\nu),-3/4]$.
On $[-3/4,1/4]$ the same follows from
\[
 \mathfrak p_\nu(\alpha)
 =(1-\nu/2)(\alpha+2)^2(3\alpha^2+16\alpha+12)
 +\frac\nu2\alpha^2(5\alpha^2+28\alpha+20),
\]
whose two quadratic factors are positive there. Finally,
\[
 \begin{aligned}
 (\alpha+2)\mathfrak q_\nu(\alpha)-\mathfrak p_\nu(\alpha)
  =(\alpha+1)\bigl[&(\nu-1)\alpha^3-(13\nu+4)\alpha^2\\
                 &+(32-27\nu)\alpha+72-26\nu\bigr].
 \end{aligned}
\]
The cubic factor is concave on $[-1,1/4]$: its second derivative is
$\nu(6\alpha-26)-6\alpha-8<0$. Its endpoint values are
$37-13\nu$ and $(5103-2147\nu)/64$, both positive.
Since $\alpha\ge\alpha_\sharp\ge\alpha_{\mathrm{rep}}>-1$,
the identity gives $\mathfrak q_\nu(\alpha)>0$. All coefficients of
$\mathcal C_{\nu,\alpha}$ are therefore nonnegative.
\end{proof}

\subsection{Orbitwise proof and sharpness}

\begin{proof}[Proof of Theorem~\ref{thm:exact-first-jacobi-pair}]
\medskip
\noindent\emph{Step 1: threshold geometry.}
The monotonicity calculation above establishes the repeated-pair threshold.
To compare it with the single-pair threshold, evaluate
$\mathfrak p_\nu$ at the latter:
\begin{equation}\label{eq:first-pair-crossover-evaluation}
 \mathfrak p_\nu(\alpha_{\mathrm{sing}}(\nu))
 =-\frac{128\nu(\nu-2)(\nu+1)(2\nu^2+\nu-2)}{(\nu+2)^4}.
\end{equation}
The function $\alpha_{\mathrm{sing}}(\nu)$ is strictly increasing, with
\[
 \alpha_{\mathrm{sing}}(2/3)=-1,
 \qquad
 \alpha_{\mathrm{sing}}(10/11)=-3/4.
\]
For $0<\nu\le2/3$ it lies to the left of
$\alpha_{\mathrm{rep}}(\nu)$.  For
$2/3<\nu<10/11$, both quantities lie in $(-1,-3/4)$, where
$\mathfrak p_\nu$ is strictly increasing.  In this interval,
\eqref{eq:first-pair-crossover-evaluation} shows that their order changes
exactly when
$2\nu^2+\nu-2=0$, namely at
$\nu_0=(\sqrt{17}-1)/4$.  Finally, for
$10/11\le\nu<2$, one has
$\alpha_{\mathrm{sing}}(\nu)\ge-3/4>
\alpha_{\mathrm{rep}}(\nu)$.  This proves the branch formula
\eqref{eq:first-pair-piecewise-extremal}.

\medskip
\noindent\emph{Step 2: individual zero-orbit factors.}
Assume first that $F$ is a polynomial.  The threshold satisfies
$\alpha_\sharp(\nu)>-1$, so
$\mathfrak h_{1,\nu}^{\mathrm{opt}}<\sqrt5/2<3/2$.  Thus the real points
\(3/2\) and \(5/2\) lie strictly to the right of the permitted zero strip.
The real-valued function \(F\) has no real zero on \([3/2,\infty)\), so
$F(3/2)$ and $F(5/2)$ have the same nonzero sign.  Normalize them to be
positive and use the parameters $(X_F,Q_F)$ from
Lemma~\ref{lem:exact-first-pencil-scalar-criterion}.  If $F$ is constant,
then $(X_F,Q_F)=(1,1)$ and all three conditions in
\eqref{eq:exact-first-pair-three-conditions} hold with equality where
appropriate, so the assertion is immediate.

We may therefore assume that $F$ has at least one zero orbit. The
parameters multiply under products of
even real zero-orbit polynomials.  Every such factor has $|X|\le1$ and
$0<Q\le1$ in the present strip.  For the degree-two factors this is immediate
from the formulas below; for a nonreal quartet it follows from the differences
displayed when that factor is treated.

For every quadratic orbit write $F(u)=u^2+\alpha-1/4$:
$\alpha=1/4-a^2$ for a real pair and $\alpha=1/4+b^2$ for an imaginary pair.
Then
\begin{equation}\label{eq:general-first-pair-real-atom}
 X=\frac{\alpha}{\alpha+2},
 \qquad
 Q=\frac{\alpha+2}{\alpha+6},
\end{equation}
and
\begin{equation}\label{eq:general-first-pair-real-Phi}
 \Phi_\nu(X,Q)
 =-\frac{12((\nu+2)\alpha+4-2\nu)}
 {\,(\alpha+2)^2(\alpha+6)}.
\end{equation}
For $\alpha>0$ this gives $0<X,Q<1$ and $\Phi_\nu(X,Q)<0$,
including imaginary pairs. At $\alpha=0$ one has $X=0$, treated below.

Consider a reflected nonreal quartet
\[
 E_{a,b}(u)=((u-a)^2+b^2)((u+a)^2+b^2),
 \qquad b>0,
\]
and write $\alpha=1/4-a^2$, $B=b^2$, and
\[
 D_j=\left|\left(j+\frac12\right)^2-(a+\mathrm{i}b)^2\right|^2
 \qquad(j=0,1,2).
\]
Lemma~\ref{lem:first-pair-quartet-certificate} defines
$\mathcal C_{\nu,\alpha}(B)$ and shows that it is nonnegative throughout the
present parameter range.  Since $X=D_0/D_1$ and $Q=D_1/D_2$, a direct
calculation gives
\begin{equation}\label{eq:general-first-pair-quartet-Phi}
 \Phi_\nu(X,Q)
 =-\frac{24\mathcal C_{\nu,\alpha}(B)}{D_1^2D_2}\le0.
\end{equation}
Moreover,
\[
 D_1-D_0=4(1+\alpha+B)>0,
 \qquad
 D_2-D_1=8(4+\alpha+B)>0,
\]
so $0<X\le1$ and $0<Q\le1$.

\medskip
\noindent\emph{Step 3: paired outer real factors and composition.}
It remains to group the real pairs with $\alpha<0$ two at a time.  For such a
pair put
\[
 p=\frac{a^2-1/4}{9/4-a^2}=-\frac\alpha{\alpha+2}.
\]
Thus $0<p\le p_\sharp$, where
$p_\sharp=-\alpha_\sharp/(\alpha_\sharp+2)$.  For two such pairs with
parameters $p_1,p_2$, the product has
\[
 X=p_1p_2,
 \qquad
 Q=\frac1{(3+2p_1)(3+2p_2)}.
\]
The function
\[
 (p_1,p_2)\longmapsto
 \Phi_\nu\left(p_1p_2,\frac1{(3+2p_1)(3+2p_2)}\right)
\]
is separately convex on the square $[0,p_\sharp]^2$, since
\[
 \frac{\partial^2}{\partial p_1^2}
 \Phi_\nu\left(p_1p_2,\frac1{(3+2p_1)(3+2p_2)}\right)
 =6\nu p_2^2+\frac{12(\nu+1)p_2}{(3+2p_2)(3+2p_1)^3}\ge0,
\]
and similarly in $p_2$. It suffices to check the corners: three have
value $\nu-2$, while the remaining one has value
\[
 \Phi_\nu\left(p_\sharp^2,\frac1{(3+2p_\sharp)^2}\right)
 =-\frac{24\mathfrak p_\nu(\alpha_\sharp)}
 {(\alpha_\sharp+2)^4(\alpha_\sharp+6)^2}\le0
\]
by the definition of $\alpha_\sharp$ and
\eqref{eq:first-pair-crossover-evaluation}.  Hence every product of two real
pairs with negative shifts satisfies
\[
 0<X\le1,\qquad 0<Q\le1,\qquad \Phi_\nu(X,Q)\le0.
\]

If a factor has $\alpha=0$, the product has $X_F=0$, so
$\Phi_\nu=\nu-2<0$ and $\Omega_\nu=4-2\nu>0$.
Otherwise form blocks as follows: individual quartets, individual quadratic
orbits with $\alpha>0$, and pairs of quadratic orbits with $\alpha<0$.
Their parameters have $0<X,Q\le1$ and $\Phi_\nu\le0$.
Lemma~\ref{lem:general-first-pair-multiplicative}(i) gives the same
conditions for their product $(X_+,Q_+)$, with $(1,1)$ for no blocks.
At most one negative-shift factor remains. If none remains,
$(X_F,Q_F)=(X_+,Q_+)$. Otherwise its parameter
\[
 p=-\frac\alpha{\alpha+2}\le-\frac{\alpha_{\rm sing}}{\alpha_{\rm sing}+2}
   =\frac{2-\nu}{2\nu}
\]
gives $(X_F,Q_F)=(-pX_+,Q_+/(3+2p))$.
Part (ii) of that lemma supplies $\Phi_\nu(X_F,Q_F)\le0$.
We have obtained $0<Q_F\le1$ and $\Phi_\nu(X_F,Q_F)\le0$ for the whole source.
The final assertion of Lemma~\ref{lem:exact-first-pencil-scalar-criterion}
therefore proves \eqref{eq:exact-first-jacobi-pencil} for polynomials.
For entire $F$, use Lemma~\ref{lem:derivative-strip}, the finite-sample
continuity from Section~\ref{sec:sampling-framework}, and closedness of
$\mathcal E_2$.

\medskip
\noindent\emph{Step 4: sharpness and endpoint values.}
For sharpness, first suppose $0<\nu<\nu_0$.  Given
$h>\mathfrak h_{1,\nu}^{\mathrm{opt}}$, choose $a$ slightly larger than the
optimal width but still smaller than $\min\{h,\sqrt5/2\}$, and take the
repeated-pair polynomial $F(u)=(u^2-a^2)^2$.  With $\alpha=1/4-a^2$, one has
$\alpha<\alpha_{\mathrm{rep}}(\nu)$ and hence
\[
 \Phi_\nu(X_F,Q_F)
 =-\frac{24\mathfrak p_\nu(\alpha)}
 {(\alpha+2)^4(\alpha+6)^2}>0.
\]
Thus even weak interlacing fails.  If $\nu>\nu_0$, take instead the
single-pair polynomial $F(u)=u^2-a^2$; then
$\alpha<\alpha_{\mathrm{sing}}(\nu)$ and
\eqref{eq:general-first-pair-real-Phi} is positive.  At $\nu=\nu_0$, either
family gives the same boundary.  This proves optimality.

At $\nu=1/2$, the repeated-pair polynomial
$2\mathfrak p_{1/2}$ is
$7\alpha^4+56\alpha^3+142\alpha^2+168\alpha+72$.
This is the polynomial defining $\alpha_*$, so uniqueness on $(-1,-3/4)$
gives $\alpha_{\mathrm{rep}}(1/2)=\alpha_*$ and
$\mathfrak h_{1,1/2}^{\mathrm{opt}}=h_*$.
At $\nu=3/2$, the single-pair branch gives
$\alpha_\sharp=-2/7$ and hence
$\mathfrak h_{1,3/2}^{\mathrm{opt}}=\sqrt{15/28}$.
\end{proof}

\begin{remark}[Comparison with the odd block range]
\label{rem:odd-first-row-constant}
The odd first-row width $h_*$ lies strictly inside the block range used in
the next section:
\[
 h_*^2<\frac87<A_{\rm ext}:=\frac{35-8\sqrt7}{12}<\frac76.
\]
For the first inequality, $\mathfrak p_{1/2}(-25/28)<0$ and
$\mathfrak p_{1/2}(-3/4)>0$, so uniqueness places $\alpha_*$ above
$-25/28$. The remaining inequalities follow by squaring positive quantities.
This first-row example supplies the sharpness part of
Theorem~\ref{thm:odd-uniform-preservation}.
\end{remark}
\section{Odd block preservation beyond the TN range}
\label{sec:odd-enlarged-blocks}
\label{sec:global-normalization}
The first-row calculation has identified $h_*$ as the uniform target in
the odd case. To reach that width at higher rows, we prove one-defect
preservation for paired real factors and nonreal quartet factors beyond
the width-one TN range.

The argument is at fixed degree $N\ge3$: we start from known preservers,
vary the block parameters continuously, and exclude the first possible
loss of preservation. The counterexample below shows why single-factor
preservation alone does not suffice.

We use $x=y/4$ and the shifted odd operator $\mathcal L$ from
Section~\ref{subsec:normalized-exact-degree}, whose eigenvalues are
$(j+\tfrac12)^2$. The constant $A_{\rm ext}$ was defined in
Remark~\ref{rem:odd-first-row-constant}; its square root exceeds $h_*$.
For $A,B\ge0$, set
\[
 \mathcal Q_{A,B}=(\mathcal L+(B-A)\Id)^2+4AB\Id.
\]
A real source pair gives $\mathcal L-a^2\Id$, an imaginary pair gives
$\mathcal L+b^2\Id$, and the orbit $\{\pm(a+\mathrm{i}b),\pm(a-\mathrm{i}b)\}$ gives
$\mathcal Q_{a^2,b^2}$.

\begin{remark}[Failure of single-factor preservation]
\label{rem:E4-single-factor-obstruction}
Exact degree does not make every target-strip real factor a full
\(\mathfrak E_4\)-preserver.  For example, at \(\theta=11/10\),
\[
 f(x)=\left(x+\frac12\right)(x-1)^3
\]
satisfies
\[
 (\mathcal L-11/10)f
 =\frac{(x-1)^2}{40}(766x^2-143x+7),
\]
and the displayed quadratic has discriminant $-999<0$. Thus a real factor
in the target strip need not preserve the full class. We pair such factors
into blocks, applying any unpaired factor first to the endpoint seed.
\end{remark}

\subsection{Contact geometry for connected paths}
\label{sec:global-connected-paths}

For real \(\Omega,\chi\), set
\[
 \mathcal T_{\Omega,\chi}
  :=\bigl(\mathcal L+(\Omega-\tfrac54)\operatorname{Id}\bigr)^2
       +\chi\operatorname{Id}.
\]
If
\[
 T=\sum_{j=0}^d c_j(x)\partial_x^j,
\]
then the frozen symbol of \(T\) at an observation point \(\xi\) is
\[
 p_\xi(t)=\sum_{j=0}^d c_j(\xi)t^j.
\]
A contact is called off-support when the observation point is not a zero of
the input polynomial.  For \(g=x^k(1-x)^\ell\), we call
\(\xi\in(0,1)\) an endpoint-binomial companion collision if
\[
 Tg(\xi)=T(xg)(\xi)=0.
\]
A matrix is \(\mathrm{TN}_2\) if all of its minors of order one and two are
nonnegative.  For an infinite monomial matrix, this means that every finite
truncation has this property.

Appendix~\ref{app:global-contact-lemmas} proves the frozen-symbol,
strict-deletion and cluster-protection lemmas used below, together with
the pure-power exclusion. We invoke each at the point where it is needed.

\begin{theorem}[Connected admissible-path preservation]
\label{thm:connected-path-preserver}
Fix an integer \(N\ge3\).  Let
\(\mathcal T_\tau=\mathcal T_{\Omega_\tau,\chi_\tau}\),
\(0\le\tau\le1\), be coefficientwise continuous and assume:
\begin{enumerate}[label=\textup{(\alph*)}]
\item \(\Omega_\tau>0\) and
      \(-\min\{\Omega_\tau^2,1\}\le\chi_\tau\le5\Omega_\tau\);
\item the constant differential coefficient
      \((\Omega_\tau-1)^2+\chi_\tau\) is nonnegative;
\item \(\chi_\tau\le3\Omega_\tau^2+10\Omega_\tau-1\);
\item \(\mathcal T_0\) preserves \(\mathfrak E_N\).
\end{enumerate}
If endpoint-binomial companion collisions are excluded for every pair
\(k,\ell\in\{0\}\cup[1,\infty)\) with \(k+\ell=N-1\), then every
\(\mathcal T_\tau\) preserves \(\mathfrak E_N\).
\end{theorem}
Conditions (a)--(b) supply the interior and endpoint contact estimates;
(a) also prevents degree loss, and (c) controls real exterior zeros.
Starting from (d), the proof classifies a possible first loss of
preservation. Exterior counts exclude failure through two real exterior
roots, while the leading multipliers keep the projective root at infinity
simple. The remaining failures must pass through multiple real contacts.
The local lemmas either exclude contacts inherited from input roots or
protect their forced clusters. Off-support rigidity then leaves pure
powers and endpoint-supported factors with one free zero. A separate
pure-power exclusion leaves precisely the companion contacts in the
hypothesis.

\begin{proof}
At each path parameter, the degree-$j$ multiplier for $j\ge1$ satisfies
\[
 \bigl(j(j+1)+\Omega_\tau-1\bigr)^2+\chi_\tau
 \ge(\Omega_\tau+1)^2-\Omega_\tau^2=1+2\Omega_\tau>0.
\]
Thus the degree-$N$ and degree-$(N-1)$ multipliers are positive.
The exterior calculation in Appendix~\ref{app:global-contact-lemmas} shows
that (a)--(b) make both exterior matrices entrywise nonnegative and that
(c) is equivalent to their $\mathrm{TN}_2$ property.

Suppose preservation first fails.  Lemma~\ref{lem:fixed-degree-first-exit}
selects a first parameter \(\tau_c\) at which \(\mathcal T_{\tau_c}\)
weakly preserves the compactified class and a limiting output lies on its
boundary.

The resulting \(\mathrm{TN}_2\) property and Descartes' rule exclude a failure made only
of two real exterior roots: a one-defect input can produce at most one real
exterior output root, on the same side as the input defect.  Hence a failing
sequence contains a nonreal conjugate pair.  That pair converges to a multiple
real projective root; simple real roots persist as real.

We first exclude contacts inherited from an input root
$\xi\in(0,1)$ of multiplicity $e$. For $e\ge4$, the nonzero
fourth-order coefficient $\xi^2(1-\xi)^2$ gives exact output order
$e-4$. A multiple contact therefore requires $e\ge6$, and its forced
cluster is protected by Lemma~\ref{lem:local-weighted-rolle}. This applies
both to finite inputs and to affine representatives of infinity inputs.

For $1\le e\le3$, write a finite input as $f=(x-\beta)g$, with $g$
interval-rooted and $1\le\operatorname{ord}_\xi g\le3$. Choose the
exterior root as $\beta$ if one exists; otherwise choose a copy of $\xi$
when $e\ge2$, and a different root when $e=1$ (possible since $N\ge3$).
Moving $\beta$ in both directions is admissible. At a multiple output
root, Lemma~\ref{lem:two-sided-multiple-root-tangent} gives
$\mathcal T_{\tau_c}g(\xi)=0$; together with
$\mathcal T_{\tau_c}f(\xi)=0$, this implies
$\mathcal T_{\tau_c}(xg)(\xi)=0$, contradicting
Lemma~\ref{lem:inherited-multiplicity}. For an infinity input, its
interval-rooted affine representative $f$ admits the two-sided variation
$(1+s x)f$. The same tangent condition gives
$\mathcal T_{\tau_c}(xf)(\xi)=0$, and
Lemma~\ref{lem:inherited-multiplicity} applied to $f$ gives the same
contradiction.

At an endpoint, Lemma~\ref{lem:admissible-endpoint-contacts}
requires endpoint input multiplicity at least four and gives exact output
multiplicity \(e-2\).  Such a contact is impossible when \(N=3\), and at
higher degrees it is protected by the quadratic lifts and weighted Rolle
factorization.  The projective root at infinity stays simple because the two
leading multipliers are positive.

Thus the limiting multiple contact is real, interior, and avoids the full
input support. Lemma~\ref{lem:rigid-contact-shapes} applies by
Lemma~\ref{lem:strict-deletion-global} and leaves three shapes.

First suppose \(f=(x-r)^N\) is a pure power. Move one labelled copy of \(r\):
\[
 f_\varepsilon=(x-r-\varepsilon)(x-r)^{N-1}
 =f-\varepsilon(x-r)^{N-1}.
\]
This is a two-sided admissible one-defect variation even for \(r=0,1\).
The tangent lemma yields
\(\mathcal T_{\tau_c}(x-r)^{N-1}(\xi)=0\).
Writing $S_N(x,r)=\mathcal T_{\tau_c}((x-r)^N)$, we obtain
\[
 S_N(\xi,r)=\partial_xS_N(\xi,r)=\partial_rS_N(\xi,r)=0,
 \qquad
 \partial_rS_N=-N\mathcal T_{\tau_c}(x-r)^{N-1},
\]
contradicting Theorem~\ref{thm:diagonal-transversality-appendix}.

The remaining finite- and infinity-defect shapes have
\(g=x^k(1-x)^\ell\), \(k+\ell=N-1\).  Moving the finite defect, or using the
normalized infinity variation, gives
\[
 \mathcal T_{\tau_c}g(\xi)
 =\mathcal T_{\tau_c}(xg)(\xi)=0.
\]
This is the excluded endpoint-binomial companion collision.
\end{proof}

\subsection{Excluding endpoint-binomial companion collisions}
\label{sec:global-charge-gap}

We now prove the companion-collision exclusion needed in
Theorem~\ref{thm:connected-path-preserver}.  The exponents \(k\) and \(\ell\)
record the multiplicities at the two endpoints, and their sum
\(m=k+\ell\) is the total endpoint exponent.  These exponents may be real;
nonintegral powers are used only as smooth positive functions on \((0,1)\).
Let
\[
 k,\ell\in\{0\}\cup[1,\infty),
 \qquad m=k+\ell\ge2,
 \qquad g(x)=x^k(1-x)^\ell,
\]
and abbreviate $\mathcal T=\mathcal T_{\Omega,\chi}$. Put
\[
 \Delta=x(1-x),\qquad \mathfrak d=2(m+1)x-2k-\tfrac12,
\]
\[
 \mathcal V=(m+\tfrac12)^2+\Omega-\tfrac54
       -\frac{k(k-\tfrac12)}{x}-\frac{\ell(\ell+\tfrac12)}{1-x}.
\]
Logarithmic differentiation gives
\begin{equation}
\label{eq:global-gauge-identities}
 \mathfrak u:=\frac{\mathcal T_{\Omega,\chi}g}{g}
  =\mathcal V^2-\Delta\mathcal V''+\mathfrak d\mathcal V'+\chi,
 \qquad
 \mathfrak w:=\frac{[\mathcal T_{\Omega,\chi},x]g}{2g}
  =(\mathcal V+m+1)\mathfrak d-\Delta\mathcal V'.
\end{equation}
Since \(\mathcal T(xg)=x\mathcal Tg+[\mathcal T,x]g\), a companion
collision is exactly \(\mathfrak u=\mathfrak w=0\).

\begin{lemma}[Coercivity at the two endpoint poles]
\label{lem:global-two-pole-coercivity}
If \(\Omega\ge0\), \(m\ge2\), and each endpoint exponent is either zero or
at least one, then
\[
 \mathcal V^2-\Delta\mathcal V''
 >\min\left\{(m/2-1+\Omega)^2,1\right\}
 \ge\min\{\Omega^2,1\}.
\]
\end{lemma}
\begin{proof}
If both exponents are nonzero, then
\(k(k-\tfrac12)\ge1/2\) and \(\ell(\ell+\tfrac12)\ge3/2\), whence
\[
 -\Delta\mathcal V''
 \ge\frac{1-x}{x^2}+\frac{x}{(1-x)^2}\ge4.
\]
Suppose one exponent vanishes.  If \(k=0\), set
\[
 a=m(m+\tfrac12),\qquad c=m/2-1+\Omega,
 \qquad \vartheta=\frac{x}{1-x};
\]
if \(\ell=0\), set
\[
 a=m(m-\tfrac12),\qquad c=3m/2-1+\Omega,
 \qquad \vartheta=\frac{1-x}{x}.
\]
In either case \(a>0\), \(\vartheta>0\),
\(c\ge m/2-1+\Omega\ge0\), and
\[
 \mathcal V=c-a\vartheta,
 \qquad -\Delta\mathcal V''=2a\vartheta(1+\vartheta).
\]
For \(0\le c\le1\),
\[
 (c-a\vartheta)^2+2a\vartheta(1+\vartheta)
 =c^2+2a(1-c)\vartheta+a(a+2)\vartheta^2>c^2.
\]
For \(c\ge1\), the same expression equals
\[
 a(a+2)\left(\vartheta-\frac{c-1}{a+2}\right)^2
 +1+\frac{2(c-1)(c+a+1)}{a+2}>1,
\]
with strictness also at \(c=1\) because \(\vartheta>0\).
\end{proof}

At a hypothetical companion collision satisfying
\[
 -\min\{\Omega^2,1\}\le\chi,
\]
elimination of \(\mathcal V'\) in
\eqref{eq:global-gauge-identities} gives
\[
 0=\mathcal V^2-\Delta\mathcal V''+\chi
   +(\mathcal V+m+1)\frac{\mathfrak d^2}{\Delta}.
\]
The first three terms have strictly positive sum.  Hence
\begin{equation}
\label{eq:global-negative-potential}
 \mathcal V<-(m+1),\qquad \mathfrak d\ne0.
\end{equation}

Define
\[
 \zeta=k-\frac14-mx,
 \qquad
 Y=\frac{\zeta^2-1/16}{\Delta}.
\]
The endpoint weights complete to a square:
\begin{equation}
\label{eq:global-square-completion}
 \mathcal V=m+\Omega-1-Y,
 \qquad
 \mathfrak d=-2\zeta-(1-2x).
\end{equation}
By \eqref{eq:global-negative-potential},
\[
 Y>2m+\Omega,\qquad |\zeta|>\frac14.
\]
We may therefore define $\eta=(1-2x)/(2\zeta)$; it satisfies $|\eta|<2$.
The equation \(\mathfrak w=0\) is equivalent to
\begin{equation}
\label{eq:global-balance-equations}
 Y=m(3+2\eta)+(1+\eta)\Omega,
 \qquad
 \zeta^2=\frac{4Y+1}{16(1+Y\eta^2)}.
\end{equation}

\begin{lemma}[Feasible imbalance regions]
\label{lem:global-feasible-regions}
At a companion collision, either
\[
 0\le\eta<\frac45,
 \qquad
 \eta<\frac34\quad\hbox{when }2\le m\le3,
\]
or \(\eta=-\tau<0\), where the negative branch is empty for \(m=2\), and
for \(m>2\)
\[
 0<\tau<\tau_m:=\frac{m-2}{4m}<\frac14,
 \qquad
 0\le\Omega\le\Omega_{\max}(\tau,m),
\]
with
\[
 \Omega_{\max}(\tau,m)
 =\frac{m\{(8m+2)\tau^2+(1-11m)\tau+2m-4\}}
 {(1-\tau)((4m+1)\tau+2)}.
\]
\end{lemma}
\begin{proof}
Suppose \(\eta\ge3/4\).  Then \(m\eta>1\) and
\[
 \zeta(1-m\eta)=k-\frac14-\frac m2.
\]
If \(k=0\), this forces \(\zeta>0\), contrary to
\(\zeta=-1/4-mx<0\).  If \(\ell=0\), it forces \(\zeta<0\); writing
\(v=1-x>0\) then gives
\(\eta=(1-2v)/(1/2-2mv)>2\), also impossible.  Thus both exponents are at
least one, and
\[
 |\zeta|\le\frac{2m-3}{4(m\eta-1)}.
\]
Together with \eqref{eq:global-balance-equations}, this implies
\[
 (4Y+1)(m\eta-1)^2-(2m-3)^2(1+Y\eta^2)\le0.
\]
The coefficient of \(Y\) is positive for \(m\ge2\), \(\eta\ge3/4\).
Substituting \(Y\ge m(3+2\eta)\) gives the lower bound
\[
 F_m(\eta)
 =m^2(24\eta^3+21\eta^2-24\eta-4)
 +m(-18\eta^3-27\eta^2+6\eta+24)-8.
\]
This is increasing on \([3/4,\infty)\) for every \(m\ge2\): if the two
cubic coefficients are denoted by \(a\) and \(b\), then
\(a'>0\) and \(2a'+b'>0\) there.  Moreover
\[
 F_m(4/5)=\frac{316m^2+288m-1000}{125}>0\quad(m\ge3),
\]
while
\[
 F_m(3/4)=\frac{-2m^2+183m-256}{32}>0
 \quad(2\le m\le3).
\]
This proves the positive-branch bounds.

Now write \(\eta=-\tau<0\).  The negative-potential condition gives
\(0<\tau<m/(2m+\Omega)\le1/2\).  The endpoint gap yields
\[
 |\zeta|\le\frac{2m+1}{4(1+m\tau)}.
\]
Substitution in \eqref{eq:global-balance-equations} gives
\[
 (2-\tau)\left[(1-\tau)((4m+1)\tau+2)\Omega
 -m\{(8m+2)\tau^2+(1-11m)\tau+2m-4\}\right]\le0.
\]
The numerator polynomial in braces is strictly decreasing on
\([0,1/2]\), and at \(\tau_m=(m-2)/(4m)\) it equals
\[
 -\frac{(m-2)(m+2)(2m+1)}{8m^2}.
\]
Thus the negative branch is empty at \(m=2\), and for \(m>2\) it has the
stated range.
\end{proof}

On \(\mathfrak w=0\), direct substitution gives
\begin{equation}
\label{eq:global-critical-value}
 (4-\eta^2)\mathfrak u
 =\Phi_m(\eta,\Omega)-(4-\eta^2)(5\Omega-\chi),
\end{equation}
where
\[
 \Phi_m(\eta,\Omega)
 =\phi_2(\eta)\Omega^2+\phi_1(m,\eta)\Omega+\phi_0(m,\eta),
\]
\[
 \phi_2=15\eta^4+16\eta^3-12\eta^2-16\eta,
\]
\[
 \phi_1=m(60\eta^4+92\eta^3-32\eta^2-80\eta-16)
          -7\eta^2+12\eta+28,
\]
\[
 \phi_0=m^2(60\eta^4+120\eta^3-6\eta^2-96\eta-24)
          +m(-6\eta^2+24\eta+36)+4-\eta^2.
\]

\begin{theorem}[Endpoint-binomial companion-collision exclusion]
\label{thm:two-unit-charge-gap}
Assume
\[
 m\ge2,
 \quad k,\ell\in\{0\}\cup[1,\infty),
 \quad \Omega\ge0,
 \quad -\min\{\Omega^2,1\}\le\chi\le5\Omega.
\]
Then \(\mathcal T_{\Omega,\chi}g\) and
\(\mathcal T_{\Omega,\chi}(xg)\) have no common zero in \((0,1)\).
\end{theorem}
\begin{proof}
On \(0\le\eta\le4/5\),
\[
 \phi_2(\eta)\le-\frac{192}{25}\eta\le0,
 \qquad
 \phi_1(m,\eta)
 \le-16m+28+(12-16m)\eta-7\eta^2<0.
\]
The first inequality follows from
\[
 \frac{\phi_2(\eta)}\eta+\frac{192}{25}
 =\frac{(5\eta-4)(75\eta^2+140\eta+52)}{25},
\]
and the second from
\[
 15\eta^3+23\eta^2-8\eta-16
 =(5\eta-4)(3\eta^2+7\eta+4).
\]
The derivative of \(\phi_0(m,\eta)\) has, in descending powers of \(\eta\),
two positive and then two negative coefficients.  Hence it has at most one
positive zero; since it is negative at zero, any such zero is a change from
negative to positive.  Thus \(\phi_0\) has no positive local maximum.  The
needed endpoint values are
\[
 \phi_0(m,0)=-24m^2+36m+4<0,
\]
\[
 \phi_0(m,3/4)=\frac{-1905m^2+3240m+220}{64}<0
 \quad(m\ge2),
\]
\[
 \phi_0(m,4/5)=\frac{-2328m^2+6420m+420}{125}<0
 \quad(m\ge3).
\]
Using Lemma~\ref{lem:global-feasible-regions}, this proves
\(\Phi_m<0\) on the entire nonnegative branch.

On the negative branch write \(\eta=-\tau\).  For
\(0<\tau<1/4\), \(\phi_2(-\tau)>0\), so
\(\Phi_m(-\tau,\Omega)\) is convex in \(\Omega\).  At \(\Omega=0\),
\[
 \partial_\tau\Phi_m(-\tau,0)
 \ge\frac{141}{2}m^2-27m-\frac12>0,
\]
and
\[
 \Phi_m(-\tau_m,0)
 =-\frac{129m^4+480m^3+60m^2-400m-224}{64m^2}<0.
\]
Thus the first endpoint is negative throughout the feasible interval.  At the
second endpoint,
\[
 \Phi_m(-\tau,\Omega_{\max})
 =\frac{(4-\tau^2)P_m(\tau)}
 {(1-\tau)^2((4m+1)\tau+2)^2},
\]
where the quartic $P_m$ has the following exact Bernstein representation.
Set $v=m-2\ge0$ and $\tau_m=(m-2)/(4m)$. Direct expansion gives
\begin{equation}\label{eq:negative-branch-bernstein}
 -256m^4P_m(\tau_m z)
 =\sum_{i=0}^4 b_i(v)\binom4i z^i(1-z)^{4-i},
\end{equation}
with
\begin{align*}
 b_0(v)&=1024(v+2)^4(4v^3+15v^2+17v+5),\\
 b_1(v)&=32(v+2)^3(8v^5+124v^4+594v^3+1229v^2+1108v+320),\\
 b_2(v)&=\tfrac83(v+2)^2\bigl(108v^6+1468v^5+8051v^4+22861v^3\\
       &\hspace{8em}+35153v^2+27072v+7680\bigr),\\
 b_3(v)&=(v+2)\bigl(204v^7+3156v^6+20887v^5+76493v^4\\
       &\hspace{5em}+166500v^3+213392v^2+146944v+40960\bigr),\\
 b_4(v)&=71v^8+1550v^7+14213v^6+71736v^5+218427v^4\\
       &\hspace{5em}+411616v^3+470080v^2+299008v+81920.
\end{align*}
Every $b_i(v)$ is strictly positive for $v\ge0$. The Bernstein basis is
nonnegative and sums to $1$ on $[0,1]$, so $P_m(\tau)<0$ throughout
$[0,\tau_m]$ for $m>2$. At $m=2$ the interval is $\{0\}$ and
$P_2(0)=-20$. Thus one certificate covers every real $m\ge2$.

Both endpoint values of the convex function
\(\Phi_m(-\tau,\Omega)\) are negative.  Therefore \(\Phi_m<0\) on the
negative branch as well.  Since \(|\eta|<2\) and \(\chi\le5\Omega\),
\eqref{eq:global-critical-value} gives \(\mathfrak u<0\), contradicting a
companion collision.
\end{proof}

\subsection{Quartet blocks and paired real factors}
\label{sec:global-block-preservers}

\begin{theorem}[Quartet block preservation]
\label{thm:quartet-block-preserver}
For every integer \(N\ge3\), \(0\le A\le A_{\rm ext}\), and \(B\ge0\),
\[
 \mathcal Q_{A,B}\mathfrak E_N\subseteq\mathfrak E_N.
\]
\end{theorem}
\begin{proof}
Use the simultaneous path
\[
 \mathcal T_\tau=\mathcal Q_{A_\tau,B_\tau},
 \qquad A_\tau=\tau A,
 \qquad B_\tau=\tau B,
\]
so
\[
 \Omega_\tau=\frac54+B_\tau-A_\tau,
 \qquad \chi_\tau=4A_\tau B_\tau.
\]
Then \(\Omega_\tau>0\), \(\chi_\tau\ge0\), and
\[
 5\Omega_\tau-\chi_\tau
 =(\tfrac54-A_\tau)(5+4B_\tau)>0.
\]
The constant differential coefficient is
\[
 (\tfrac14+B_\tau-A_\tau)^2+4A_\tau B_\tau\ge0,
\]
and the remaining scalar condition is
\[
 \begin{aligned}
 3\Omega_\tau^2+10\Omega_\tau-1-\chi_\tau
 ={}&3B_\tau^2+(\tfrac{35}{2}-10A_\tau)B_\tau\\
    &+3A_\tau^2-\tfrac{35}{2}A_\tau+\tfrac{259}{16}\ge0.
 \end{aligned}
\]
Indeed $0\le A_\tau\le A_{\rm ext}<7/6$, and $A_{\rm ext}$ is the
smaller root of the last quadratic.

The starting operator is \(\mathcal L^2\), which preserves
\(\mathfrak E_N\) by two applications of
Corollary~\ref{cor:charge-symmetry-slice} with \(\theta=0\).  Theorem
\ref{thm:two-unit-charge-gap} excludes every endpoint-binomial companion
collision because its total exponent is \(N-1\ge2\).  Theorem
\ref{thm:connected-path-preserver} completes the proof.
\end{proof}

The preceding theorem supplies square blocks as starting points for the
paired-factor paths. For comparison with the single-factor threshold $\Theta_m$ of Corollary~\ref{cor:charge-symmetry-slice}, recall
\[
 \frac34<\Theta_2=\frac{35}{44}<1<\Theta_3=\frac{21}{20}
 <\frac87<A_{\rm ext}<\frac76<\frac54.
\]

\begin{theorem}[Paired real factors]
\label{thm:paired-real-preserver}
If \(\theta_1,\theta_2\in[3/4,A_{\rm ext}]\), then for every integer
\(N\ge3\),
\[
 (\mathcal L-\theta_1)(\mathcal L-\theta_2)\mathfrak E_N
 \subseteq\mathfrak E_N.
\]
\end{theorem}
\begin{proof}
Put $c=(\theta_1+\theta_2)/2$, $\delta=(\theta_1-\theta_2)/2$, and use
\[
 \mathcal T_\tau=(\mathcal L-c)^2-\tau\delta^2\Id
 =\mathcal T_{\Omega,-\tau\delta^2},\qquad
 \Omega=\tfrac54-c\in(0,1/2],\quad 0\le\tau\le1.
\]
The initial operator $\mathcal Q_{c,0}$ preserves $\mathfrak E_N$ by
Theorem~\ref{thm:quartet-block-preserver}. Since
\[
 \Omega^2-\delta^2
 =(\tfrac54-\theta_1)(\tfrac54-\theta_2)>0,
\]
the parameters satisfy $-\min\{\Omega^2,1\}< -\tau\delta^2\le5\Omega$.
The other scalar conditions follow from
\[
 (\Omega-1)^2-\tau\delta^2
 \ge(\tfrac14-\theta_1)(\tfrac14-\theta_2)>0
\]
and
\[
 3\Omega^2+10\Omega-1+\tau\delta^2
 =3c^2-\tfrac{35}{2}c+\tfrac{259}{16}+\tau\delta^2\ge0,
\]
because $3/4\le c\le A_{\rm ext}$. Theorem~\ref{thm:two-unit-charge-gap}
excludes the companion collisions. Theorem~\ref{thm:connected-path-preserver}
therefore gives preservation at $\tau=1$, the desired product.
\end{proof}

The paired-factor range extends beyond the single-factor threshold without
requiring either factor to preserve the class separately.

\section{Endpoint completion and source assembly}
\label{sec:endpoint-sturm}
We assemble the commuting source factors into the required endpoint
pencils. Real factors that need not preserve the whole class individually
are paired; any unpaired factor acts first on the original endpoint seed.
The remaining blocks then preserve the whole class.

\subsection{All finite-source rows \texorpdfstring{$n\ge2$}{n>=2}}
\label{sec:global-source-orbit}

\begin{theorem}[All finite sources in rows \(n\ge2\)]
\label{thm:all-sources-two-and-higher}
\label{thm:all-sources-row-three}
Let \(F(u)=H(u^2)\) be a nonzero real even polynomial with positive leading
coefficient and
\[
 \Zset(F)\subseteq\{u:|\Re u|\le \sqrt{A_{\rm ext}}\}.
\]
Then, for every integer \(n\ge2\) and every real \(t\),
\[
 H(\mathcal L)[x^n(x+t)]\in\mathfrak E_{n+1}.
\]
\end{theorem}
\begin{proof}
Factor \(H\) over \(\R\) into real spectral factors and conjugate nonreal
pairs.  A real source zero produces a nonnegative spectral root; a purely
imaginary source zero produces a negative spectral root; and an orbit
\(\{\pm(a+\mathrm{i}b),\pm(a-\mathrm{i}b)\}\) produces the quartet block with
\(A=a^2\le A_{\rm ext}\) and \(B=b^2>0\).

Every quartet block preserves \(\mathfrak E_{n+1}\) by
Theorem~\ref{thm:quartet-block-preserver}.  Every real factor with
\(\theta\le3/4\) preserves it individually, because
\(\Theta_n\ge\Theta_2>3/4\).  Pair all remaining real factors, whose roots
lie in \((3/4,A_{\rm ext}]\); each pair preserves the whole class by
Theorem~\ref{thm:paired-real-preserver}.

For every real $\theta\le A_{\rm ext}$, a single factor preserves the
original seed. Indeed, put $\alpha=1/4-\theta$. If $\alpha\ge0$, use
Lemma~\ref{lem:positive-charge-shift}; otherwise
$\alpha>-1>-\mathfrak e_{n,1/2}$, so
Lemma~\ref{lem:single-outer-endpoint-pencil} applies after $y=4x$ and
$t\mapsto4t$. The top multiplier $(n+3/2)^2-\theta>0$ retains exact degree
and positive leading coefficient, giving
$(\mathcal L-\theta)[x^n(x+t)]\in\mathfrak E_{n+1}$ for every real $t$.
If one unpaired factor remains, commute it to the right to act first on
this seed. All remaining blocks preserve the class.  This proves the claim, with no bound on the source degree, number of
factors, or multiplicities.
\end{proof}

Here the odd endpoint is $\nu=1/2$. If $G$ is the original odd source,
its even spectral source is $\widetilde G=G/u$; the theorem below is stated
for a generic even spectral source $F$.

\begin{theorem}[The sharp odd uniform endpoint strip]
\label{thm:odd-uniform-preservation}
Let $F\not\equiv0$ be a real even polynomial or an entire function of order
at most one, with all zeros in $S_{h_*}$. Then
\[
 \JacMult_{F,1/2}[y^n(y+t)]\in\mathcal E_{n+1}
 \qquad(n\ge0,\ t\in\mathbb R).
\]
For a fixed $n$, the image has exact degree $n+1$ and nonzero leading
coefficient. The half-width $h_*$ is optimal uniformly over all rows.
\end{theorem}
\begin{proof}
First let $F$ be a polynomial, normalized to have positive leading coefficient.
The case $n=0$ is of degree one, with leading sample $F(3/2)>0$.
The case $n=1$ is Theorem~\ref{thm:exact-first-jacobi-pair} with $\nu=1/2$.
For $n\ge2$, use Theorem~\ref{thm:all-sources-two-and-higher}, since
$h_*^2<A_{\rm ext}$. The pullback $y=4x$ gives the statement on $[0,4]$.

For an entire source, normalize its constant sign on $(h_*,\infty)$ and take
the positive-leading canonical truncations in
Remark~\ref{rem:positive-truncations}. The finite-sample continuity from
Section~\ref{sec:sampling-framework} gives coefficientwise convergence
of each fixed output. The degree-$(n+1)$ leading coefficient is $F(n+3/2)>0$,
which prevents a degree drop. Closedness of $\mathcal E_{n+1}$ retains
real-rootedness and at most one exterior zero. Restore the harmless nonzero
real scalar normalization.
The polynomial first-row witness in Theorem~\ref{thm:exact-first-jacobi-pair}
gives failure for every larger strip and proves optimality.
\end{proof}

\subsection{The general Jacobi endpoint-pencil range}
\label{sec:general-endpoint-assembly}
We prove the general theorem using width-one TN blocks, separately from
the odd enlargement.

\begin{proof}[Proof of Theorem~\ref{thm:jacobi-family-endpoint-pencil}]
For polynomial $F$, factor into the real, imaginary and quartet orbits of
Section~\ref{sec:zero-orbit-factors}. Nonnegative real spectral shifts
preserve every defect class by Lemma~\ref{lem:positive-charge-shift},
as do quartets with $|a|\le\mathfrak h_{n,\nu}\le1$ by
the consequence of Theorem~\ref{thm:sharp-quartet-TN} stated in
Section~\ref{sec:zero-orbit-factors}. Pair the negative shifts, all
of which lie in $[-3/4,0)$, using Lemma~\ref{lem:square-outer-real-pairing}.
If one remains, its parameter satisfies
\[
 \alpha_a=\tfrac14-a^2\ge-\min\{\mathfrak e_{n,\nu},3/4\}
                       \ge-\mathfrak e_{n,\nu}.
\]
It acts first on $y^n(y+t)$ by Lemma~\ref{lem:single-outer-endpoint-pencil}.
Commutation lets that factor act first, followed by the preserving pairs
and the other blocks. A nonzero real scalar is harmless for $\mathcal E_{n+1}$.

For entire $F$, the truncations of Lemma~\ref{lem:derivative-strip} retain
the strip, and finite-sample continuity gives convergence of each output. Closedness of $\mathcal E_{n+1}$ gives the
conclusion, including zero outputs and degree drops.

If $\mathfrak e_{n,\nu}\le3/4$, choose $a$ just above
$\mathfrak h_{n,\nu}$ and take $F_a(u)=u^2-a^2$. Then
$\alpha_a<-\mathfrak e_{n,\nu}$ is arbitrarily close to that boundary,
and the sharpness part of Lemma~\ref{lem:single-outer-endpoint-pencil}
gives a nonreal pair. This proves (i), including the cap boundary.
Finally $\mathfrak e_{n,\nu}$ increases with $n$, and
$\mathfrak e_{1,\nu}\le3/4$ is equivalent to $\nu\ge10/11$.
The first row therefore gives both the uniform bound and its stated
sharpness; $n=0$ compares a constant with a polynomial of degree at most one.
\end{proof}

\subsection{Quotient interlacing and interval support}
\label{subsec:odd-endpoint}
The complete pencil, not merely real-rootedness of its two generators,
also determines interval support.

\begin{lemma}[Support from a complete one-defect pencil]
\label{lem:pencil-interval-support}
Let $I$ be a compact interval and let $U,V\in\R[y]$ have exact degrees $n,n+1$.
If every $V+tU$, $t\in\R$, is real-rooted with at most one zero outside
$I$, counted with multiplicity, then every zero of $U$ lies in $I$.
\end{lemma}
\begin{proof}
The assertion is vacuous for $n=0$. As $t\to+\infty$,
$t^{-1}(V+tU)\to U$ coefficientwise. Rouch\'e's theorem on small circles
about the zeros of $U$ counts $n$ nearby zeros with their multiplicities;
the remaining zero escapes every fixed disk. All these zeros are real.
A nonreal zero of $U$ would contradict their persistence, and a real zero
outside $I$ would persist alongside the escaping zero, giving two exterior
zeros. Both are impossible.
\end{proof}

\begin{proof}[Proof of Theorem~\ref{thm:main-sturm}]
\emph{Even case.} Specialize Theorem~\ref{thm:jacobi-family-endpoint-pencil}
to $\nu=3/2$. Its uniform bound is
$\mathfrak h_{1,3/2}=\sqrt{15/28}$ and is sharp. The identity
$C_{F,n}^-(y-2)=\JacMult_{F,3/2}(y^n)$ transfers the pencil relation and
its sharpness to the quotients.

\emph{Odd case.} Apply the multiplier to the even source $G/u$, extended
at zero. By Theorem~\ref{thm:odd-uniform-preservation} and the raising identity,
\[
 (C_{G,n+1}^++tC_{G,n}^+)(y-2)
 =(n+\tfrac32)\JacMult_{G/u,1/2}
 \left[y^n\left(y+\frac{n+1/2}{n+3/2}t\right)\right].
\]
The parameter on the right ranges over $\R$, so every nonzero member is
real-rooted. The first-row witness in
Theorem~\ref{thm:exact-first-jacobi-pair}, whose width at $\nu=1/2$ is
$h_*$, gives the odd sharpness witness after multiplication by $u$. In both parity cases $n=0$ compares a constant and a
polynomial of degree at most one.

We also record support, for use below. For $n\ge1$, both even spectral
sources, $F$ and $G/u$, have nonzero leading samples in degrees $n$ and
$n+1$: the even width is below $1$, and $h_*<3/2$. At the respective Jacobi
parameters, each complete pencil has at most one zero outside $[0,4]$.
Lemma~\ref{lem:pencil-interval-support} therefore puts its degree-$n$ image
in $[0,4]$. Applying this at every row, and translating
by $y=X+2$, proves that every nonzero quotient lies in $[-2,2]$; the
zeroth quotient is constant.
\end{proof}

\begin{corollary}[Scaled endpoint interlacing and support]
\label{cor:scaled-endpoint-quotients}
Let $h\ge0$, $\delta>0$, and let $F,G$ be nonzero even and odd real entire
functions, respectively, of order at most one with zeros in $S_h$.
For $F_\delta(v)=F(\delta v)$ and $G_\delta(v)=G(\delta v)$, set
\begin{equation}\label{eq:scaled-C}
 C_{F,n}^{-,(\delta)}:=C_{F_\delta,n}^-,\qquad
 C_{G,n}^{+,(\delta)}:=C_{G_\delta,n}^+.
\end{equation}
For all $n\ge0$, the even quotients weakly interlace if
$h\le\delta\sqrt{15/28}$, and the odd quotients weakly interlace if
$h\le\delta h_*$. In each case every nonzero quotient has all its zeros
in $[-2,2]$.
\end{corollary}
\begin{proof}
Apply Theorem~\ref{thm:main-sturm} and its support conclusion to the
rescaled sources, whose zeros lie in $S_{h/\delta}$.
\end{proof}

\begin{corollary}[Fixed-$n$ strip range]
\label{cor:fixed-index-horizontal}
Let \(n\ge1\), and let \(F\not\equiv0\) be an even real entire function of
order at most one whose zeros lie in
\[
 \lvert\Re u\rvert\le \mathfrak h_{n,3/2}
 =\min\left\{1,\sqrt{\frac14+\frac{n(n+1)}{4n+3}}\right\}.
\]  Then the quotients with indices \(n\) and \(n+1\) satisfy
\[
 C_{F,n}^{-}\preceq C_{F,n+1}^{-},
\]
and both nonzero quotients have all zeros in \([-2,2]\).  Explicitly,
\[
 \mathfrak h_{1,3/2}=\sqrt{\frac{15}{28}},\qquad
 \mathfrak h_{2,3/2}=\sqrt{\frac{35}{44}},\qquad
 \mathfrak h_{n,3/2}=1\quad(n\ge3).
\]
The constants \(\mathfrak h_{1,3/2}\) and
\(\mathfrak h_{2,3/2}\) are optimal for $n=1$ and $n=2$, respectively.  More generally, if the zeros lie in \(\lvert\Re u\rvert\le h\), then at
sampling scale \(\delta>0\) the corresponding sufficient condition is
\(h\le\delta \mathfrak h_{n,3/2}\).
No unit-spacing half-width valid for every such even source can exceed
$\sqrt{1/4+n(n+1)/(4n+3)}$. The lower and upper bounds agree for $n=1,2$;
for $n\ge3$ they are $1$ and this quantity, respectively, giving
$1$ and $\sqrt{21/20}$ at $n=3$.
\end{corollary}

\begin{proof}
Theorem~\ref{thm:jacobi-family-endpoint-pencil} at $\nu=3/2$ gives the
weak relation and sharpness for $n=1,2$. Since
$\mathfrak h_{n,3/2}\le\mathfrak h_{n+1,3/2}\le1$, its pencils at rows
$n$ and $n+1$ are both available. Their leading samples are nonzero, so
Lemma~\ref{lem:pencil-interval-support} gives support for both quotients.
Apply the same argument to $F_\delta(v)=F(\delta v)$ for the scaled assertion.
For the upper bound, the quadratic source $F_a(u)=u^2-a^2$ and the
sharpness calculation in Lemma~\ref{lem:single-outer-endpoint-pencil}
give failure in every strip wider than the displayed upper bound.
\end{proof}

\begin{remark}[The first odd quotient at width one]
\label{rem:first-odd-quotient-location}
For a nonzero odd real entire function $G$ of order at most one with zeros
in $S_1$, put $\widetilde G=G/u$, extended at zero. Using
$X_{\widetilde G}$ from Section~\ref{sec:exact-first-pencil},
\[
 C_{G,1}^+(y-2)=\tfrac32\widetilde G(3/2)
 (y-1+X_{\widetilde G}),\qquad
 X_{\widetilde G}\in[-3/5,1].
\]
Thus its unique zero lies in $[-2,-2/5]$. The bound on
$X_{\widetilde G}$ follows from
\eqref{eq:general-first-pair-real-atom} for quadratic orbits,
\[
 E_{a,b}(3/2)-E_{a,b}(1/2)=5+4b^2-4a^2>0
\]
for quartets, multiplicative closure of $[-3/5,1]$, and the canonical
truncations of Lemma~\ref{lem:derivative-strip}.
\end{remark}

\subsection{Separation from fixed-degree unit-circle support}
\begin{proposition}[Unit-circle-rooted samples need not have weakly interlacing quotients]
\label{prop:circle-vs-interlacing}
For \(a>0\), put \(F_a(u)=u^2-a^2\), and use the samples
$B_d[F_a]$ from \eqref{eq:centered-B-notation}.
If
\[
 \sqrt{\frac{15}{28}}<a\le\frac{\sqrt3}{2},
\]
then both \(B_3[F_a]\) and \(B_5[F_a]\) have all zeros on \(\T\), but
\[
 C_{F_a,1}^{-}\not\preceq C_{F_a,2}^{-}.
\]
More precisely, for some \(t\in\R\), the pencil member
\(C_{F_a,2}^{-}+tC_{F_a,1}^{-}\) has a nonreal conjugate pair.  Hence individual
unit-circle-rootedness of two adjacent samples does not determine whether their
quotients weakly interlace.
\end{proposition}

\begin{proof}
Applying $(zD_z-d/2)^2-a^2$ to $(1+z)^d$ gives
\[
 B_d[F_a](z)=(1+z)^{d-2}\bigl[A_d(z^2+1)+D_dz\bigr],
 \quad A_d=d^2/4-a^2,\quad D_d=d-d^2/2-2a^2.
\]
For $d=3,5$ and $0<a^2\le3/4$, $A_d>0$ and
$-D_d/A_d=2-(d-4a^2)/A_d\in(-2,2]$. Thus the remaining reciprocal
quadratic, and hence each sample, has all its zeros on $\T$.

For noninterlacing set $\alpha=1/4-a^2\in[-1/2,-2/7)$ and use
Lemma~\ref{lem:single-outer-endpoint-pencil} at $(n,\nu)=(1,3/2)$.
Its two leading coefficients remain positive. With
$r_1=6/(2+\alpha)$ and $r_2=20/(6+\alpha)$, we have
\[
 (2+\alpha)(6+\alpha)(r_2-r_1)=4+14\alpha<0.
\]
Thus its normalized pencil $y(y-r_2)+s(y-r_1)$ is nonreal-rooted at
$s=r_2-2r_1$, exactly as in that lemma's proof. The quotient dictionary
transfers this failure to $C_{F_a,2}^-+tC_{F_a,1}^-$.
\end{proof}

\Needspace{12\baselineskip}
\noindent\textit{Proved ranges and optimality.}
Here uniformity is in the quotient index $n$.
\begin{center}
\begin{tabular}{@{}l l p{.46\textwidth}@{}}
\hline
Regime & Half-width & Exactness \\
\hline
Even / odd & $\sqrt{15/28}$ / $h_*$ & Both uniform weak thresholds are exact.\\
Nonpolynomial & $\sqrt{15/28}$ / $h_*$ & Simple interior zeros and sharp uniform
 strict interlacing; see Theorem~\ref{thm:remote-orbit-strictification}
 and Corollary~\ref{cor:sharp-source-strictness}.\\
First row & $\mathfrak h_{1,\nu}^{\rm opt}$ & Exact for all $0<\nu<2$;
 extremal family changes at $\nu_0$.\\
Fixed $n\ge2$ & $\mathfrak h_{n,\nu}$ & Exact if $\mathfrak e_{n,\nu}\le3/4$;
 sufficient in the capped regime.\\
All rows & $\mathfrak h_{1,\nu}$ & Exact for $\nu\ge10/11$; the odd slice
 improves to $h_*$.\\
\hline
\end{tabular}
\end{center}
At $\nu=1/2$, the first-row value is
$\mathfrak h_{1,1/2}^{\mathrm{opt}}=h_*$, whereas the general bound is
$\mathfrak h_{1,1/2}=1$. Theorem~\ref{thm:odd-uniform-preservation}
extends the sharper value $h_*$ to every row. Higher rows also admit
$\sqrt{A_{\rm ext}}$, an auxiliary sufficient width not claimed optimal.
For $0<\nu<10/11$ with $\nu\ne1/2$, $\mathfrak h_{1,\nu}$ is a sufficient uniform width.
The table summarizes the proved strip ranges and their optimality,
including both fixed-index and uniform statements.
Theorem~\ref{thm:eventual-finite-exception} provides a complementary regime:
finitely many arbitrary exceptional zero orbits are permitted, but the
conclusion begins only after a cutoff determined by their finite spectral
factor and the Jacobi parameter, independently of the remaining positive
spectral factors.

\section{Source-side strictness and eventual interlacing}
\label{sec:strict-consecutive-degrees}

Weak endpoint-pencil preservation permits multiple roots, endpoint roots
and common roots of adjacent images. We first exclude these degeneracies
for nonpolynomial sources under the strip conditions, then prove eventual
strict interlacing in the presence of finitely many arbitrary exceptional
zero orbits. Two-sided deformations of source zeros give the Jacobi
direction through a remote-orbit limit in
Theorem~\ref{thm:remote-orbit-strictification} and finite interpolation
for the infinite-product case of Theorem~\ref{thm:eventual-finite-exception}.
A positive Gaussian factor supplies the same direction directly.

Unlike ordinary Nuij-type differential deformations
\cite{NuijHyperbolic,KurdykaPaunescuNuij}, the variations are made in the
source while preserving the hypotheses of the relevant weak theorem.
For the strip result, that input is
Theorem~\ref{thm:jacobi-family-endpoint-pencil}, or
Theorem~\ref{thm:odd-uniform-preservation} at the sharp odd width $h_*$.
The eventual weak input is established separately in
Section~\ref{subsec:eventual-finite-exception}.

\subsection{Why finite total positivity is not enough}

\begin{example}[A totally positive Jacobi block need not force strict interlacing]
\label{ex:finite-TP-not-horizontal-strict}
A finite matrix is \emph{totally positive} (TP) if every minor is strictly
positive. For an even real-valued $F$, write
\[
 \mathbf M_{N,\nu}[F]=\Sigma_N[\JacMult_{F,\nu}]_N\Sigma_N
\]
for its degree-$N$ checkerboard Bernstein matrix. Let
\[
 p(\lambda)=\frac{23\lambda^2-24\lambda+66}{110},\qquad
 F(u)=p(u^2-1/4).
\]
Then $F(1/2)=3/5$, $F(3/2)=1$, and $F(5/2)=75/11$.
By Lemma~\ref{lem:bernstein-jacobi-matrix},
$\mathbf M_{2,3/2}[F]=p(\mathbf J_{2,3/2}(0))$ has positive entries,
least $2\times2$ minor $27/220$, and determinant $45/11$, hence is TP.
Nevertheless,
\[
 C_{F,1}^-(y-2)=\frac{5y-6}{5},\qquad
 C_{F,2}^-(y-2)=\frac{(5y-6)(15y-46)}{11},
\]
so the two quotients share $y=6/5$. Since every source zero $\rho$ satisfies
$(\Re\rho)^2=(71+5\sqrt{1081})/184>1$, this does not contradict
Theorem~\ref{thm:remote-orbit-strictification}.
\end{example}

\subsection{Inward Jacobi transversality}
\label{subsec:remote-orbit-strictification}

On each fixed space $\R[y]_{\le M}$, identify a polynomial with its
monomial coefficient vector, padded by zeros. Differentiation and convergence
of paths are coefficientwise; $R_t=o(t)$ means $R_t/t\to0$ in that space.
We first record the Jacobi direction at a simple common zero, and then
the stronger tangent conditions needed at multiple and endpoint roots.

For a coefficientwise $C^1$ path $r_s$ of fixed exact degree, the
implicit-function theorem gives a local branch $\eta(s)$ through each
simple real root at $s_0$, with $\dot r_s=\partial_s r_s$, $r_s'=\partial_y r_s$,
and
\begin{equation}\label{eq:local-root-velocity}
 \eta'(s_0)=-\dot r_{s_0}(\eta(s_0))/r'_{s_0}(\eta(s_0)).
\end{equation}
Adding a scalar multiple of the polynomial to its variation does not
change its root velocities.

\begin{lemma}[Strict inward Jacobi transversality]
\label{lem:jacobi-inward-transversality}
Fix $0<\nu<2$. Let $p,q$ have exact degrees $n,n+1$, simple roots in
$(0,4)$, and weakly interlacing zero lists
$\beta_0\le\alpha_1\le\beta_1\le\cdots\le\alpha_n\le\beta_n$.
For coefficientwise $C^1$ paths in the respective degree spaces with
\[
 p_t=p+t\JacOp_\nu p+o(t),\qquad q_t=q+t\JacOp_\nu q+o(t),
\]
each collapsed oriented gap has strictly positive derivative at $t=0$.
The oriented gap is $\beta_j-\alpha_j$ at $\xi=\alpha_j=\beta_j$ and
$\alpha_j-\beta_{j-1}$ at $\xi=\alpha_j=\beta_{j-1}$.
\end{lemma}

\begin{proof}
For a polynomial $r$ with simple zero $\xi$,
$r''(\xi)/(2r'(\xi))=\sum_{\zeta\ne\xi}(\xi-\zeta)^{-1}$.
Thus, when $\xi=\alpha_j=\beta_j$,
\begin{equation}\label{eq:inward-root-sum}
 \frac12\left(\frac{q''(\xi)}{q'(\xi)}-\frac{p''(\xi)}{p'(\xi)}\right)
 =\frac1{\xi-\beta_0}
  +\sum_{\substack{1\le i\le n\\ i\ne j}}
   \left(\frac1{\xi-\beta_i}-\frac1{\xi-\alpha_i}\right)>0.
\end{equation}
Each bracket is nonnegative by interlacing, and the unpaired term is
positive. Formula~\eqref{eq:local-root-velocity} gives root velocity
$-\xi(\xi-4)r''(\xi)/r'(\xi)-(2\xi-4\nu)$.
The drift cancels between the two roots, and $-\xi(\xi-4)>0$, so the right
collapsed gap opens strictly. The left-gap case follows from
$(y,\nu)\mapsto(4-y,2-\nu)$ and
$\JacOp_\nu[f(4-y)]=(\JacOp_{2-\nu}f)(4-y)$.
\end{proof}

\begin{lemma}[Two-sided tangent conditions]
\label{lem:source-two-sided-jets}
Let $p_s=p+s v+O(s^2)$ be a coefficientwise smooth family of real
polynomials of a fixed exact degree, real-rooted for all sufficiently
small positive and negative $s$. If $p$ has a root of multiplicity $e\ge2$
at $\xi$, then
\[
 v^{(j)}(\xi)=0\qquad(0\le j\le e-2).
\]
If all roots of $p_s$ lie in $[0,4]$ and $p$ has a root of multiplicity
$e\ge1$ at $b\in\{0,4\}$, then also $v^{(e-1)}(b)=0$.
\end{lemma}

\begin{proof}
For $0\le j\le e-2$, Rolle's theorem makes $D_y^j p_s$ real-rooted and
$D_y^j p$ has multiplicity at least two at $\xi$.
Lemma~\ref{lem:two-sided-multiple-root-tangent} applied to this family gives
$v^{(j)}(\xi)=0$.
For an endpoint root, $D_y^{e-1}p$ has a simple root at $b$.
The implicit-function theorem gives its local root branch $\rho(s)$;
no simplicity of the other roots is needed. Rolle's theorem keeps all
roots of $D_y^{e-1}p_s$ in $[0,4]$. Thus $\rho'(0)=0$ at either endpoint,
and the root-velocity formula gives
$0=-v^{(e-1)}(b)/p^{(e)}(b)$.
\end{proof}

\subsection{Remote zero orbits and projective generators}
\label{subsec:remote-zero-orbit-deformations}

By the product representation in Lemma~\ref{lem:derivative-strip}, a
nonpolynomial even source of order at most one has infinitely many zero
orbits. In a bounded vertical strip their imaginary parts are unbounded.
Thus imaginary pairs or reflected quartets with $b\to\infty$ occur;
we call these \emph{remote orbits}.

Write $B=b^2$ and $\lambda=u^2-1/4$. The corresponding unnormalized factors
are, respectively,
\[
 h_B(\lambda)=\lambda+B+\tfrac14\quad(r=1),\qquad
 h_B(\lambda)=(\lambda+B+\tfrac14-a^2)^2+4a^2B\quad(r=2).
\]
For a source with zeros in a fixed strip $S_H$, a quartet has $0<a\le H$;
its real part is fixed during the deformation. The case $a=0$ is a
repeated imaginary pair, handled one copy at a time. Move only
$B$ to $B/(1+s)$ for real $s$ near $0$. Both signs of $s$ are allowed, and
the real parts of all zeros stay fixed. The logarithmic generator is
\[
 G_B(\lambda)=-B\,\partial_B\log h_B(\lambda).
\]
No logarithm branch is needed here: this denotes $-B(\partial_Bh_B)/h_B$.
Both $h_B$ are positive on the nonnegative Jacobi spectrum.
For a pair $G_B=-B/(B+\lambda+1/4)$. For a quartet,
\[
 h_B=B^2+2B(\lambda+a^2+\tfrac14)
                +(\lambda+\tfrac14-a^2)^2.
\]
Expanding the rational logarithmic derivative gives, uniformly for bounded
$\lambda$ and $|a|\le H$, for every fixed $H<\infty$, as $B\to\infty$,
\begin{equation}\label{eq:remote-projective-generator}
 G_B(\lambda)-G_B(0)=\frac{r\lambda}{B}+O(B^{-2}),\qquad
 \frac{B}{r}\bigl(G_B(\JacOp_\nu)-G_B(0)\Id\bigr)\longrightarrow\JacOp_\nu.
\end{equation}
The operator limit is on each fixed $\R[y]_{\le M}$: its Jacobi basis is
fixed, and convergence at the finite spectrum $j(j+1)$, $0\le j\le M$,
is operator-norm convergence. No estimate on an unbounded spectrum is used.

For each selected orbit of an even source $E$, remove one copy of its
factor and set
\[
 E_s(u)=\frac{E(u)}{h_B(u^2-1/4)}h_{B/(1+s)}(u^2-1/4).
\]
The quotient is entire. The source remains real and even, has order at
most one, and has zeros in the same strip for both signs of small $s$.
For any fixed $V\in\R[y]_{\le M}$, finite spectral calculus gives
\[
 \left.\partial_s\JacMult_{E_s,\nu}V\right|_{s=0}
   =G_B(\JacOp_\nu)\JacMult_{E,\nu}V.
\]
Only $h_B(j(j+1))$ is inverted; these values are positive. No inverse of
$\JacMult_{E,\nu}$ is required.

\subsection{Proof of source-side strictness}
\begin{proof}[Proof of Theorem~\ref{thm:remote-orbit-strictification}]
Put $P_d=\JacMult_{E,\nu}(y^d)$. In the general range, monotonicity of
$\mathfrak h_{d,\nu}$ makes the complete pencils at rows $n$ and $n+1$
available from Theorem~\ref{thm:jacobi-family-endpoint-pencil}. In the
sharp odd range they follow from Theorem~\ref{thm:odd-uniform-preservation}.
The same assertions hold for every source deformation $E_s$ above.
All leading samples in degrees $n,n+1,n+2$ are nonzero: the source strip
has half-width at most $1$ in the first case and $h_*<3/2$ in the second,
whereas $n+1/2\ge3/2$.
Lemma~\ref{lem:pencil-interval-support} at these two rows therefore puts the roots
of $P_n,P_{n+1}$, and of their deformations, in $[0,4]$. Their exact degrees
are $n,n+1$, and the pair weakly interlaces.

Suppose one of these images $P$ has a root of multiplicity $e\ge2$ at
$\xi\in(0,4)$. Lemma~\ref{lem:source-two-sided-jets} gives
\[
 \bigl(G_B(\JacOp_\nu)P\bigr)^{(e-2)}(\xi)=0
\]
for every remote orbit. Since $P^{(e-2)}(\xi)=0$, we can subtract the
scalar part, multiply by $B/r$, and use
\eqref{eq:remote-projective-generator}. It follows that
$(\JacOp_\nu P)^{(e-2)}(\xi)=0$.
Writing $P=(y-\xi)^e g$, with $g(\xi)\ne0$, gives instead
\[
 (\JacOp_\nu P)^{(e-2)}(\xi)
    =e!\,\xi(\xi-4)g(\xi)\ne0.
\]
Thus neither image has an interior multiple root.

If $P$ has an endpoint root of multiplicity $e\ge1$, the endpoint part
of Lemma~\ref{lem:source-two-sided-jets} and the same operator limit force
the corresponding derivative of order $e-1$ of $\JacOp_\nu P$ to vanish.
But, for the indicated factorizations,
\begin{equation}\label{eq:strict-source-endpoint-jets}
 \begin{aligned}
 P=y^e g &: & (\JacOp_\nu P)^{(e-1)}(0)
    &=-4e!(e-1+\nu)g(0)\ne0,\\
 P=(y-4)^e g &: & (\JacOp_\nu P)^{(e-1)}(4)
    &=4e!(e+1-\nu)g(4)\ne0.
 \end{aligned}
\end{equation}
Here $0<\nu<2$ is essential. Hence all roots of both images are simple
and interior.

Finally, if $P_n,P_{n+1}$ shared a root $\xi$, choose
$t_*=-P_{n+1}'(\xi)/P_n'(\xi)$. Then
$R=P_{n+1}+t_*P_n$ has a multiple root at $\xi\in(0,4)$.
The family
$R_s=\JacMult_{E_s,\nu}[y^n(y+t_*)]$ is real-rooted for both signs of $s$
by the same weak pencil theorem, with exact degree $n+1$.
Apply the interior multiple-root argument to $R_s$ to obtain a contradiction.
This step requires only real-rootedness of the pencil members, not that
all their roots lie in $[0,4]$. The pair has no common root and therefore
strictly interlaces.
\end{proof}

\subsection{Sharp uniform strictness and polynomial sources}

\begin{corollary}[Sharp uniform strict interlacing]
\label{cor:sharp-source-strictness}
Let $F$ and $G$ be nonpolynomial even and odd real entire functions,
respectively, of order at most one. If
$\Zset(F)\subseteq S_{\sqrt{15/28}}$, then for every $n\ge1$ the quotient
$C^-_{F,n}$ has $n$ simple zeros in $(-2,2)$ and strictly interlaces
$C^-_{F,n+1}$. The same conclusions hold for $C^+_{G,n}$ and
$C^+_{G,n+1}$ when $\Zset(G)\subseteq S_{h_*}$.
Both uniform half-widths are optimal even within the nonpolynomial classes.
\end{corollary}

\begin{proof}
Apply Theorem~\ref{thm:remote-orbit-strictification} to $F$ at $\nu=3/2$
and to the even entire source $G/u$ at $\nu=1/2$.
If $G/u$ were polynomial, then $G$ would be polynomial. The quotient
dictionary and its positive raising factors transfer the conclusions.
For optimality, take a polynomial sharpness witness $H_0$ from
Theorem~\ref{thm:main-sturm} in any larger strip and put
$H_\epsilon(u)=H_0(u)\cosh(\epsilon u)$, $\epsilon>0$.
This preserves parity and the strip, adds only imaginary zeros, and gives
a nonpolynomial real entire source of order one. Local uniform convergence
$H_\epsilon\to H_0$ gives coefficientwise convergence of each fixed sampled
pencil. A nonreal root of the witness pencil persists on a small disk
disjoint from the real axis by Rouch\'e's theorem. Thus weak, and hence
strict, interlacing fails for some such nonpolynomial source in every
larger strip.
\end{proof}

\begin{proposition}[Endpoint multiplicity for polynomial sources]
\label{prop:polynomial-endpoint-multiplicity}
Let $0<\nu<2$ and $E(u)=p(u^2-1/4)$, where the real polynomial $p\ne0$ has degree $K$.
For every $d>K$,
\[
 \operatorname{ord}_{y=0}\bigl(\JacMult_{E,\nu}(y^d)\bigr)=d-K.
\]
Consequently, for a nonzero even real entire source of order at most one
with zeros in $S_{\mathfrak h_{1,\nu}}$, all positive-index images have
simple interior zeros and all adjacent pairs strictly interlace if and
only if the source is nonpolynomial. The same equivalence holds at
$\nu=1/2$ under the larger condition $\Zset(E)\subseteq S_{h_*}$.
\end{proposition}

\begin{proof}
The dictionary gives $\JacMult_{E,\nu}=p(\JacOp_\nu)$, and
\[
 \JacOp_\nu y^j=j(j+1)y^j-4j(j-1+\nu)y^{j-1}.
\]
Each application lowers the least exponent by at most one. If $p_K$ is
the leading coefficient of $p$, the coefficient of $y^{d-K}$ is
\[
 p_K(-4)^K\prod_{j=0}^{K-1}(d-j)(d-j-1+\nu)\ne0,
\]
with empty product $1$ when $K=0$; lower powers of $\JacOp_\nu$ cannot
contribute to this coefficient. This proves the exact multiplicity.
Polynomial sources therefore have endpoint roots in every row $d>K$ and
common endpoint roots in adjacent such rows. Conversely, the strip and
Theorem~\ref{thm:remote-orbit-strictification} give the asserted properties
for every nonpolynomial source, using the monotonicity of
$\mathfrak h_{d,\nu}$ in the general range.
\end{proof}

\begin{remark}[Source hypotheses and the zeroth row]
Nonpolynomiality and the order bound supply infinitely many zero orbits,
and the bounded strip makes their imaginary parts unbounded. Simplicity
and interiority are conclusions, not additional output hypotheses.
The polynomial obstruction above concerns all rows simultaneously; a
fixed pair of rows may be strict for a polynomial source. We restrict
the strict conclusions to positive indices, since $E(1/2)$, the zeroth
image, can vanish. No positive gap bound uniform over all admissible
sources is asserted: as $\epsilon\downarrow0$, the admissible sources
$E(u)=\cosh(\epsilon u)$ give $\JacMult_{E,\nu}(y^d)\to y^d$.
\end{remark}

\subsection{A finite imaginary-orbit criterion}

\begin{corollary}[Finite imaginary-pair criterion]
\label{cor:finite-imaginary-strictness}
Fix $0<\nu<2$ and $n\ge1$. Let $E\ne0$ be an even real polynomial or an
even real entire function of order at most one. Assume either
$\Zset(E)\subseteq S_{\mathfrak h_{n,\nu}}$, or $\nu=1/2$ and
$\Zset(E)\subseteq S_{h_*}$.
If $E$ has $n+1$ distinct imaginary zero pairs $\{\pm\mathrm i b_j\}$,
$b_j>0$, then $\JacMult_{E,\nu}(y^n)$ and
$\JacMult_{E,\nu}(y^{n+1})$ have exact degrees $n,n+1$, respectively,
have simple zeros in $(0,4)$, and strictly interlace. Repetitions of one pair do not count as distinct pairs.
\end{corollary}

\begin{proof}
Set $M=n+1$, $B_j=b_j^2$ and $s_j=B_j+1/4$ for $1\le j\le M$, and put
$\lambda_i=i(i+1)$ for $0\le i\le M$.
Moving the $j$th pair alone gives an admissible two-sided source path and
\[
 G_{B_j}(\lambda)-G_{B_j}(0)=\frac{B_j\lambda}{s_j(s_j+\lambda)}.
\]
On the nonzero spectrum, the generator-value matrix factors as
\[
 \left(\frac{B_j\lambda_i}{s_j(s_j+\lambda_i)}\right)_{i,j=1}^M
 =\diag(\lambda_1,\ldots,\lambda_M)
   \left(\frac1{\lambda_i+s_j}\right)_{i,j=1}^M
   \diag\left(\frac{B_1}{s_1},\ldots,\frac{B_M}{s_M}\right).
\]
Invertibility of the diagonal factors and of the Cauchy matrix yields real
coefficients $c_j$ such that
\[
 \JacOp_\nu=\sum_{j=1}^M c_j
     \bigl(G_{B_j}(\JacOp_\nu)-G_{B_j}(0)\Id\bigr)
     \quad\text{on }\R[y]_{\le M};
\]
this is interpolation at $\lambda_1,\ldots,\lambda_M$, with both sides
vanishing at $\lambda_0=0$. The two weak pencils give support and exact
degree. Applying the linear tangent conditions of
Lemma~\ref{lem:source-two-sided-jets} and repeating the three contradictions
from the proof of Theorem~\ref{thm:remote-orbit-strictification} excludes
interior multiple roots, endpoint roots, and common adjacent roots.
\end{proof}

\subsection{Eventual interlacing with positive spectral factors}
\label{subsec:eventual-finite-exception}

The finite exceptional part is handled by a high-degree Jacobi--Laguerre
limit. We prove this step for arbitrary positive endpoint parameters,
then attach positive spectral factors without changing its cutoff. The
finite-core corollary below makes a sufficient cutoff effective; the final
remark records the sign distinction for Gaussian factors.

Work in the normalized coordinate $x$ with
$J=J_{\kappa_0,\kappa_1}$ from Section~\ref{sec:real-spectral-factors},
where $\kappa_0,\kappa_1>0$, and put $\sigma=\kappa_0+\kappa_1$.
In this subsection write $\lambda_j=j(j+\sigma-1)$ and
\[
 \Pi_k(\lambda)=\prod_{r=0}^{k-1}(\lambda-\lambda_r),
 \qquad \Pi_0=1.
\]
The original slice is $(\kappa_0,\kappa_1)=(\nu,2-\nu)$, with $y=4x$.
We use $\mathrm P_k^{(\alpha,\beta)}$ for the classically normalized Jacobi
polynomial, so that
$\mathrm P_k^{(\alpha,\beta)}(1)=\binom{k+\alpha}{k}$.

\begin{lemma}[Newton--Jacobi identity]
\label{lem:newton-jacobi-core}
For integers $d\ge k\ge0$,
\begin{equation}\label{eq:newton-jacobi-core}
 \Pi_k(J)x^d
 =(d)_{\underline{k}}\,k!\,x^{d-k}
 \mathrm P_k^{(\kappa_1-1,\,d-k+\kappa_0-1)}(2x-1),
\end{equation}
where $(d)_{\underline{k}}=d(d-1)\cdots(d-k+1)$.
\end{lemma}

\begin{proof}
For $0\le k\le d$, put
\[
 F_{d,k}(x)=x^{d-k}
 \mathrm P_k^{(\kappa_1-1,\,d-k+\kappa_0-1)}(2x-1).
\]
The classical binomial expansion
\[
 \mathrm P_k^{(\alpha,\beta)}(2x-1)
 =\sum_{j=0}^k
   \binom{k+\alpha}{k-j}\binom{k+\beta}{j}
   (x-1)^j x^{k-j}
\]
gives, for $0\le k<d$, by direct differentiation
\[
 \bigl(J-k(k+\sigma-1)\Id\bigr)F_{d,k}
 =(d-k)(k+1)F_{d,k+1}.
\]
Since $F_{d,0}=x^d$, induction on $k$ proves
\eqref{eq:newton-jacobi-core}. The displayed normalization also separates
these classical Jacobi polynomials from the monic spectral basis
$\mathsf P_r^{(\nu)}$ used elsewhere.
\end{proof}

\begin{proposition}[Eventual preservation by a finite spectral core]
\label{prop:eventual-finite-core}
Fix $\kappa_0,\kappa_1>0$ and $0\ne p\in\R[\lambda]$.
There exists $d_0=d_0(p,\kappa_0,\kappa_1)\ge\max\{1,\deg p\}$ such that,
for every integer $d\ge d_0$ and every $t\in\R$,
\[
 p(J)\bigl[x^d(x+t)\bigr]
\]
has exact degree $d+1$, is real-rooted and has at most one zero outside
$[0,1]$, counted with multiplicity. Moreover, if $K=\deg p\ge1$, then
\[
 p(J)x^d=x^{d-K}R_d(x),\qquad \deg R_d=K,
\]
and, for all sufficiently large $d$, the zeros of $R_d$ satisfy
\begin{equation}\label{eq:eventual-residual-order}
 0<\rho_{1,d}<\rho_{1,d+1}<\rho_{2,d}<\rho_{2,d+1}
 <\cdots<\rho_{K,d}<\rho_{K,d+1}<1.
\end{equation}
In particular, $p(J)x^d$ has exact degree $d$ and all its zeros in $[0,1]$.
\end{proposition}

\begin{proof}
The case $K=0$ is immediate, so assume $K\ge1$. Expand $p$ in the Newton
basis
\[
 p(\lambda)=\sum_{k=0}^K a_k\Pi_k(\lambda),\qquad a_K=p_K,
\]
where $p_K$ is the leading coefficient. Lemma~\ref{lem:newton-jacobi-core}
gives
\begin{equation}\label{eq:finite-core-expansion}
 R_d(x)=\sum_{k=0}^K a_k(d)_{\underline{k}}k!\,
 x^{K-k}\mathrm P_k^{(\kappa_1-1,\,d-k+\kappa_0-1)}(2x-1).
\end{equation}
The coefficient of $x^K$ is $p(\lambda_d)$, so $R_d$ has exact degree $K$
for all sufficiently large $d$.

The classical Jacobi--Laguerre limit \cite[(18.7.21)]{NISTHandbook}
and \eqref{eq:finite-core-expansion} give, coefficientwise,
\begin{equation}\label{eq:finite-core-laguerre-limit}
 \frac{R_d(1-z/d)}{p_Kd^K}\longrightarrow K!L(z),
 \qquad L(z)=\mathrm L_K^{(\kappa_1-1)}(z).
\end{equation}
For each fixed $k$, the scaled Jacobi polynomial is coefficientwise
bounded, so the $k<K$ terms vanish after multiplication by
$(d)_{\underline{k}}/d^K$; the $k=K$ term has the displayed limit.
Since $\kappa_1-1>-1$, $L$ has $K$ simple positive zeros. Thus, for all
large $d$, $R_d$ has $K$ simple roots in $(0,1)$, each satisfying
$d(1-\xi_d)\to\ell$ for a zero $\ell$ of $L$.

To compare consecutive rows, commutation with $p(J)$ gives
\[
 (\lambda_d\Id-J)p(J)x^d
 =d(d+\kappa_0-1)p(J)x^{d-1}.
\]
At a root $\xi$ of $R_d$, substituting $x^{d-K}R_d(x)$ yields
\[
 \begin{aligned}
 d(d+\kappa_0-1)\frac{R_{d-1}(\xi)}{\xi R_d'(\xi)}
 &=\xi(1-\xi)\frac{R_d''(\xi)}{R_d'(\xi)}\\
 &\quad+2d-2K+\kappa_0-(2d-2K+\sigma)\xi.
 \end{aligned}
\]
Set
\[
 Q_d(z)=\frac{R_d(1-z/d)}{p_Kd^K},\qquad z_d=d(1-\xi_d).
\]
Then $R_d''(\xi_d)/R_d'(\xi_d)=-dQ_d''(z_d)/Q_d'(z_d)$.
Substituting into the full right-hand side of the lowering relation,
coefficientwise convergence in \eqref{eq:finite-core-laguerre-limit}
and convergence of derivatives give, along $z_d\to\ell$, the limit
\[
 -\ell\frac{L''(\ell)}{L'(\ell)}-\kappa_1+2\ell=\ell>0,
\]
by $zL''+(\kappa_1-z)L'+KL=0$ \cite[\S18.8]{NISTHandbook}.
Hence $R_{d-1}(\xi)/R_d'(\xi)>0$ at every root of $R_d$ for all large $d$.
Their leading coefficients have the same eventual sign. Alternating
signs at these roots place $K-1$ roots of $R_{d-1}$ in the intervening
intervals, and the sign at $-\infty$ places the remaining root to the
left. All these roots already lie in $(0,1)$; replacing $d$ by $d+1$
gives \eqref{eq:eventual-residual-order}.

After nonzero rescaling and a reparametrization of $t$, make $R_d$ and
$xR_{d+1}$ monic. By \eqref{eq:eventual-residual-order}, the signs of
$xR_{d+1}+tR_d$ at the roots of $R_d$ and at $\pm\infty$ give one real
root in each complementary interval. Of these, $K-1$ lie between
consecutive roots of $R_d$. If $t\le0$, the left outer root lies in
$[0,\rho_{1,d})$; if $t\ge0$, the right outer root lies in
$(\rho_{K,d},\rho_{K,d+1}]$. Thus at least $K$ roots lie in $[0,1]$.
Multiplication by $x^{d-K}$ proves the pencil assertion. The residual
formula also proves the final assertion for individual images.
\end{proof}

\begin{corollary}[An effective finite-core cutoff]
\label{cor:effective-finite-core-cutoff}
Let $K=\deg p\ge1$ and extend $R_d$ by
\eqref{eq:finite-core-expansion} to a real parameter $d$. The expression
\[
 \mathcal C_p(d)=\operatorname{lc}_x(R_d)R_d(0)R_d(1)
 \Disc_x(R_d)\Resultant_x(R_d,R_{d+1})
\]
is a nonzero polynomial in $d$. Here $\operatorname{lc}_x(R_d)$ means
the coefficient $p(d(d+\sigma-1))$ of $x^K$, and the discriminant of
a linear polynomial is $1$. Any integer $d_0\ge\max\{1,K\}$ strictly larger than
every real zero of $\mathcal C_p$ is a sufficient cutoff in
Proposition~\ref{prop:eventual-finite-core} and in the positive-factor
extension below. For real-algebraic coefficients of $p$ and real-algebraic
endpoint parameters, such a cutoff is computable in exact arithmetic.
It need not be minimal. For $K=0$, take $d_0=1$.
\end{corollary}

\begin{proof}
The finite binomial formula in Lemma~\ref{lem:newton-jacobi-core} shows
that every coefficient of $R_d$ is polynomial in $d$. The eventual
simplicity, interiority and strict ordering in
Proposition~\ref{prop:eventual-finite-core} show that none of the five
factors defining $\mathcal C_p$ is identically zero. Beyond its largest
real zero, no degree drop, endpoint crossing or multiple root is
possible. Continuity of polynomial roots therefore propagates the
simple interior configuration from a sufficiently large integer to the
whole connected real ray. This also applies at $d+1$. The nonvanishing
resultant prevents crossings between the two root lists, so
\eqref{eq:eventual-residual-order} persists on that ray. The pencil
argument of the proposition now applies at every integer $d\ge d_0$.
If $\mathcal C_p$ has no real zero, use the entire real line for this
continuation. Real-root isolation, or a polynomial root bound, makes
the choice effective for real-algebraic input data. The extension below
chooses its cutoff before adjoining the positive factors.
\end{proof}

\Needspace{6\baselineskip}
\begin{theorem}[Eventual interlacing with positive spectral factors]
\label{thm:general-positive-factor-eventual}
Fix $\kappa_0,\kappa_1>0$, $0\ne p\in\R[\lambda]$, and let
\[
 \Phi(\lambda)=C e^{a\lambda}\prod_{j\in I}
       \left(1+\frac{\lambda}{s_j}\right),
 \qquad C\in\R\setminus\{0\},\quad a\ge0,\quad s_j>0.
\]
The index set $I$ is finite or countably infinite, with
$\sum_{j\in I}s_j^{-1}<\infty$ in the infinite case. Define
$P_n=\Phi(J)p(J)x^n$ by the values of $\Phi$ on the finite Jacobi spectrum.
There is $n_0=n_0(p,\kappa_0,\kappa_1)$, independent of $\Phi$, such that
for every $n\ge n_0$, the polynomials $P_n,P_{n+1}$ have exact degrees
$n,n+1$, have all their zeros in $[0,1]$, and weakly interlace. Their
pencil has at most one exterior zero, counted with multiplicity.
If $a>0$ or $I$ is infinite, the zeros are simple and interior and the
interlacing is strict. If $a=0$ and $I$ is finite, the endpoint $0$ has
exact multiplicity $n-\deg(p\Phi)$ in $P_n$ for $n>\deg(p\Phi)$.
\end{theorem}

\begin{proof}
Choose $n_0$ as in Proposition~\ref{prop:eventual-finite-core}, or by
Corollary~\ref{cor:effective-finite-core-cutoff}. Each $J+s_j\Id$
preserves every interval-defect class by
Lemma~\ref{lem:positive-charge-shift}. So does $e^{aJ}$, since on each
finite-dimensional degree space
\[
 e^{aJ}=\lim_{M\to\infty}(\Id+aJ/M)^M\qquad(a>0),
\]
and each factor is a positive multiple of a positive Jacobi shift;
$a=0$ gives the identity. The infinite product converges on the finite
spectrum. The closedness argument of Section~\ref{subsec:defect-closure},
after scaling the interval, therefore gives preservation by $\Phi(J)$.
Each $\Phi(\lambda_j)$ is nonzero. Applying it to the individual
finite-core images and their pencils proves the exact-degree, interval
support and weak-interlacing assertions, with no $\Phi$-dependence of
$n_0$.

Fix $n\ge n_0$. If $a>0$, varying $a$ in both directions inside
$(0,\infty)$ gives the exact tangent generator $J$. If $I$ is infinite,
summability supplies $n+1$ distinct parameters $s_j$. Replacing one
$s_j$ by $s_j/(1+\epsilon)$ gives an admissible two-sided deformation
with generator $g_j(J)$, where
\[
 g_j(\lambda)=\frac{\lambda}{s_j+\lambda}.
\]
The matrix $(g_j(\lambda_i))_{i,j=1}^{n+1}$ is a nonsingular positive
diagonal matrix times the Cauchy matrix
$(1/(\lambda_i+s_j))_{i,j=1}^{n+1}$. Its two node lists are distinct and
positive. Interpolation therefore expresses $J$ as a real linear
combination of these generators on $\R[x]_{\le n+1}$; both sides vanish
on the constant eigenvector.

Every deformation retains the weak conclusions and the same cutoff.
Lemma~\ref{lem:source-two-sided-jets}, after scaling, forces the
variation derivatives through order $e-2$ to vanish at an interior
root of multiplicity $e\ge2$. If $P=(x-\xi)^e h$, however,
\[
 (JP)^{(e-2)}(\xi)=e!\xi(\xi-1)h(\xi)\ne0.
\]
At an endpoint of an interval-supported image the derivative of order
$e-1$ of every admissible variation must also vanish. The corresponding
Jacobi derivatives are
\[
 (JP)^{(e-1)}(b)=
 \begin{cases}
 -e!(e-1+\kappa_0)h(0),&b=0,\\
 e!(e-1+\kappa_1)h(1),&b=1,
 \end{cases}
 \qquad P=(x-b)^e h,
\]
and are nonzero. Linearity of the vanishing conditions and the generator
span exclude all these roots. If adjacent images shared an interior
root, a fixed real linear combination could cancel the first derivative
there, producing a nonzero pencil member with a multiple interior root.
The same two-sided real-rootedness argument excludes it. Thus the
interlacing is strict.

Finally, when $a=0$ and $I$ is finite, $q=p\Phi$ is a nonzero polynomial.
If $L=\deg q$ and $n>L$, the least-degree term of $q(J)x^n$ has
coefficient
\[
 q_L(-1)^L(n)_{\underline L}(n+\kappa_0-1)_{\underline L}\ne0
\]
at $x^{n-L}$. Each application of $J$ lowers the least exponent by at
most one, so lower powers of $J$ cannot contribute to that term. This
proves the asserted multiplicity.
\end{proof}

The theorem concerns endpoint pencils, not eventual preservation of the
whole exact-degree one-exterior-root class by an arbitrary $p(J)$.

\begin{proof}[Proof of Theorem~\ref{thm:eventual-finite-exception}]
Take $(\kappa_0,\kappa_1)=(\nu,2-\nu)$ and $y=4x$. On the Jacobi
spectrum the nonexceptional factor has the form
\[
 \Phi(\lambda)=e^{a/4}H(1/2)e^{a\lambda}(1+4\lambda)^r
       \prod_{j\in I}\left(1+\frac{\lambda}{b_j^2+1/4}\right).
\]
The prefactor is nonzero, and the reciprocals of the positive shifts
are summable in the infinite case. Moreover,
$\JacMult_{E,\nu}=\Phi(\JacOp_\nu)p(\JacOp_\nu)$.
Theorem~\ref{thm:general-positive-factor-eventual} therefore gives the
claimed strictness, with $n_0(p,\nu)$ independent of $a$ and $H$.

For the final assertion, write $E(u)=e(u^2)$. The real entire function
$e$ has growth order less than one and only finitely many zeros off
$(-\infty,0]$. Its genus-zero Hadamard product
\cite[Chapter~I]{LevinEntire}, with those finitely many zeros separated,
gives $E(u)=p(u^2-1/4)H(u)$ with $H$ of the stated product form and
$\sum_j b_j^{-2}<\infty$. Nonpolynomiality forces infinitely many
imaginary pairs, so the factorized case with $a=0$ applies.
\end{proof}

\medskip
\noindent\textit{A source outside the uniform strips.}
The source
\[
 E(u)=(u^4+16)\cosh u
\]
has four zeros off the imaginary axis, all with $|\Re u|=\sqrt2>h_*$. It
therefore lies in neither sharp uniform strip of
Theorem~\ref{thm:main-sturm}, while Theorem~\ref{thm:eventual-finite-exception}
gives simple interior zeros and strict interlacing for all sufficiently
high adjacent images. This example illustrates the finite-exception regime;
no minimal cutoff for this particular source is asserted.

\medskip
\noindent\textit{The sign of the Gaussian factor.}
The restriction to a nonnegative Gaussian exponent is essential for an
order-two extension based only on zero location. Indeed, let
$E_a(u)=e^{-au^2}$ with $a>0$. Spectral calculus gives
$\JacMult_{E_a,\nu}=e^{-a/4}e^{-a\JacOp_\nu}$. In the degree-$n$ Bernstein
basis, Lemma~\ref{lem:bernstein-jacobi-matrix} shows that
$-[\JacOp_\nu]_n$ has strictly positive adjacent off-diagonal entries.
After adding a sufficiently large scalar multiple of the identity, its
tridiagonal graph is irreducible and nonnegative, so the exponential is
entrywise strictly positive. Applied to the Bernstein coefficient vector
of $y^n$, it gives a polynomial strictly positive on $[0,4]$, of exact
degree $n$. Thus $E_a$ is even, real entire, nonpolynomial and zero-free of
order two, but none of its positive-index images has a zero in the interval.
This obstructs an unrestricted replacement of growth order less than two
by growth order at most two; it does not assert failure of weak
interlacing itself.
\section{Arithmetic applications}
\label{sec:arithmetic-applications}

Horizontal comparison varies the quotient index $n$ at fixed derivative
order $m$; the source-side theorem supplies both simplicity and strictness.
Vertical comparison varies $m$ at fixed $n$; it uses the fixed-degree theorem
of \cite{JinSharpSampling}, the nondegeneracy checks below, and parity-correct
descent. Both comparisons then take place on the same real interval. We
present the Dedekind and newform applications first, then record additional
signed and common-root consequences of horizontal interlacing.

Recall that $B_{d,\delta}[F]$ is the centered binomial sample from the
Introduction, while $C_{F,n}^{-}$ and $C_{G,n}^{+}$ are the reduced endpoint
quotients from Section~\ref{sec:sampling-framework}.  For $\zeta\in\T$ and
$\delta>0$, define the centered two-point operator
\begin{equation}\label{eq:centered-difference-operator}
 \mathscr T_{\zeta,\delta}\varphi(u)
 =\varphi(u-\delta/2)+\zeta\,\varphi(u+\delta/2).
\end{equation}
The superscript in $\mathscr T_{\zeta,\delta}^{\,\ell}$ denotes $\ell$-fold
composition, with the zeroth iterate equal to the identity.  This is the
nondegeneracy operator in the fixed-degree derivative theorem.
Derivatives of completed $L$-functions are taken with respect to their
complex argument $s$, and derivatives of centered sources with respect to
$u$. On a quotient family, the superscript $(m)$ records the derivative
order of its source.

\subsection{Parity-correct descent from the unit circle}
\label{sec:derivative-sturm}

For consecutive derivatives, parity alternates: remove $z+1$ from the
reciprocal sample and $z-1$ from the anti-reciprocal sample before passing
to $X=z+z^{-1}$.

Let \(F\) be a real entire function of fixed parity
\[
 F(-u)=\epsilon F(u),\qquad \epsilon\in\{\pm1\},
\]
and write \(F_m=F^{(m)}\).  Thus
\[
 F_m(-u)=\epsilon(-1)^mF_m(u).
\]
For \(n\ge0\), using the Chebyshev convention $U_{-1}=0$ from
Section~\ref{sec:sampling-framework}, define
\begin{equation}
\label{eq:parity-correct-Q}
 \mathscr Q_{F,n}^{(m)}(X)=
 \sum_{r=0}^{n}\binom{2n+1}{n-r}
 F_m\!\left(r+\frac12\right)
 \bigl(U_r(X/2)-\epsilon(-1)^mU_{r-1}(X/2)\bigr).
\end{equation}
Rows for which \(\epsilon(-1)^m=1\) use \(U_r-U_{r-1}\), while rows for
which \(\epsilon(-1)^m=-1\) use \(U_r+U_{r-1}\).  We reserve
\(\mathscr Q\) for this parity-correct derivative quotient family.  The
symbols \(C_{F,n}^{-,(\delta)}\) and \(C_{G,n}^{+,(\delta)}\) retain their
meanings for the scaled quotients of fixed even and odd sources, respectively.
In particular,
\begin{equation}\label{eq:two-Q-families}
 \mathscr Q_{F,n}^{(m)}
 =
 \begin{cases}
 C_{F^{(m)},n}^{-},&\epsilon(-1)^m=1,\\[1mm]
 C_{F^{(m)},n}^{+},&\epsilon(-1)^m=-1.
 \end{cases}
\end{equation}
Hence every derivative row uses the endpoint dictated by its actual parity.

For $m,n\ge0$, put $P_{m,n}=B_{2n+1}[F_m]$. Applying the quotient
formulas from Section~\ref{sec:sampling-framework} to $F_m$ gives the
Laurent identity
\begin{equation}
\label{eq:parity-correct-factorization}
 P_{m,n}(z)=z^n\left(z+\epsilon(-1)^m\right)
 \mathscr Q_{F,n}^{(m)}(z+z^{-1}).
\end{equation}

Two degree-$d$ multisets on $\T$ \emph{strictly cyclically interlace} if
each is simple, the two multisets are disjoint, and their points alternate in
cyclic order.  Two coprime nonzero real polynomials of degrees differing by at
most one form a \emph{strict Obreschkoff pair} if their zeros strictly
interlace and every nonzero real linear combination has only simple real
zeros in its actual degree.

\begin{proposition}[Descent of strict cyclic interlacing]
\label{prop:shifted-descent}
Retain the preceding $F$, $\epsilon$, and $P_{m,n}$.  Let \(m\ge0\) and
\(n\ge1\).  Assume that \(P_{m,n}\) and \(P_{m+1,n}\) have \(2n+1\) simple zeros on
\(\T\), and that their zeros strictly cyclically interlace.  Let \(m_{\mathrm E},m_{\mathrm O}\in\{m,m+1\}\) be determined by \(\epsilon(-1)^{m_{\mathrm E}}=+1\) and \(\epsilon(-1)^{m_{\mathrm O}}=-1\), respectively.  Thus $F_{m_{\mathrm E}}$ is even and $F_{m_{\mathrm O}}$ is odd; the
subscripts refer to the parity of the derivative function, not necessarily to
the parity of the integers when $\epsilon=-1$.  Then
\(\mathscr Q_{F,n}^{(m_{\mathrm E})}\) and \(\mathscr Q_{F,n}^{(m_{\mathrm O})}\) have degree
\(n\) and have \(n\) simple zeros in \((-2,2)\).  List these zeros in decreasing
order as
\[
 X_1^{\mathrm E}>\cdots>X_n^{\mathrm E},\qquad
 X_1^{\mathrm O}>\cdots>X_n^{\mathrm O}.
\]
Then
\begin{equation}
\label{eq:shifted-interlacing-order}
 2>X_1^{\mathrm E}>X_1^{\mathrm O}>
 X_2^{\mathrm E}>X_2^{\mathrm O}>
 \cdots>
 X_n^{\mathrm E}>X_n^{\mathrm O}>-2.
\end{equation}
Equivalently, the augmented sets
\begin{equation}
\label{eq:augmented-interlacing}
 \{2\}\cup \Zset(\mathscr Q_{F,n}^{(m_{\mathrm O})})
 \quad\text{and}\quad
 \Zset(\mathscr Q_{F,n}^{(m_{\mathrm E})})\cup\{-2\}
\end{equation}
strictly interlace on \([-2,2]\).  In particular,
\(\mathscr Q_{F,n}^{(m_{\mathrm E})}\) and
\(\mathscr Q_{F,n}^{(m_{\mathrm O})}\) form a strict Obreschkoff pair in
the actual-degree convention: every nonzero real linear combination has only
real zeros, simple in its actual degree.
\end{proposition}

\begin{proof}
Identity~\eqref{eq:parity-correct-factorization} removes $-1$ from the even
sample and $1$ from the odd one. The identities
$X-2=(z-1)^2/z$ and $X+2=(z+1)^2/z$ exclude endpoint quotient zeros
by sample simplicity. The other $2n$ unit-circle zeros form reciprocal
pairs, giving exact degree $n$ and simple quotient zeros in $(-2,2)$.
On the upper semicircle the even angles $\theta_j$ and odd angles
$\phi_j$, with their forced endpoints included, satisfy
\[
 0=\phi_0<\theta_1<\phi_1<\cdots<\theta_n<\phi_n<\pi=\theta_{n+1}.
\]
The decreasing map $\theta\mapsto2\cos\theta$ proves
\eqref{eq:shifted-interlacing-order}--\eqref{eq:augmented-interlacing}.
Obreschkoff's theorem \cite{Obreschkoff1963} gives real-rooted pencil members.
Their roots are simple: a multiple root would make the nonvanishing
Wronskian of the strictly interlacing quotient pair vanish. This also
applies after leading-coefficient cancellation, in the actual degree.
\end{proof}

\subsection{Dedekind zeta derivatives}

We now apply the two comparisons to completed Dedekind zeta derivatives.  Let
$K$ be a number field of absolute discriminant $D_K$ and signature
$(r_1,r_2)$, where $r_1$ is the number of real embeddings and $r_2$ the
number of conjugate pairs of complex embeddings. Write $\zeta_K$ for its
Dedekind zeta function, and set
\[
 \Gamma_{\R}(s)=\pi^{-s/2}\Gamma(s/2),\qquad
 \Gamma_{\C}(s)=2(2\pi)^{-s}\Gamma(s),
\]
\begin{equation}\label{eq:DedekindLambda}
 \Lambda_K(s)=D_K^{s/2}\Gamma_{\R}(s)^{r_1}
 \Gamma_{\C}(s)^{r_2}\zeta_K(s),
\end{equation}
\begin{equation}\label{eq:XiK}
 \Xi_K(s)=s(s-1)\Lambda_K(s),\qquad
 F_K(u)=\Xi_K\!\left(\frac12+u\right).
\end{equation}
We use the standard analytic continuation, functional equation, Euler
product, and growth properties of the Dedekind zeta function; see
\cite{NeukirchANT}.  In particular, \(F_K\) is a real and even entire function of order one,
and \(\Zset(F_K)\subseteq S_{1/2}\).
For $m,n\ge0$, define $\mathscr Q_{K,n}^{(m)}$ by
\begin{equation}\label{eq:intro-dedekind-quotient}
 B_{2n+1}[F_K^{(m)}](z)
 =z^n\bigl(z+(-1)^m\bigr)
 \mathscr Q_{K,n}^{(m)}(z+z^{-1}).
\end{equation}
Equivalently, by \eqref{eq:parity-correct-Q},
\begin{equation}\label{eq:dedekind-derivative-quotient}
 \mathscr Q_{K,n}^{(m)}(X)
 =\sum_{r=0}^{n}\binom{2n+1}{n-r}\Xi_K^{(m)}(r+1)
   \bigl(U_r(X/2)-(-1)^mU_{r-1}(X/2)\bigr),\quad U_{-1}=0.
\end{equation}
\Needspace{6\baselineskip}
\begin{theorem}[Horizontal and vertical strict interlacing for Dedekind zeta derivatives]
\label{thm:main-dedekind-array}
Let $K$ be a number field, $m\ge0$, and $n\ge1$.  Then:
\begin{enumerate}[label=\textup{(\roman*)}]
\item $\mathscr Q_{K,n}^{(m)}$ has $n$ simple zeros in $(-2,2)$;
\item $\mathscr Q_{K,n}^{(m)}$ and $\mathscr Q_{K,n+1}^{(m)}$
strictly interlace;
\item if $m_{\mathrm E}$ and $m_{\mathrm O}$ are the even and odd
members of $\{m,m+1\}$, and
\[
 X_1^{\mathrm E}>\cdots>X_n^{\mathrm E},\qquad
 X_1^{\mathrm O}>\cdots>X_n^{\mathrm O}
\]
are the zeros of \(\mathscr Q_{K,n}^{(m_{\mathrm E})}\) and
\(\mathscr Q_{K,n}^{(m_{\mathrm O})}\), respectively, then
\[
 2>X_1^{\mathrm E}>X_1^{\mathrm O}>X_2^{\mathrm E}>
 X_2^{\mathrm O}>\cdots>X_n^{\mathrm E}>X_n^{\mathrm O}>-2.
\]
\end{enumerate}
\end{theorem}

Equivalently, the augmented sets
$\{2\}\cup\Zset(\mathscr Q_{K,n}^{(m_{\rm O})})$ and
$\Zset(\mathscr Q_{K,n}^{(m_{\rm E})})\cup\{-2\}$ strictly interlace.
Every consecutive derivative pair is also a strict Obreschkoff pair in the
actual-degree convention. We prove (i) and (ii) directly from the source-side
strictness theorem, then use the fixed-degree derivative theorem for (iii).
The nondegeneracy required for that vertical comparison is proved in
Appendix~\ref{app:dedekind-nondegeneracy}.

\begin{lemma}[Dedekind finite-difference nondegeneracy]
\label{lem:dedekind-strict-nondeg}
Let $K$ be a number field and let $F_K$ be the centered source defined in
\eqref{eq:XiK}.
For every finite-difference order \(\ell\ge2\), every \(m\ge0\), and every
\(\zeta\in\T\),
\begin{equation}
\label{eq:dedekind-nondeg}
 \mathscr T_{\zeta,1}^{\ell}F_K^{(m)}\not\equiv0.
\end{equation}
\end{lemma}

The proof uses Stirling's formula and the Euler product: the rightmost
translate dominates the others uniformly in $\zeta\in\T$.

\begin{proof}[Proof of Theorem~\ref{thm:main-dedekind-array}]
\emph{Simplicity and horizontal comparison, (i) and (ii).}
By Lemma~\ref{lem:derivative-strip}, $F_K^{(m)}$ is a real entire
function of the required parity, of order at most one, with zeros in
$S_{1/2}$. It is nonpolynomial, since otherwise $F_K$ would be
polynomial, contrary to its positive-ray gamma growth
(Appendix~\ref{app:dedekind-nondegeneracy}). Since
$1/2<\sqrt{15/28}<h_*$, Corollary~\ref{cor:sharp-source-strictness},
with the quotient normalization \eqref{eq:intro-dedekind-quotient},
gives (i) and (ii).

\emph{Vertical comparison, (iii).}
Fix $n\ge1$ and set $d=2n+1$. Apply
\cite[Theorem~1.2]{JinSharpSampling} with source $F_K$ and spacing
$\delta=1$. The real even entire function $F_K$ has order one and zeros
in $S_{1/2}$, with $1/2<\sqrt d/2$.
Lemma~\ref{lem:dedekind-strict-nondeg} verifies the nonvanishing conditions
of that theorem, which gives simple
unit-circle zeros and strict cyclic interlacing for the samples of
$F_K^{(m)}$ and $F_K^{(m+1)}$.
Proposition~\ref{prop:shifted-descent}, with the factorization
\eqref{eq:intro-dedekind-quotient}, gives (iii), the augmented form,
and the actual-degree strict Obreschkoff conclusion.
\end{proof}

\subsection{Nested centered blocks at critical integers for self-dual newforms}
\label{sec:newform-blocks}

For a self-dual newform of weight $k$, we sample the critical integers
$1,\ldots,k-1$.
An arbitrary-spacing rightmost-translate estimate supplies both fixed-degree
nondegeneracy and the nonpolynomiality needed for horizontal strictness.

Let
\[
 f(\tau)=\sum_{r\ge1}a_f(r)e^{2\pi\mathrm{i}r\tau}
 \in S_k^{\mathrm{new}}(\Gamma_0(N),\chi)
\]
be a normalized primitive holomorphic newform of weight $k$, level $N$,
and nebentypus $\chi$, with real Fourier coefficients; $\Im\tau>0$.
Here self-dual means that its completed $L$-function satisfies the real
functional equation displayed below.  With
\begin{equation}\label{eq:newform-completion}
 L(f,s)=\sum_{r\ge1}\frac{a_f(r)}{r^s},\qquad
 \Lambda(f,s)=\left(\frac{\sqrt N}{2\pi}\right)^s\Gamma(s)L(f,s),
\end{equation}
one has
\begin{equation}\label{eq:newform-functional-equation}
 \Lambda(f,s)=\varepsilon_f\Lambda(f,k-s),
 \qquad \varepsilon_f\in\{\pm1\}.
\end{equation}
The sign $\varepsilon_f$ is the root number.  The completion is entire of
order one.  Absolute convergence and nonvanishing of the Euler product to the
right of $(k+1)/2$, together with the functional equation, give
\begin{equation}\label{eq:newform-zero-strip}
 \Zset\!\left(u\longmapsto\Lambda\!\left(f,\frac{k}{2}+u\right)\right)
 \subseteq S_{1/2};
\end{equation}
see \cite{DeligneWeilI,IwaniecKowalski}.  Put
\[
 F_f(u)=\Lambda\!\left(f,\frac{k}{2}+u\right).
\]

For integers $\delta\ge1$ and $n\ge1$ satisfying
\begin{equation}\label{eq:intro-newform-admissibility}
 \delta(2n+1)\le k-2,
 \qquad
 k-\delta(2n+1)\equiv0\pmod2,
\end{equation}
every sampled argument is an integer in $\{1,2,\ldots,k-1\}$.
We call such a pair $(\delta,n)$, and the resulting block, \emph{admissible}.
Define the centered block at critical integers
\begin{equation}\label{eq:intro-newform-sample}
 \mathcal P_{f,\delta,n}^{(m)}(z)
 =\sum_{j=0}^{2n+1}\binom{2n+1}{j}
 \Lambda^{(m)}\!\left(
 f,\frac{k}{2}+\delta\left(j-n-\frac12\right)
 \right)z^j.
\end{equation}
The functional equation gives a unique real quotient
$\mathscr Q_{f,\delta,n}^{(m)}$ such that
\begin{equation}\label{eq:intro-newform-factorization}
 \mathcal P_{f,\delta,n}^{(m)}(z)
 =z^n\bigl(z+\varepsilon_f(-1)^m\bigr)
 \mathscr Q_{f,\delta,n}^{(m)}(z+z^{-1}).
\end{equation}
Increasing $n$ by one adds one sampled value of
$\Lambda^{(m)}(f,\mathord\cdot)$ at a critical integer at each end while
preserving the old sampling points.

\begin{corollary}[Nested blocks at critical integers for self-dual newforms]
\label{cor:main-newform-blocks}
Under the preceding assumptions, let $m\ge0$, $\delta\ge1$, and $n\ge1$, and
assume that the larger block is admissible:
\[
 \delta(2n+3)\le k-2,
 \qquad
 k-\delta(2n+3)\equiv0\pmod2.
\]
Then $\mathscr Q_{f,\delta,n}^{(m)}$ and
$\mathscr Q_{f,\delta,n+1}^{(m)}$ have degrees $n$ and $n+1$, respectively,
have simple zeros in $(-2,2)$, and strictly interlace.
\end{corollary}

The next lemma records arbitrary-spacing nondegeneracy. The block
comparison below uses the positive-ray growth estimate in its proof
(Appendix~\ref{app:newform-nondegeneracy}) to verify nonpolynomiality.

\begin{lemma}[Arbitrary-spacing nondegeneracy for newforms]
\label{lem:newform-arbitrary-spacing-nondegeneracy}
For every fixed \(\delta>0\), every \(m,\ell\in\mathbb Z_{\ge0}\), and every
\(\zeta\in\T\),
\[
 \mathscr T_{\zeta,\delta}^{\,\ell}F_f^{(m)}\not\equiv0.
\]
More precisely, as \(X\to+\infty\),
\begin{equation}\label{eq:newform-arbitrary-spacing-dominance}
 \frac{\mathscr T_{\zeta,\delta}^{\,\ell}F_f^{(m)}(X)}
 {\zeta^{\ell}F_f^{(m)}(X+\delta\ell/2)}
 =1+O_{f,m,\ell,\delta}(X^{-\delta}),
\end{equation}
uniformly for \(\zeta\in\T\).
\end{lemma}

The proof is deferred to Appendix~\ref{app:newform-nondegeneracy}.

\begin{proof}[Proof of Corollary~\ref{cor:main-newform-blocks}]
Scale to unit spacing. The real source $u\mapsto F_f^{(m)}(\delta u)$
has the required parity, order at most one and zeros in $S_{1/(2\delta)}$
by Lemma~\ref{lem:derivative-strip};
\eqref{eq:newform-arbitrary-spacing-individual-growth} ensures
nonpolynomiality. Since $\delta\ge1$ and
$1/(2\delta)\le1/2<\sqrt{15/28}<h_*$,
Corollary~\ref{cor:sharp-source-strictness} applies to the quotients in
\eqref{eq:intro-newform-factorization}. The admissibility conditions
ensure that both blocks use only critical integers.
\end{proof}

\begin{example}[Nested blocks for the Ramanujan form]
\label{ex:ramanujan-nested-blocks}
For $f=\Delta$, $k=12$, and $\delta=2$, formula
\eqref{eq:intro-newform-sample} samples $\Lambda^{(m)}(\Delta,s)$ at
\[
 \begin{array}{c|c|c}
 n&\text{critical arguments}&\text{binomial weights}\\ \hline
 1&3,5,7,9&1,3,3,1\\
 2&1,3,5,7,9,11&1,5,10,10,5,1
 \end{array}
\]
with the listed weights multiplying $z^j$ in order. The two real quotients
strictly interlace in $(-2,2)$ for every $m\ge0$.
\end{example}

\subsection{Further consequences of horizontal interlacing}
\label{sec:number-fields}
\label{subsec:horizontal-common-root-criterion}

The first nontrivial Dedekind comparison yields an explicit signed
inequality, while a lowering identity for arbitrary finite sample data
gives a pointwise common-root criterion. These are consequences of the
horizontal theory rather than additional arithmetic inputs.

\medskip
\noindent\textit{A low-degree signed consequence.}

The first nontrivial horizontal sampling-degree comparison already yields
an explicit inequality among values of completed Dedekind zeta derivatives.
Retain the normalization \eqref{eq:DedekindLambda}--\eqref{eq:XiK}, and set
\begin{equation}\label{eq:completed-residue-notation}
 \mathcal R_K=\Xi_K(1)=\Residue_{s=1}\Lambda_K(s).
\end{equation}
For $m\ge0$, put
\[
 R_m=\Xi_K^{(m)}(1),\qquad
 X_m=\Xi_K^{(m)}(2),\qquad
 Y_m=\Xi_K^{(m)}(3),\qquad
 \epsilon_m=(-1)^m.
\]
Thus $R_0=\mathcal R_K$.
The quotients for $n=1$ and $n=2$ (sampling degrees $3$ and $5$) are
\begin{align}
 \mathscr Q_{K,1}^{(m)}(X)
 &=X_mX+3R_m-\epsilon_mX_m,\label{eq:dedekind-degree-three-Q}\\
 \mathscr Q_{K,2}^{(m)}(X)
 &=Y_mX^2+(5X_m-\epsilon_mY_m)X
   +10R_m-5\epsilon_mX_m-Y_m.\label{eq:dedekind-degree-five-Q}
\end{align}
\Needspace{8\baselineskip}
We use the standard convention
\[
 \Resultant_X(A,B)=a_A^{\deg B}\prod_{A(\alpha)=0}B(\alpha),
\]
where $a_A$ is the leading coefficient of $A$ and roots are counted by
multiplicity; the variable subscript is omitted when clear.

For every number field $K$ and every $m\ge0$,
\begin{equation}\label{eq:dedekind-consecutive-degree-resultant}
 Y_m\Bigl(
 9R_m^2Y_m-5X_m^2R_m
 -3\epsilon_mX_mR_mY_m-X_m^2Y_m
 \Bigr)<0.
\end{equation}
Equivalently,
\[
 Y_m\,\Resultant_X\!\left(
 \mathscr Q_{K,1}^{(m)},\mathscr Q_{K,2}^{(m)}
 \right)<0.
\]
For $m=0$, writing
\[
 X_K=\Xi_K(2)=2\Lambda_K(2),\qquad
 Y_K=\Xi_K(3)=6\Lambda_K(3),
\]
one obtains the positive completed-value inequality
\begin{equation}\label{eq:dedekind-consecutive-degree-residue}
 5\mathcal R_KX_K^2+X_K^2Y_K
 +3\mathcal R_KX_KY_K-9\mathcal R_K^2Y_K>0.
\end{equation}

\begin{proof}
Let
\[
 \xi_m=\epsilon_m-\frac{3R_m}{X_m},
\]
the zero of $\mathscr Q_{K,1}^{(m)}$. By
Theorem~\ref{thm:main-dedekind-array}(ii), $X_mY_m\ne0$ and $\xi_m$ lies
strictly between the two zeros of $\mathscr Q_{K,2}^{(m)}$; hence
\[
 Y_m\mathscr Q_{K,2}^{(m)}(\xi_m)<0.
\]
Since
\[
 \Resultant_X(\mathscr Q_{K,1}^{(m)},\mathscr Q_{K,2}^{(m)})
 =X_m^2\mathscr Q_{K,2}^{(m)}(\xi_m),
\]
substitution of \eqref{eq:dedekind-degree-five-Q} gives
\eqref{eq:dedekind-consecutive-degree-resultant}. The case $m=0$, together
with positivity of $\mathcal R_K,X_K,Y_K$, yields
\eqref{eq:dedekind-consecutive-degree-residue}.
\end{proof}

\noindent\emph{Interpretation.} The inequality \eqref{eq:dedekind-consecutive-degree-resultant} uses the relative
position of zeros for the two consecutive quotient indices $n=1$ and $n=2$ (sampling degrees $3$ and $5$).
It cannot be recovered merely from the fact that each quotient is separately
rooted in $[-2,2]$.

To identify when the horizontal resultant vanishes, we now derive a
lowering identity that holds for arbitrary finite sample data, independently
of the arithmetic source. It gives a pointwise test at an individual zero,
in contrast to the global nondegeneracy condition for vertical comparison.

Put
\(\mathcal D_z=z\,d/dz\).

\begin{proposition}[Lowering identity between sampling degrees $d$ and $d+2$]
\label{prop:consecutive-degree-lowering}
For every function \(F\) for which the samples are defined and every \(d\ge0\),
\begin{equation}\label{eq:consecutive-degree-lowering}
 \mathcal D_z(d+2-\mathcal D_z)B_{d+2}[F](z)
 =(d+2)(d+1)zB_d[F](z).
\end{equation}
Let \(n\ge0\), \(\varepsilon_{\mathrm{end}}\in\{1,-1\}\), and suppose that
\[
 B_{2n+3}[F](z)=z^{n+1}(z+\varepsilon_{\mathrm{end}})P_{n+1}(X),
 \qquad
 B_{2n+1}[F](z)=z^n(z+\varepsilon_{\mathrm{end}})P_n(X),
 \qquad X=z+z^{-1}.
\]
Then
\begin{equation}\label{eq:consecutive-degree-differential-lowering}
 \begin{split}
 &(2n+3)(2n+2)P_n(X)\\
 &\quad=(4-X^2)P_{n+1}''(X)
       +2(\varepsilon_{\mathrm{end}}-X)P_{n+1}'(X)
       +(n+1)(n+2)P_{n+1}(X).
 \end{split}
\end{equation}
In particular, if \(P_{n+1}\not\equiv0\), an interior zero \(\xi\) of
\(P_{n+1}\) is also a zero of \(P_n\) if and only if
\begin{equation}\label{eq:horizontal-curvature-obstruction}
 (4-\xi^2)P_{n+1}''(\xi)
 +2(\varepsilon_{\mathrm{end}}-\xi)P_{n+1}'(\xi)=0.
\end{equation}
\end{proposition}

\begin{proof}
For $1\le j\le d+1$,
\[
 j(d+2-j)\binom{d+2}{j}=(d+2)(d+1)\binom d{j-1},\qquad
 j-\frac{d+2}{2}=(j-1)-\frac d2.
\]
This proves \eqref{eq:consecutive-degree-lowering} coefficientwise; both
endpoint coefficients vanish.

For the quotient identity set $W=z^{n+1}(z+\varepsilon_{\mathrm{end}})$ and $X=z+z^{-1}$.
The product rule, $\mathcal D_zX=z-z^{-1}$ and
$\mathcal D_z(z-z^{-1})=X$ give, for every polynomial $P$,
\[
 \begin{aligned}
 W^{-1}\mathcal D_z(2n+3-\mathcal D_z)(WP(X))
 ={}&(4-X^2)P''(X)+2(\varepsilon_{\mathrm{end}}-X)P'(X)\\
    &+(n+1)(n+2)P(X).
 \end{aligned}
\]
Apply \eqref{eq:consecutive-degree-lowering} and cancel the common
factor $W$ to obtain \eqref{eq:consecutive-degree-differential-lowering}.
These are rational identities, hence polynomial identities after clearing
denominators; no global parity, analyticity or reality of the finite data
is required. Evaluating at a zero gives
\eqref{eq:horizontal-curvature-obstruction}.
\end{proof}

For a weak pair $P_n\preceq P_{n+1}$ of exact adjacent degrees with simple
zeros in $(-2,2)$ and the preceding endpoint factorizations,
\eqref{eq:horizontal-curvature-obstruction} says that strict interlacing,
nonvanishing of its left-hand expression at every zero of $P_{n+1}$, and
$\Resultant(P_n,P_{n+1})\ne0$ are equivalent: each excludes precisely the
common roots. In particular, Theorem~\ref{thm:main-dedekind-array} gives,
for $P=\mathscr Q_{K,n+1}^{(m)}$ and $\varepsilon_{\mathrm{end}}=(-1)^m$,
\begin{equation}\label{eq:dedekind-horizontal-strictness-criterion}
 (4-\xi^2)P''(\xi)+2(\varepsilon_{\mathrm{end}}-\xi)P'(\xi)\ne0\qquad(P(\xi)=0).
\end{equation}
\appendix
\section{Contact geometry for Jacobi operators}
\label{app:global-contact-lemmas}

The compactification, tangent and local Rolle arguments in this appendix
are used in Sections~\ref{sec:real-spectral-factors} and
\ref{sec:odd-enlarged-blocks}. The remaining lemmas analyze the fourth-order
odd blocks. All statements use the normalized interval coordinate $x$ of
Section~\ref{subsec:normalized-exact-degree}.

\subsection{The frozen symbol}

For fixed real $\Omega,\chi$, abbreviate $\mathcal T=\mathcal T_{\Omega,\chi}$. Write
\[
 \mathcal T_{\Omega,\chi}=\sum_{j=0}^4c_j(x)\partial_x^j,
\]
and abbreviate
\[
 \mathsf a=x(x-1),\qquad
 \mathsf b=2x-\frac12,\qquad
 \mathsf s=4x-\frac32.
\]
A direct composition gives
\[
 c_4=\mathsf a^2,
 \quad c_3=2\mathsf a\mathsf s,
 \quad c_2=(4+2\Omega)\mathsf a+\mathsf b\mathsf s,
 \quad c_1=2\Omega\mathsf b,
 \quad c_0=(\Omega-1)^2+\chi.
\]
At an observation point \(\xi\in(0,1)\), let
\[
 p_\xi(t)=\sum_{j=0}^4c_j(\xi)t^j.
\]

\begin{lemma}[Strict criticality]
\label{lem:strict-frozen-symbol}
If \(\Omega>0\), then \(p_\xi'\) has three distinct real roots for every
\(\xi\in(0,1)\).  The assertion is independent of \(\chi\).
\end{lemma}
\begin{proof}
Put $\Delta=\xi(1-\xi)$ and $b_\xi=2\xi-1/2$.
The cubic $p_\xi'$ has leading coefficient $4\Delta^2>0$, and
\[
 p_\xi'(0)=2\Omega b_\xi,\qquad
 p_\xi'\!\left(\frac{b_\xi}{2\Delta}\right)
 =\frac{b_\xi}{2\Delta}\left(4\xi^2-5\xi-\frac12\right).
\]
The last factor is negative for $0<\xi<1$. If $b_\xi\ne0$, the two
evaluation points, placed between $-\infty$ and $+\infty$ in increasing
order, give signs $-,+,-,+$. Thus the cubic has three distinct real roots.
If $b_\xi=0$, then
\[
 p_\xi'(t)=t\{4\Delta^2t^2-6\Delta(4\xi-3/2)t-2(4+2\Omega)\Delta\}.
\]
The quadratic has positive leading coefficient and negative constant term,
so its roots are nonzero and of opposite sign. The parameter $\chi$
occurs only in $p_\xi$'s constant term and disappears on differentiation.
\end{proof}

\subsection{Root updates, inherited contacts, and strict deletion}

Let \(g(x)=\prod_i(x-r_i)\) be real-rooted and suppose \(g(\xi)\ne0\).
Set
\[
 z_i=\frac1{\xi-r_i},
 \qquad \mathcal U_z=\operatorname{Id}+z\partial_t,
 \qquad
 \mathfrak F_{\xi,g}=\prod_i\mathcal U_{z_i}p_\xi.
\]
Here \(t\) is the symbol variable; the quantities \(z_i\) are fixed update parameters.  Leibniz's
rule gives
\begin{equation}
\label{eq:symbol-evaluations}
 \mathfrak F_{\xi,g}(0)=\frac{(\mathcal Tg)(\xi)}{g(\xi)},
 \qquad
 \mathfrak F_{\xi,g}'(0)=\frac{[\mathcal T,M_x]g(\xi)}{g(\xi)}.
\end{equation}
For a polynomial $q$ with simple real zeros $\lambda_j$, the zeros of
$q+zq'$ for real $z\ne0$ avoid the $\lambda_j$ and satisfy
\[
 \frac{q'(t)}{q(t)}=\sum_j\frac1{t-\lambda_j}=-\frac1z.
\]
The left side is strictly decreasing between poles: there is one simple
zero in each bounded gap and one in the appropriate exterior interval.
The case $z=0$ is immediate. Rolle's theorem also gives simple real zeros
for each nonconstant derivative. Thus successive updates retain the
simple real-rootedness of $\mathfrak F'_{\xi,g}$ needed below.

\begin{lemma}[Inherited multiplicity]
\label{lem:inherited-multiplicity}
Assume Lemma~\ref{lem:strict-frozen-symbol}.  If a real-rooted input has a root
of multiplicity \(e\ge1\) at \(\xi\), then
\[
 \min\{\operatorname{ord}_\xi(\mathcal Tg),
        \operatorname{ord}_\xi(\mathcal T(xg))\}
 =\max\{e-4,0\}.
\]
\end{lemma}
\begin{proof}
Write \(g=(x-\xi)^eh\), with \(h(\xi)\ne0\).  The master identity is
\[
 \frac{\mathcal T[(x-\xi)^eh](\xi)}{h(\xi)}
 =\mathfrak F_{\xi,h}^{(e)}(0).
\]
For \(e<4\), simultaneous vanishing of the two columns would make
\(\mathfrak F_{\xi,h}^{(e)}\) and its derivative vanish at zero, impossible
because the former is simple real-rooted.  For \(e\ge4\), the nonzero fourth-order coefficient gives
\(\operatorname{ord}_\xi(\mathcal Tg)=e-4\).  Also
\(\mathcal T(xg)=\xi\mathcal Tg+\mathcal T((x-\xi)g)\), where the last
term has order \(e-3\); since \(\xi\ne0\),
\(\operatorname{ord}_\xi(\mathcal T(xg))=e-4\).
\end{proof}

\begin{lemma}[Strict deletion]
\label{lem:strict-deletion-global}
Let \(\Omega>0\), \(\chi\in\mathbb R\),
\(\mathcal T=\mathcal T_{\Omega,\chi}\), and \(\xi\in(0,1)\).  Let \(g\)
be nonzero and real-rooted, with \(g(\xi)\ne0\), and suppose
\[
 (\mathcal Tg)(\xi)=(\mathcal T(xg))(\xi)=0.
\]
For every labelled root \(r_i\) of \(g\), put \(g_i=g/(x-r_i)\).  Then
\[
 (\mathcal Tg_i)(\xi)\ne0.
\]
\end{lemma}
\begin{proof}
Put $\mathfrak F=\mathfrak F_{\xi,g}$ and delete the update
\(\mathcal U_{z_i}\) to obtain \(\mathfrak F_i\).
Lemma~\ref{lem:strict-frozen-symbol} and the preceding update argument make
\(\mathfrak F_i'\) simple real-rooted, while
\[
 \mathfrak F=\mathfrak F_i+z_i\mathfrak F_i',
 \qquad \mathfrak F(0)=\mathfrak F'(0)=0.
\]
If \(\mathfrak F_i(0)=0\), the two displayed equations force
\(\mathfrak F_i'(0)=\mathfrak F_i''(0)=0\), contradicting simple
real-rootedness of \(\mathfrak F_i'\).  Finally
\((\mathcal Tg_i)(\xi)=g(\xi)z_i\mathfrak F_i(0)\ne0\).
\end{proof}

\subsection{Compactification and rigid contacts}

Use the compactification $\widehat{\mathfrak E}_N$ defined in
Section~\ref{subsec:normalized-exact-degree}. It is compact as the continuous
image of $[0,1]^{N-1}\times\mathbb{RP}^1$; its infinity stratum has
an interval-rooted affine representative of degree $N-1$.

\begin{lemma}[Fixed-degree first exit]
\label{lem:fixed-degree-first-exit}
Let \(T_\tau:\mathbb R[x]_{\le N}\to\mathbb R[x]_{\le N}\) be a
coefficientwise continuous path preserving \(\mathbb R[x]_{\le d}\) for
\(0\le d\le N\), with nonzero induced leading multipliers in degrees \(N\)
and \(N-1\).  If \(T_0\) preserves \(\mathfrak E_N\) and \(T_1\) does not,
then there is a first parameter \(\tau_c\) at which \(T_{\tau_c}\)
weakly preserves \(\widehat{\mathfrak E}_N\), and a compactified input
whose output lies on the boundary.
\end{lemma}
\begin{proof}
The two leading multipliers keep finite-stratum images of exact degree
$N$ and infinity-stratum images of affine degree $N-1$. After
homogenization, $(\tau,[f])\mapsto[T_\tau f]$ is therefore well defined
and continuous on $[0,1]\times\widehat{\mathfrak E}_N$.
The set $\Sigma=\{\tau\in[0,1]:T_\tau\widehat{\mathfrak E}_N
\subseteq\widehat{\mathfrak E}_N\}$ is closed, contains $0$, and
excludes $1$. Hence $\tau_c=\sup\{\tau:[0,\tau]\subseteq\Sigma\}$
belongs to $\Sigma$ and is less than $1$.
Choose failing parameters $\tau_j\downarrow\tau_c$ and corresponding
inputs $[f_j]$. Compactness gives a subsequence converging to $[f_*]$,
with outputs converging to $[T_{\tau_c}f_*]\in\widehat{\mathfrak E}_N$.
The limit cannot be interior, since the nearby outputs fail to belong
to the class. It is the required boundary output.
\end{proof}

\begin{lemma}[Two-sided multiple-root tangent]
\label{lem:two-sided-multiple-root-tangent}
Let \(p_\varepsilon=p_0+\varepsilon q+O(\varepsilon^2)\) be a
coefficientwise smooth real-polynomial family.  Suppose \(p_0\) has a root of
multiplicity \(e\ge2\) at \(\xi\), and \(p_\varepsilon\) is real-rooted for
all sufficiently small positive and negative \(\varepsilon\).  Then
\(q(\xi)=0\).
\end{lemma}
\begin{proof}
If \(q(\xi)\ne0\), use the scale
\(x-\xi=|\varepsilon|^{1/e}t\).  After division by \(|\varepsilon|\), one
sign of \(\varepsilon\) converges to a binomial \(a_et^e\pm q(\xi)\) having
a simple nonreal root.  Rouch\'e's theorem on a small disk disjoint from the
real axis then gives a nonreal root of \(p_\varepsilon\), a contradiction.
\end{proof}

\begin{lemma}[Rigid off-support contact shapes]
\label{lem:rigid-contact-shapes}
Let a real linear operator $T$ weakly preserve $\widehat{\mathfrak E}_N$.
Assume that, for every interval-rooted degree-$(N-1)$ polynomial $g$,
an off-support companion collision $Tg(\xi)=T(xg)(\xi)=0$ in $(0,1)$
has strict deletion: $T(g/(x-r_i))(\xi)\ne0$ for every labelled root $r_i$.
If a compactified input has a multiple output root $\xi\in(0,1)$
avoiding its full support, then, up to nonzero scale, it is one of
\[
 (x-\beta)x^k(1-x)^\ell,\qquad (x-r)^N,\qquad x^k(1-x)^\ell,
\]
where $k,\ell\ge0$ are integers with $k+\ell=N-1$, $\beta\notin[0,1]$,
$r\in[0,1]$, and the last form represents an infinity defect.
\end{lemma}
\begin{proof}
For a finite input write $f=(x-\beta)g$. Moving $\beta$ in both directions
and applying Lemma~\ref{lem:two-sided-multiple-root-tangent} gives
$Tg(\xi)=T(xg)(\xi)=0$. For each two-sided movable labelled root $r_i$,
with $g_i=g/(x-r_i)$, the same tangent condition gives
\[
 0=T[(x-\beta)g_i](\xi)=(r_i-\beta)(Tg_i)(\xi).
\]
Strict deletion forces $r_i=\beta$. If $\beta$ is exterior, all interior
roots of $g$ are therefore excluded. If $\beta\in[0,1]$, every copy,
including an endpoint, is movable as the sole allowed defect, so all
roots equal $\beta$. At infinity the admissible variation $(1+\varepsilon x)g$
gives the companion equations. Moving an interior root of $g$ would give
$Tg_i(\xi)=0$, again contradicting strict deletion.
\end{proof}

\begin{lemma}[Endpoint contacts on an admissible path]
\label{lem:admissible-endpoint-contacts}
Assume \(\Omega>0\) and \(c_0=(\Omega-1)^2+\chi\ge0\).  At a weak-preserver
boundary for \(\mathcal T_{\Omega,\chi}\), a multiple output root at either
endpoint can occur only when the endpoint input multiplicity is \(e\ge4\); its
exact output multiplicity is \(e-2\).
\end{lemma}
\begin{proof}
Use the inward coordinate $t=x$ at $0$ and $t=1-x$ at $1$, and put
$\kappa=1/2$ and $3/2$, respectively. In either coordinate the operator is
\[
 \mathcal T=(J_{\kappa,2-\kappa}+(\Omega-1)\Id)^2+\chi\Id.
\]
With derivatives taken in $t$, direct evaluation gives
\[
 \begin{aligned}
 \frac{\mathcal Th(0)}{h(0)}
 &=c_0-2\kappa\Omega\frac{h'(0)}{h(0)}
       +\kappa(\kappa+1)\frac{h''(0)}{h(0)},\\
 \frac{\mathcal T(th)(0)}{h(0)}
 &=-2\kappa\Omega+2\kappa(\kappa+1)\frac{h'(0)}{h(0)}.
 \end{aligned}
\]
For an interval-rooted $h$ avoiding $0$, one has $h'(0)/h(0)\le0$
and $h''(0)/h(0)\ge0$, with the first inequality strict in positive
degree. Thus the first ratio is positive in positive degree and the
second is negative. For $e\ge2$ and any $h$ with $h(0)\ne0$,
\[
 \mathcal T(t^eh)
 =e(e-1)(e-1+\kappa)(e-2+\kappa)h(0)t^{e-2}+O(t^{e-1}).
\]
The indicial coefficient is positive. Hence the exact output order is
$e-2$: there is no endpoint zero for $e=2$, and a simple one for $e=3$.

It remains to justify that the same conclusion holds for the full one-defect
class. Returning to $x$, write a finite-defect input near an endpoint \(b\in\{0,1\}\) as
\[
 f=(x-b)^e(x-\beta)h,
 \qquad \beta\notin[0,1],
\]
where \(h\) is interval-rooted and avoids \(b\).  If \(e=0\) or \(e=1\)
and \(\mathcal Tf\) had a multiple root at \(b\), moving \(\beta\) in both
directions would remain in the one-defect class.  Lemma
\ref{lem:two-sided-multiple-root-tangent} would then force
\(\mathcal T((x-b)^eh)(b)=0\), contradicting the strict formulas above.
For \(e=2\) and \(e=3\), the displayed nonzero indicial coefficients apply
directly to the remaining nonvanishing factor \((x-\beta)h\).  The
projective-infinity defect is represented affinely by an interval-rooted
polynomial of degree \(N-1\), so the same argument applies.  Thus a multiple
endpoint output requires \(e\ge4\), and its multiplicity is exactly \(e-2\).
\end{proof}

\subsection{Protection of forced clusters and exterior roots}

\begin{lemma}[Local weighted Rolle factorization]
\label{lem:local-weighted-rolle}
Let \(T_s\) be a continuous family of real differential operators of order
\(d\), with nonvanishing leading coefficient near \((s_0,\xi)\).  On a
common smaller interval and for \(s\) near \(s_0\), one has
\[
 T_s=M_{\gamma_s}\partial_xM_{1/u_{s,d-1}}\partial_x\cdots
       \partial_xM_{1/u_{s,0}},
\]
where every \(u_{s,j}\) is positive and \(\gamma_s\) is nonzero.  Hence an
input with \(e\) real zeros in the interval, counted with multiplicity,
produces at least \(e-d\) real output zeros there.
\end{lemma}
\begin{proof}
Choose the local solution \(u_{s,0}\) of \(T_su=0\) with
\(u_{s,0}(\xi)=1\) and fixed Cauchy data.  Continuous dependence gives
\(u_{s,0}>0\) on a common smaller interval.  Since
\(T_sM_{u_{s,0}}\) annihilates constants, it factors as
\(T_sM_{u_{s,0}}=T_s^{[1]}\partial_x\); the leading coefficient of
\(T_s^{[1]}\) is the leading coefficient of \(T_s\) multiplied by
\(u_{s,0}\), and is therefore nonzero.  Repeat the construction for
\(T_s^{[1]}\).  Multiplication by a positive function preserves zero
multiplicity, while each differentiation loses at most one real zero by the
multiplicity form of Rolle's theorem.
\end{proof}

If an input cluster has multiplicity \(e>d\), the leading differential term
gives exact limiting output multiplicity \(e-d\).  Root continuity gives
exactly \(e-d\) nearby output roots in a small complex disk, while Lemma
\ref{lem:local-weighted-rolle} supplies at least that many real roots there;
thus the whole forced cluster is real.

At the endpoints, the exact lifts
\[
 (\mathcal E_-f)(t)=f(t^2),
 \qquad
 (\mathcal E_+f)(t)=t f(1-t^2)
\]
intertwine \(\mathcal L\) with
\[
 L_-=(t^2-1)\partial_t^2/4+3t\partial_t/4+1/4,
 \qquad
 L_+=(t^2-1)\partial_t^2/4+t\partial_t/4.
\]
For the left-endpoint lift, a wrong-side root is a root of the original
polynomial in \((-\infty,0)\); for the right-endpoint lift, it is a root in
\((1,\infty)\).  Such a root becomes a pair of purely imaginary roots under
the corresponding quadratic lift.

The lifted quartet is regular of order four at zero, with leading coefficient
\(1/16\).  More generally, suppose a quadratic lift of parity
\(\epsilon\in\{0,1\}\) intertwines an order-\(2r\) regular operator.  If an
input cluster of multiplicity \(e\) contains \(q\) wrong-side roots and the
lifted output has its exact forced multiplicity
\(M=2(e-r)+\epsilon\), then, writing \(\nu_{\rm out}\) for the number of wrong-side real
output roots and \(v\) for the number of nonreal conjugate output pairs,
\[
 \nu_{\rm out}+2v\le q.
\]
Indeed the lifted input has at least \(2(e-q)+\epsilon\) real roots, so the
weighted Rolle factorization leaves at least \(M-2q\) real output roots.
A wrong-side real root consumes two purely imaginary lifted roots, while one
nonreal conjugate pair consumes four.  Thus \(2\nu_{\rm out}+4v\le2q\).  For the
one-defect class \(q\le1\), so a new nonreal pair is impossible.

A projective-infinity input is controlled by the degree filtration.  The
positive degree-\(N\) and degree-\((N-1)\) multipliers keep the projective
root at infinity simple under the path, and Lemma
\ref{lem:local-weighted-rolle} protects the remaining forced affine cluster.

For exterior roots, the scalar conditions in
Theorem~\ref{thm:connected-path-preserver} imply that both exterior monomial
matrices are entrywise nonnegative and \(\mathrm{TN}_2\), as verified below.  Such a matrix maps a coefficient vector with at most one sign change to
one with at most one sign change.  For real-rooted polynomials, Descartes'
rule is exact after zero roots are removed.  Hence a one-defect input produces
at most one real exterior output root, on the same side as its defect; an
interval-rooted input produces none.

\paragraph{Exterior matrices and a scalar criterion.}
In either exterior coordinate the odd second-order operator is
\[
 \mathcal J_\kappa=t(t+1)\partial_t^2+(2t+\kappa)\partial_t
                       +\tfrac14\Id,\qquad \kappa\in\{1/2,3/2\}.
\]
For $T=\mathcal T_{\Omega,\chi}$ put
$v_j=j(j+1)+\Omega-1$ and $b_j=j(j-1+\kappa)$. Then
\[
 Tt^j=(v_j^2+\chi)t^j+b_j(v_j+v_{j-1})t^{j-1}
                         +b_jb_{j-1}t^{j-2},
\]
with terms below degree zero omitted. Under conditions (a)--(b) of
Theorem~\ref{thm:connected-path-preserver}, the diagonal entries at $j\ge1$
are at least $1+2\Omega$, and the $j=0$ entry is nonnegative. Also
$v_j+v_{j-1}=2(j^2+\Omega-1)>0$ for $j\ge1$.
Hence all entries are nonnegative. In this three-diagonal triangular
matrix, every order-two minor is zero or a product of entries, except
those on rows $j-1,j$ and columns $j,j+1$. After division by
$b_jb_{j+1}>0$, such a minor is
\[
 \begin{aligned}
 &(v_{j-1}+v_j)(v_j+v_{j+1})-(v_j^2+\chi)\\
 &\quad=3\Omega^2+10\Omega-1-\chi
       +3(j-1)(j+2)\bigl(j(j+1)+2\Omega\bigr),\qquad j\ge1.
 \end{aligned}
\]
The last term is nonnegative and vanishes at $j=1$. Thus both exterior
matrices are $\mathrm{TN}_2$ exactly when
$\chi\le3\Omega^2+10\Omega-1$. The same equivalence holds for every
truncation through degree $N\ge2$, since it includes the $j=1$ minor.

\subsection{Diagonal transversality}

Let
\[
 \pi_{N,r}(x)=(x-r)^N,
 \qquad
 S_N(x,r)=\mathcal T_{\Omega,\chi}\pi_{N,r}(x).
\]

\begin{theorem}[Real interior diagonal transversality]
\label{thm:diagonal-transversality-appendix}
For every integer \(N\ge3\), \(\Omega\ge0\), \(\chi\in\R\),
\(r\in[0,1]\), and \(\xi\in(0,1)\setminus\{r\}\), the three quantities
\[
 S_N(\xi,r),\qquad
 \partial_xS_N(\xi,r),\qquad
 \partial_rS_N(\xi,r)
\]
do not vanish simultaneously.
\end{theorem}
\begin{proof}
Put \(m=N-1\ge2\), \(z=x-r\), and write the coefficients of
\(\mathcal T_{\Omega,\chi}\) as in the first subsection.  Define
\begin{align*}
 \mathscr R_m^{(0)}={}&c_1z^3+2mc_2z^2
  +3m(m-1)c_3z+4m(m-1)(m-2)c_4,\\
 \mathscr R_m^{(1)}={}&c_1'z^3+mc_2'z^2
  +m(m-1)c_3'z+m(m-1)(m-2)c_4'.
\end{align*}
Leibniz's rule gives, away from \(z=0\),
\[
 \mathcal T\pi_{N,r}-z\mathcal T\pi_{N-1,r}
 =z^{N-4}\mathscr R_m^{(0)},
\]
\[
 (\mathcal T\pi_{N,r})'-N\mathcal T\pi_{N-1,r}
 =Nz^{N-4}\mathscr R_m^{(1)}.
\]
For \(N=3\), these identities are read at the nonzero value \(z=\xi-r\);
the apparent negative power cancels against the factor \(z\) in the
corresponding \(\mathscr R_m^{(i)}\).  Since
\(\partial_rS_N=-N\mathcal T\pi_{N-1,r}\), a forbidden triple zero forces
\(\mathscr R_m^{(0)}=\mathscr R_m^{(1)}=0\).  Both expressions are linear in
\(\Omega\), and \(\chi\) has disappeared. Write $[0]$ and $[1]$ for
the constant and linear coefficients in $\Omega$. Eliminating this scalar yields
\[
 \mathscr R_m^{(0)}[0]\mathscr R_m^{(1)}[1]
 -\mathscr R_m^{(0)}[1]\mathscr R_m^{(1)}[0]
 =mz^3\mathfrak q_m(x,r),
\]
where
\[
 \mathfrak q_m(x,r)
 =-3(4x+1)r^2
 +\{4(2m+1)x^2-8(m-1)x+3\}r
 -(4m^2-1)x(1-x).
\]
The key identity is
\begin{equation}
\label{eq:diagonal-degree-monotonicity}
 \mathfrak q_m(x,r)
 =\mathfrak q_2(x,r)
 -4x(1-x)(m-2)(m+2+2r).
\end{equation}
Thus \(\mathfrak q_m=0\) implies \(\mathfrak q_2\ge0\).

As a quadratic in \(r\), \(\mathfrak q_2\) has discriminant
\[
 (20x^2-6x-9)(20x^2+26x-1).
\]
It is negative for \(1/16\le x\le4/5\), while the leading coefficient is
negative, so no contact lies in that interval.  Write
\(\eta=r-x\).  Then
\[
 \mathfrak q_2
 =-3(4x+1)\eta^2+(3-14x-4x^2)\eta
  -4x(1-x)(3+2x).
\]
For \(x<1/16\), the two roots in \(\eta\), if real, have positive product
and positive sum, so \(\mathfrak q_2\ge0\) forces \(\eta>0\), hence
\(z=-\eta<0\).  For \(x>4/5\), their product is positive and their sum is
negative, so \(\eta<0\), hence \(z>0\).

A direct coefficient calculation gives
\[
 \mathscr R_m^{(1)}
 =2\Omega z^2\{2z+m(2x-1)\}+m\mathscr M_m,
\]
\[
 \mathscr M_m
 =(24x-9)z^2+(m-1)(24x^2-22x+3)z
 -2(m-1)(m-2)x(1-x)(2x-1).
\]
Put $B_m=(8m-20)x(1-x)+18x-3$ and
\[
 \mathfrak h_m
 =2(m-1)(m+1+2x)(24x^2-22x+3)-(m-2)(2x-1)B_m.
\]
On $\mathfrak q_m=0$, substitution of
$4(m-1)(m+1+2x)x(1-x)=B_mz-3(4x+1)z^2$ gives
\[
 \begin{aligned}
 \frac{\mathscr M_m}{z}
 ={}&z\left[24x-9+
      \frac{3(m-2)(2x-1)(4x+1)}{2(m+1+2x)}\right]\\
    &+\frac{\mathfrak h_m}{2(m+1+2x)}.
 \end{aligned}
\]
On either remaining interval the bracket has the sign of $z$, so its
product with $z$ is positive. To handle the other term, write
$\mu=m-2\ge0$ and
\[
 \mathfrak h_m=h_2\mu^2+h_1\mu+h_0,
\]
where
\[
 \begin{aligned}
 h_2&=16x^3+24x^2-36x+6,\\
 h_1&=88x^3+80x^2-144x+21,\\
 h_0&=2(2x+3)(4x-3)(6x-1).
 \end{aligned}
\]
For $0<x<1/16$, $h_2>15/4$, $h_1>12$, and $h_0>0$.
On $[4/5,1]$, $h_2,h_1$ are increasing, with values $94/125$ and
$257/125$ at $4/5$, and $h_0>0$. Hence $\mathscr M_m/z>0$ in both cases.
Since $z$ and $2x-1$ have the same sign and $\Omega\ge0$,
\[
 \frac{\mathscr R_m^{(1)}}z
 =2\Omega z\{2z+m(2x-1)\}+m\frac{\mathscr M_m}{z}>0,
\]
contradicting $\mathscr R_m^{(1)}=0$.
\end{proof}
\section{Variation diminution on the Jacobi interval}
\label{sec:variation-diminution}

The root problem is now on $[0,4]$.  Variation-diminishing
transformations originate in the work of Schoenberg
\cite{SchoenbergVariation}; for their relation to orthogonal-polynomial
systems, see Hirschman \cite{HirschmanVariation}.  We use the modern matrix
formulations in
\cite{GantmacherKrein,KarlinTP,PinkusTPM,FallatJohnsonTN}, together with
Bernstein coefficient variation on a compact interval
\cite{LorentzBernstein,FaroukiBernstein}. For a basis degree $M\ge0$, the
comparison uses $[\mathcal A]_M^{-1}$, rather than directly $[\mathcal A]_M$.
Bernstein and checkerboard conventions are as in
Section~\ref{sec:zero-orbit-factors}.

\subsection{Bernstein coordinates and the inverse comparison}
For ordered index sets, $A[I\mid J]$ denotes the indicated submatrix;
complements inherit the order of $\{0,\ldots,M\}$. A principal minor is
\emph{contiguous} if its indices are consecutive, and \emph{leading} if
they form an initial set. Bandwidth $w$ means $A_{ij}=0$ for $|i-j|>w$.
Jacobi's complementary-minor identity \cite[Chapter~I]{GantmacherKrein}
gives, for nonsingular TN $A$,
\begin{equation}\label{eq:checkerboard-complementary-minor}
 \det\bigl((\Sigma_M A^{-1}\Sigma_M)[I\mid J]\bigr)
       =\frac{\det A[J^c\mid I^c]}{\det A}\ge0.
\end{equation}
Thus $\Sigma_M A^{-1}\Sigma_M$ is TN.

Write $\operatorname{var}(v)$ for sign changes after zero entries are
deleted, with $\operatorname{var}(0)=0$. The classical TN variation
inequality is $\operatorname{var}(Av)\le\operatorname{var}(v)$
\cite[Chapter~4]{FallatJohnsonTN}.

The needed Bernstein fact is eventual equality: if $P$ has no endpoint
zero and only simple zeros in $(0,4)$, then
\begin{equation}\label{eq:bernstein-eventual-variation}
 \operatorname{var}c_M(P)=\mathrm N_{(0,4)}(P)
                 \qquad(M\hbox{ sufficiently large}).
\end{equation}
Write $(X)_\ell=X(X-1)\cdots(X-\ell+1)$, with $(X)_0=1$, for the
falling factorial. If $P(4t)=\sum_{\ell=0}^D a_\ell t^\ell$, degree elevation gives
\[
 c_{M,k}(P)=P_M(k/M),\qquad
 P_M(t)=\sum_{\ell=0}^D a_\ell\frac{(Mt)_\ell}{(M)_\ell}
                   \longrightarrow P(4t)\quad\hbox{in }C^1[0,1],
\]
Around each simple
zero choose a disjoint interval of strict monotonicity; off these
intervals the limit is bounded away from zero. The same properties hold
for $P_M$ at large $M$, with one zero in each interval. The mesh samples
on either side therefore give exactly one sign change, also if a zero
sample is deleted, proving \eqref{eq:bernstein-eventual-variation}.

\subsection{Closure, generic inputs, and root monotonicity}
\label{subsec:defect-closure}
Each $\mathcal E_{N,q}$ is coefficientwise closed, including degree drops:
if a nonzero limit had more than $q$ roots off $[0,4]$, disjoint disks
about them avoiding that interval and Rouch\'e's theorem would give the
same excessive count for nearby polynomials. The zero limit is included.

Suppose $\mathcal A$ leaves each subspace $\R[y]_{\le d}$ invariant and
is invertible on each; its nonzero degree multipliers preserve exact degree.
On a fixed exact-degree stratum, perturb the interval roots to distinct
points of $(0,4)$, while keeping the remaining roots off the interval in
conjugate pairs. The interval-root count is locally constant. Inside the
resulting open coefficient neighborhood avoid the two proper algebraic loci
\[
 \Disc(P)P(0)P(4)=0,\qquad
 \Disc(\mathcal AP)(\mathcal AP)(0)(\mathcal AP)(4)=0;
\]
in degrees zero and one the discriminant factors are omitted. Degreewise
invertibility makes the second locus proper. Hence every input is a
coefficientwise limit, within its exact degree and with the same
interval-root count, of such $\mathcal A$-generic inputs.

\begin{proposition}[Interval-root monotonicity from total nonnegativity]
\label{prop:interval-root-monotonicity}
Let $\mathcal A:\R[y]\to\R[y]$ be linear and leave every degree space
$\R[y]_{\le M}$ invariant. Suppose that
$\Sigma_M[\mathcal A]_M\Sigma_M$ is nonsingular and totally nonnegative
for every $M\ge0$.  Then every nonzero real
polynomial \(P\) satisfies
\begin{equation}\label{eq:interval-root-monotonicity}
 \mathrm N_{[0,4]}(\mathcal A P)\ge\mathrm N_{[0,4]}(P).
\end{equation}
Consequently, for all \(N,q\ge0\),
\[
 \mathcal A(\mathcal E_{N,q})\subseteq\mathcal E_{N,q}.
\]
\end{proposition}
\begin{proof}
The hypotheses give degreewise invertibility, so the preceding generic
approximation applies. For a generic $P$ put $Q=\mathcal AP$. Applying
\eqref{eq:checkerboard-complementary-minor} to
$\Sigma_M[\mathcal A]_M\Sigma_M$ shows that $[\mathcal A]_M^{-1}$ is TN.
Since $c_M(P)=[\mathcal A]_M^{-1}c_M(Q)$, variation diminution and
\eqref{eq:bernstein-eventual-variation} give, at large $M$,
\[
 \mathrm N_{[0,4]}(P)=\operatorname{var}c_M(P)
 \le\operatorname{var}c_M(Q)=\mathrm N_{[0,4]}(Q).
\]
For arbitrary nonzero $P$ use the generic approximation just proved.
Closure of the defect classes and exact-degree preservation pass the
inequality to the limit. The defect conclusion follows, including zero.
\end{proof}

\subsection{The checkerboard Jacobi matrix}

We compute the translated Jacobi operator in the fixed Bernstein
normalization; for the classical operator and Jacobi--Bernstein conversion,
see \cite[Chapter~IV]{SzegoOrthogonal} and \cite{RababahJacobiBernstein}.
We say that $A$ is \emph{positively diagonally similar} to $S$ if
$S=D^{-1}AD$ for a diagonal matrix $D$ with positive diagonal entries.  When
$S$ is symmetric, we call it a \emph{positive diagonal symmetrization} of
$A$.  A \emph{symmetric Jacobi matrix} is a real symmetric tridiagonal
matrix with positive off-diagonal entries. We retain $[\mathcal A]_N$,
$\Sigma_N$ and the scaling $y=4t$ from
Section~\ref{sec:zero-orbit-factors}; the shifted second-order matrix is
tridiagonal and its quadratic products are pentadiagonal.
For total positivity and zero counts, see
\cite{Coppel,GantmacherKrein,KarlinTP,Karlin1971}.

Under $y=4t$,
\[
 \JacOp_{\nu}
 =-t(1-t)\frac{d^2}{dt^2}+(2t-\nu)\frac{d}{dt}.
\]

\begin{lemma}[Bernstein matrix of a shifted Jacobi factor]
\label{lem:bernstein-jacobi-matrix}
Let \(N\ge0\), \(0<\nu<2\), and \(\alpha\in\R\).  Put
\[
 \mathbf J_{N,\nu}(\alpha)
 =\Sigma_N[\JacOp_{\nu}+\alpha]_N\Sigma_N.
\]
Then
\(\mathbf J_{N,\nu}(\alpha)
 =\mathbf J_{N,\nu}(0)+\alpha I_{N+1}\)
is tridiagonal, where $I_{N+1}$ is the identity matrix.  For \(0\le k\le N\), write
\[
 d_{N,k}^{(\nu)}
 =2k(N-k)+\nu(N-k)+(2-\nu)k.
\]
Then, for \(0\le k\le N\) on the diagonal and
\(0\le k\le N-1\) on the two off-diagonals, its entries are
\[
 (\mathbf J_{N,\nu}(\alpha))_{k,k}
 =\alpha+d_{N,k}^{(\nu)},
\]
\[
 (\mathbf J_{N,\nu}(\alpha))_{k,k+1}
 =(N-k)(k+\nu),
 \qquad
 (\mathbf J_{N,\nu}(\alpha))_{k+1,k}
 =(k+1)(N-k+1-\nu).
\]
The matrix \(\mathbf J_{N,\nu}(0)\) is positively diagonally similar to
a symmetric Jacobi matrix and has spectrum
\[
 0,2,6,\ldots,N(N+1).
\]
Moreover, for \(0\le k\le N-1\),
\begin{equation}\label{eq:jacobi-neighboring-diagonal-bound}
 d_{N,k}^{(\nu)}+d_{N,k+1}^{(\nu)}\ge2.
\end{equation}
\end{lemma}

\begin{proof}
With the convention that $\mathcal B_j^N=0$ for
$j\notin\{0,1,\ldots,N\}$, a direct calculation gives
\[
\begin{aligned}
 (\JacOp_\nu+\alpha)\mathcal B_k^N
 ={}&-(N-k+1)(k+\nu-1)\mathcal B_{k-1}^N\\
 &+\bigl(\alpha+d_{N,k}^{(\nu)}\bigr)\mathcal B_k^N\\
 &-(k+1)(N-k+1-\nu)\mathcal B_{k+1}^N.
\end{aligned}
\]
Conjugation by $\Sigma_N$ changes the two off-diagonal signs, which gives the
displayed entries of $\mathbf J_{N,\nu}(\alpha)$.  The monomial action
\[
 \JacOp_{\nu}y^r
 =r(r+1)y^r-4r(r-1+\nu)y^{r-1}
\]
is triangular, so the spectrum on \(\R[y]_{\le N}\) is
\(r(r+1)\), \(0\le r\le N\).  The neighboring off-diagonal entries have positive product for
\(0<\nu<2\), which gives a positive diagonal symmetrization.

For $N\ge1$ and $0\le k\le N-1$,
\[
 \begin{aligned}
 d_{N,k}^{(\nu)}+d_{N,k+1}^{(\nu)}
 ={}&2+2(3-\nu)k+2(1+\nu)(N-1-k)\\
    &+4k(N-1-k)\ge2.
 \end{aligned}
\]
\end{proof}

For later use, denote the positive diagonal symmetrization of
$\mathbf J_{N,\nu}(0)$ by $\mathbf J_{N,\nu}^{\mathrm{sym}}$.
It is a symmetric Jacobi matrix by Lemma~\ref{lem:bernstein-jacobi-matrix}.

Proposition~\ref{prop:interval-root-monotonicity} applies to the zero-orbit
configurations whose finite Jacobi matrices are totally nonnegative.  A
possible unpaired outer real pair instead requires direct analysis on the
endpoint pencil.

\section{Total nonnegativity of orbit factors}
\label{app:orbit-tn-proofs}
\subsection{Fourth-order orbit factors}

For consecutive sets $I=\{i,\ldots,i+\ell-1\}$ and
$J=\{i+s,\ldots,i+s+\ell-1\}$, call $\det A[I\mid J]$ an
\emph{$s$-shifted solid minor}. We prove the orbit-factor criteria from
Section~\ref{sec:zero-orbit-factors}.

\begin{lemma}[Shifted cofactor identity]
\label{lem:general-shifted-cofactor}
Let \(n\ge1\), and let \(A\) be an \(n\times n\) symmetric Jacobi matrix with positive
off-diagonal entries \(\kappa_0,\ldots,\kappa_{n-2}\), and let its
eigenvalues be \(\lambda_0\le\cdots\le\lambda_{n-1}\).  For
\(\zeta,\xi\in\C\), put
\[
 D_A(w)=\det(A+wI_n)=\prod_{j=0}^{n-1}(\lambda_j+w).
\]
In the formula below, the quotient
$(D_A(\xi)-D_A(\zeta))/(\xi-\zeta)$ is the first divided difference of
$D_A$; when $\zeta=\xi$, it is interpreted as the diagonal value
$D_A'(\zeta)$.  An empty product and a $0\times0$ determinant both equal $1$.
\begin{align}
 &\det\!\left(
   \bigl((A+\zeta I_n)(A+\xi I_n)\bigr)
   [0,\ldots,n-2\mid1,\ldots,n-1]
  \right) \notag\\
 &\hspace{35mm}=
 \left(\prod_{j=0}^{n-2}\kappa_j\right)
 \frac{D_A(\xi)-D_A(\zeta)}{\xi-\zeta}.
 \label{eq:general-shifted-cofactor}
\end{align}
\end{lemma}

\begin{proof}
Assume first that the two factors are invertible and \(\zeta\ne\xi\).  The
resolvent identity gives
\[
 ((A+\zeta I_n)(A+\xi I_n))^{-1}
 =\frac{(A+\zeta I_n)^{-1}-(A+\xi I_n)^{-1}}{\xi-\zeta}.
\]
For a Jacobi matrix,
\[
 ((A+wI_n)^{-1})_{0,n-1}
 =(-1)^{n-1}\frac{\kappa_0\cdots\kappa_{n-2}}{D_A(w)}.
\]
Combining these formulas with the cofactor formula for the
\((0,n-1)\)-entry of the inverse gives
\eqref{eq:general-shifted-cofactor}.  After the divided difference is assigned
its diagonal value, both sides are polynomial in \(\zeta,\xi\).  The identity
therefore extends to singular factors and to \(\zeta=\xi\) by continuity.
For \(n=1\), it reduces to the identity \(1=1\) under the conventions in
the statement.
\end{proof}

\subsubsection{Nonreal quartets}

For the quartet factor $\mathscr A_{a,b}^{(\nu)}$ and matrix
$\mathbf K_{N,\nu}(a,b)$ from Section~\ref{sec:zero-orbit-factors}, write
\[
 \mathfrak s_{a,b}=\frac14-a^2+b^2,\qquad
 \mathfrak c_{a,b}=2ab
 \quad(a\in\R,\ b>0,\ 0<\nu<2).
\]
Their checkerboard relation is
\begin{equation}\label{eq:sharp-quartet-bernstein}
 \Sigma_N[\mathscr A_{a,b}^{(\nu)}]_N\Sigma_N
 =\mathbf K_{N,\nu}(a,b).
\end{equation}
Squaring the tridiagonal Jacobi matrix raises its bandwidth to two.
The spectrum $j(j+1)$ is independent of $\nu$. It suffices to consider
$a>0$, since the matrix depends only on $a^2$; at $a=0$ the factor is a
square of an imaginary-pair factor and follows by closure. The shifted
cofactor formula and the following finite-product bound control the spectral phase,
which transfers the minor signs to consecutive principal blocks.

\begin{lemma}[Finite spectral-product phase]
\label{lem:cosine-partial-product}
Let \(b>0\) and \(a>0\).  For \(n\ge1\), put
\[
 \operatorname{ang}_n(a,b)=
 \sum_{j=0}^{n-1}\arg
 \left(\left(j+\frac12\right)^2-(a-\mathrm{i}b)^2\right),
\]
where each argument is the unique value in \((0,\pi)\); the sum is not reduced
modulo \(2\pi\).  Then
\[
 0<\operatorname{ang}_n(a,b)<\pi
 \qquad(0<a\le1).
\]
If \(a>1\), then \(\operatorname{ang}_n(a,b)\in(\pi,2\pi)\) for at least one \(n\).
Consequently,
\begin{equation}\label{eq:cosine-product-positive}
 \frac{1}{2ab}\Im\prod_{j=0}^{n-1}
 \left(\left(j+\frac12\right)^2-(a-\mathrm{i}b)^2\right)>0
 \qquad(0<a\le1),
\end{equation}
whereas for \(a>1\) the left-hand side is negative for at least one \(n\).
\end{lemma}

\begin{proof}
Each factor has imaginary part $2ab>0$. For $s>0$, direct differentiation gives
\[
 \partial_a\arg\bigl(s^2-(a-\mathrm{i}b)^2\bigr)
 =\frac{2b(s^2+a^2+b^2)}{(s^2-a^2+b^2)^2+4a^2b^2}>0.
\]
Thus every summand, and hence every finite sum, is strictly increasing in $a$.
At $a=1$, factor
\[
 (j+\tfrac12)^2-(1-\mathrm{i}b)^2
 =(j-\tfrac12+\mathrm{i}b)(j+\tfrac32-\mathrm{i}b).
\]
The arguments telescope; including the $\pi$ correction for
$-1/2+\mathrm{i}b$ at $j=0$ gives, for every $n\ge1$,
\[
 \operatorname{ang}_n(1,b)
 =\pi-\arctan\frac{b}{n-1/2}-\arctan\frac{b}{n+1/2}.
\]
Consequently $0<\operatorname{ang}_n(a,b)\le\operatorname{ang}_n(1,b)<\pi$
for $0<a\le1$, and $\operatorname{ang}_n(1,b)\to\pi$.
For fixed $a>1$, the first summand exceeds its value at $a=1$ by a
fixed positive amount, and all other summands also increase. Hence some
partial sum exceeds $\pi$. The first such sum is less than $2\pi$, since
each increment is less than $\pi$. This proves both sign assertions.
\end{proof}

\begin{lemma}[Shifted minors for a reflected quartet]
\label{lem:sharp-quartet-shifted-minors}
Let \(0<\nu<2\), \(0<a\le1\), and \(b>0\).  For every \(N\ge1\),
\(1\le\ell\le N\), and \(0\le i\le N-\ell\),
\[
 \det\mathbf K_{N,\nu}(a,b)
 [i,\ldots,i+\ell-1\mid i+1,\ldots,i+\ell]>0.
\]
The analogous \(-1\)-shifted solid minors are also strictly positive.
If \(a>1\), then for every fixed \(b>0\) and every
\(0<\nu<2\), one of these shifted solid minors is negative for some
matrix size \(N\).
\end{lemma}

\begin{proof}
The two shifts are equivalent by transposition after the positive diagonal
symmetrization of \(\mathbf J_{N,\nu}(0)\), so consider the
\(+1\)-shift.  Put
\[
 B=\{i,i+1,\ldots,i+\ell\},\qquad n=\ell+1,
\]
and let
\[
 I=B\setminus\{i+\ell\},\qquad J=B\setminus\{i\}.
\]
An intermediate index outside \(B\) cannot be adjacent, in the tridiagonal
support, to one index in each of \(I\) and \(J\).  Hence the shifted
submatrix receives no two-step contribution through an index outside \(B\).
Writing $I_B$ for the identity on the coordinates indexed by $B$, we obtain
\[
 \mathbf K_{N,\nu}(a,b)[I\mid J]
 =\left(
   \bigl(\mathbf J_{N,\nu}(0)[B\mid B]+\mathfrak s_{a,b}I_B\bigr)^2
   +\mathfrak c_{a,b}^{\,2}I_B
  \right)[I\mid J],
\]
After passing to positive diagonal symmetrization,
\(\mathbf J_{N,\nu}(0)[B\mid B]\) becomes a symmetric Jacobi matrix
\(A_B\) with positive off-diagonal entries.  Write its eigenvalues as
\[
 \lambda_0\le\cdots\le\lambda_{n-1}.
\]
Cauchy interlacing---the eigenvalues of a principal submatrix interlace
those of the full symmetric matrix---applied to the full symmetric form of
\(\mathbf J_{N,\nu}(0)\), whose spectrum is \(j(j+1)\), gives
\begin{equation}\label{eq:sharp-lambda-lower-bound}
 \lambda_j\ge j(j+1),\qquad 0\le j\le n-1.
\end{equation}
Lemma~\ref{lem:general-shifted-cofactor}, with
\(\zeta=\mathfrak s_{a,b}-\mathrm{i}\mathfrak c_{a,b}\) and
\(\xi=\mathfrak s_{a,b}+\mathrm{i}\mathfrak c_{a,b}\), therefore reduces its sign, up
to a positive symmetrizing factor, to
\[
 \mathcal I_n(\lambda;a,b)=
 \frac{1}{2ab}\Im\prod_{j=0}^{n-1}
 \left(\lambda_j+\frac14-a^2+b^2+2\mathrm{i}ab\right).
\]
For $c=1/4-a^2+b^2$ and $t=2ab>0$, the principal argument
$\phi(\lambda)=\arg(\lambda+c+\mathrm{i}t)\in(0,\pi)$ satisfies
\[
 \phi'(\lambda)=-\frac{t}{(\lambda+c)^2+t^2}<0.
\]
Hence \eqref{eq:sharp-lambda-lower-bound} and
Lemma~\ref{lem:cosine-partial-product} give
\[
 0<\sum_{j=0}^{n-1}\phi(\lambda_j)
 \le\sum_{j=0}^{n-1}\phi(j(j+1))<\pi.
\]
The product defining $\mathcal I_n$ has positive imaginary part, as required.

For sharpness, choose \(n\) as in
Lemma~\ref{lem:cosine-partial-product} and set \(N=n-1\).  For the full block
\(B=\{0,\ldots,N\}\), the eigenvalues are exactly \(j(j+1)\), independently
of \(\nu\).  The shifted cofactor identity then gives a negative shifted
solid minor.
\end{proof}

\subsubsection{A pentadiagonal TN criterion}

General determinantal criteria for totally nonnegative pentadiagonal matrices
are available; see, for example, \cite{AdmGarloffPentadiagonal}.  The criterion
below is tailored to the positively symmetrizable Jacobi blocks arising here
and uses only positive definiteness, nonnegative second subdiagonal entries
and initial shifted minors.

We use banded $LDL^{\mathsf T}$ elimination to exhibit nonnegative
bidiagonal factors directly. Such factorizations are classical in total
positivity \cite{Whitney1952,FominZelevinskyTP,GascaPenaQR}; the argument below needs
only the initial $-1$-shifted minors, not a propagation through all minors.

\begin{lemma}[Positive pentadiagonal criterion]
\label{lem:positive-pentadiagonal-TN}
Let $N\ge0$, and let $A$ be a real $(N+1)\times(N+1)$ pentadiagonal
matrix positively diagonally similar to a symmetric positive-definite matrix.
Assume that
\[
 A_{i,i-2}\ge0\quad(2\le i\le N),\qquad
 \det A[1,\ldots,k\mid0,\ldots,k-1]>0\quad(1\le k\le N).
\]
Then $A$ is totally nonnegative.
\end{lemma}

\begin{proof}
Positive diagonal similarity preserves every minor sign, so replace $A$ by
its symmetric positive-definite form $S$. The case $N=0$ is immediate.
Unpivoted elimination gives
\[
 S=LDL^{\mathsf T},\qquad D=\diag(d_0,\ldots,d_N),\quad d_i>0,
\]
where $L$ is unit lower triangular with only two subdiagonals. This bandwidth
follows inductively from the elimination recurrence: no entry below the
second subdiagonal is created. Write
\[
 a_i=L_{i,i-1}\ (i\ge1),\qquad
 b_i=L_{i,i-2}=\frac{S_{i,i-2}}{d_{i-2}}\ge0\ (i\ge2).
\]
Set $q_0=1$ and
$q_k=\det L[1,\ldots,k\mid0,\ldots,k-1]$ for $1\le k\le N$.
In Cauchy--Binet, the rows $0,\ldots,k-1$ of $L$ force the intermediate
index set to be $\{0,\ldots,k-1\}$. Hence
\[
 \det S[1,\ldots,k\mid0,\ldots,k-1]
 =q_k\prod_{j=0}^{k-1}d_j>0.
\]
Consequently $q_k>0$. Its defining matrix is tridiagonal with superdiagonal
entries $1$, so expansion of its last row gives
\[
 q_1=a_1,\qquad q_k=a_kq_{k-1}-b_kq_{k-2}\quad(k\ge2).
\]
Put $v_i=q_i/q_{i-1}>0$, $u_1=0$ and $u_i=b_i/v_{i-1}\ge0$ for $i\ge2$.
Then $a_i=u_i+v_i$ and $b_i=u_iv_{i-1}$. If $B(u)$ and $B(v)$ are the
unit lower-bidiagonal matrices with these subdiagonals, direct multiplication
therefore gives
\[
 L=B(u)B(v),\qquad
 S=B(u)B(v)D B(v)^{\mathsf T}B(u)^{\mathsf T}.
\]
Nonnegative bidiagonal and diagonal matrices are TN, and Cauchy--Binet
preserves TN under multiplication. The factorization proves the claim.
\end{proof}

\begin{proof}[Proof of Theorem~\ref{thm:sharp-quartet-TN}]
For \(N=0\), the matrix is the positive scalar
\[
 \mathbf K_{0,\nu}(a,b)
 =\left(\frac14-a^2+b^2\right)^2+4a^2b^2,
\]
so the assertion is immediate.  Hence assume \(N\ge1\).  The matrix depends
only on $a^2$. It suffices to prove the claim for $0<a\le1$;
then $a\downarrow0$ gives the boundary by continuity of every minor.  A positive diagonal symmetrization is
\[
 \left(\mathbf J_{N,\nu}^{\mathrm{sym}}+
       \left(\frac14-a^2+b^2\right)I_{N+1}\right)^2+4a^2b^2I_{N+1},
\]
which is positive definite and has positive second off-diagonal entries.
Lemma~\ref{lem:sharp-quartet-shifted-minors} supplies the initial shifted
minors required by Lemma~\ref{lem:positive-pentadiagonal-TN}, so the matrix
is TN. The converse is the final assertion of
Lemma~\ref{lem:sharp-quartet-shifted-minors}.
\end{proof}

\subsubsection{A square pairing rule for outer real factors}

We prove the pairing rule for $\mathbf R_{N,\nu}(\alpha,\beta)$ from
Section~\ref{sec:zero-orbit-factors}. Its range $[-3/4,0]^2$ corresponds
to real pairs with $1/2\le a\le1$, the central-strip boundary $a=1/2$
being included by closure. Multiplicativity of the matrix representation gives
\[
 \mathbf R_{N,\nu}(\alpha,\beta)
 =\Sigma_N[(\JacOp_{\nu}+\alpha)(\JacOp_{\nu}+\beta)]_N\Sigma_N.
\]

\begin{proof}[Proof of Lemma~\ref{lem:square-outer-real-pairing}]
First assume $\alpha,\beta\in(-3/4,0)$. A shifted minor is supported on a
consecutive Jacobi block of size $\ell=k+1$ for a minor of order $k$:
no two-step tridiagonal path contributing to this minor leaves the union
of its row and column sets. As in the quartet proof, that symmetric
block, denoted by $A$, has eigenvalues $\lambda_j\ge j(j+1)$.
Lemma~\ref{lem:general-shifted-cofactor} reduces the minor's sign to the
first divided difference of $D_A$ at $\alpha,\beta$, including its
diagonal value.

Here positivity follows directly, without an eigenvalue induction.
For $0\le t\le3/4$ and $\ell\ge2$,
\[
 D_A'(-t)=\prod_{j=1}^{\ell-1}(\lambda_j-t)
 \left(1+(\lambda_0-t)\sum_{j=1}^{\ell-1}\frac1{\lambda_j-t}\right).
\]
All displayed factors except possibly $\lambda_0-t$ are positive. If
$\lambda_0\ge t$, the bracket is at least $1$; otherwise it is bounded below by
\[
 1-t\sum_{j=1}^{\ell-1}\frac1{j(j+1)-t}
 \ge\frac38\left(\frac1{\ell-1/2}+\frac1{\ell+1/2}\right)>0.
\]
Indeed each $t/(j(j+1)-t)$ increases on this interval, and at $t=3/4$
the sum telescopes to $1$ minus the displayed positive remainder.
For $\ell=1$, $D_A'=1$. The mean value theorem gives a positive divided
difference for $\alpha\ne\beta$, and $D_A'(\alpha)>0$ handles the
diagonal. The opposite shift follows by symmetry.

The positively symmetrized product has eigenvalues
\[
 (j(j+1)+\alpha)(j(j+1)+\beta)>0,
 \qquad 0\le j\le N.
\]
Its second off-diagonals are positive, and the preceding shifted-minor
calculation verifies the remaining hypothesis of
Lemma~\ref{lem:positive-pentadiagonal-TN}. Hence the product is TN.
The boundary of the parameter square $[-3/4,0]^2$ follows by closure.  For
interior parameters,
Proposition~\ref{prop:interval-root-monotonicity} gives preservation of every
$\mathcal E_{N,q}$; approximation from the interior and closedness give the
boundary cases.
\end{proof}

\section{Arithmetic finite-difference nondegeneracy}
\label{app:algebraic-certificates}
\label{app:dedekind-nondegeneracy}
\label{app:newform-nondegeneracy}

We prove both arithmetic nondegeneracy inputs by one finite-sum argument.

\begin{proof}[Proof of Lemmas~\ref{lem:dedekind-strict-nondeg} and
\ref{lem:newform-arbitrary-spacing-nondegeneracy}]
We use two elementary observations. Write $H=e^L>0$. If
\[
 L'(R)=c\log R+O(1),\qquad
 L^{(j)}(R)=O(R^{1-j})\quad(2\le j\le m),
\]
then Fa\`a di Bruno's formula \cite{ComtetAdvanced} gives
\[
 H^{(m)}(R)=H(R)(L'(R))^m(1+o(1));
\]
every other term contains some $L^{(j)}$, $j\ge2$, and is of lower order.
The assertion is exact for $m=0$ and in particular gives eventual
nonvanishing.

If, for fixed shifts $a<b$,
\[
 \frac{H(R+a)}{H(R+b)}=O\!\left(R^{-c(b-a)}\right),
\]
then division by the rightmost translate gives
\[
 \frac{\mathscr T_{\zeta,\delta}^{\,\ell}H(R)}
 {\zeta^\ell H(R+\delta\ell/2)}
 =1+\sum_{j=0}^{\ell-1}\binom\ell j\zeta^{j-\ell}
 \frac{H(R+\delta(j-\ell/2))}{H(R+\delta\ell/2)}
 =1+O(R^{-c\delta}),
\]
uniformly for $|\zeta|=1$ and fixed $\delta>0$; for $\ell=0$ the ratio is
exactly $1$.

For Dedekind zeta, put $c=[K:\mathbb Q]/2$. The Euler product and
Stirling's formula \cite[\S\S5.11, 5.15]{NISTHandbook} give
\[
 \begin{gathered}
 (\log\Xi_K)'(R)=c\log R+O(1),\qquad
 (\log\Xi_K)^{(j)}(R)=O(R^{1-j})\quad(j\ge2),\\
 \Xi_K(R+a)/\Xi_K(R+b)=O(R^{-c(b-a)})\quad(a<b).
 \end{gathered}
\]
Here $\Xi_K(R)>0$ for $R>1$, and the logarithmic derivatives of
$\zeta_K$ decay exponentially on the positive ray. The first observation
transfers the ratio bound to every fixed derivative and hence to
$F_K^{(m)}$. The finite-sum observation with $\delta=1$ proves
\eqref{eq:dedekind-nondeg}.

For newforms, put $A=\sqrt N/(2\pi)$ and $\mathcal G(s)=A^s\Gamma(s)$.
The absolutely convergent Dirichlet series tends to $1$ with exponentially
decaying derivatives. Thus $\log\Lambda(f,R)$ has first derivative
$\log(AR)+o(1)$ and higher derivatives $O(R^{1-j})$. The same differentiation
argument and gamma-ratio asymptotics give, uniformly on any fixed finite
set of shifts,
\begin{equation}\label{eq:newform-arbitrary-spacing-individual-growth}
 F_f^{(m)}(R+a)=\mathcal G(R+a+k/2)
               (\log(AR))^m(1+o(1)),
\end{equation}
\begin{equation}\label{eq:newform-arbitrary-spacing-shift-ratio}
 \frac{F_f^{(m)}(R+a)}{F_f^{(m)}(R+b)}=(AR)^{a-b}(1+o(1)).
\end{equation}
The denominators are eventually nonzero. The finite-sum observation with
$c=1$ proves \eqref{eq:newform-arbitrary-spacing-dominance}, including
arbitrary fixed spacing, the error $O(R^{-\delta})$, and uniformity in
$\zeta$. This is the same rightmost-translate mechanism as
\cite[Lemma~5.4]{JinSharpSampling}, applied at arbitrary spacing.
\end{proof}

\section*{Use of generative-AI tools}
The author used OpenAI ChatGPT as an aid in discussing and refining
mathematical arguments, examining possible gaps, preparing code for symbolic
checks, improving exposition, and editing \LaTeX{} and English prose. The tool
was not relied upon as an authority for mathematical correctness or as a
formal proof verifier. All proofs, computations, statements, and references
were independently checked by the author, who takes full responsibility for
the manuscript.

\end{document}